\documentclass[a4paper,11pt,reqno]{amsart}

\usepackage{a4wide,color}
\usepackage{amsmath}
\usepackage{amsfonts}
\usepackage{amsthm}
\usepackage{amssymb}

\usepackage{graphicx}
\usepackage{color}
\usepackage{subfigure}

\usepackage{setspace}
\usepackage[active]{srcltx}
\usepackage{mdwlist}

\usepackage{hyperref}
\usepackage{enumerate}
\usepackage{mathrsfs}
\usepackage{mdwlist}

\usepackage{textcomp}
\usepackage{wasysym}

\newcommand{\R}{\mathbb{R}}

\newcommand{\Z}{\mathbb{Z}}
\renewcommand{\SS}{\mathbb{S}}

\newcommand{\eps}{\varepsilon}

\theoremstyle{plain}
\newtheorem{theorem}{Theorem}
\newtheorem{lemma}[theorem]{Lemma}
\newtheorem{definition}[theorem]{Definition}

\newtheorem{corollary}[theorem]{Corollary}
\newtheorem{prop}[theorem]{Proposition}
\theoremstyle{remark}
\newtheorem{remark}{Remark}

\newtheorem{example}{Example}

\author[G.~Barles]{Guy Barles$^*$}
\address{$^*$Laboratoire de Math\'ematiques et Physique Th\'eorique,
  CNRS UMR 6083, F\'ed\'eration Denis Poisson,
  Universit\'e Fran\c{c}ois Rabelais, Parc de Grandmont,
  37200 Tours, France}
\email{barles@lmpt.univ-tours.fr}
 
\author[E. Chasseigne]{Emmanuel Chasseigne$^\dag$}
\address{$^\dag$Laboratoire de Math\'ematiques et Physique Th\'eorique,
  CNRS UMR 6083, F\'ed\'eration Denis Poisson,
  Universit\'e Fran\c{c}ois Rabelais, Parc de Grandmont,
  37200 Tours, France}
\email{emmanuel.chasseigne@lmpt.univ-tours.fr}

\author[A. Ciomaga]{Adina Ciomaga$^\ddag$}
\address{$^\ddag$Centre de Math\'ematiques et de Leurs
  Applications, Ecole Normale Sup\'erieure de Cachan, CNRS, UniverSud,
  61 avenue du pr\'esident Wilson, F-94230 Cachan, France}
\email{ciomaga@cmla.ens-cachan.fr}

\author[C. Imbert]{Cyril Imbert$^\S$}
\address{$^\S$CEREMADE, Universit\'e Paris-Dauphine, UMR CNRS
    7534, place de Lattre de Tassigny, 75775 Paris cedex 16, France \& DMA,
    Ecole Normale Sup\'erieure, 45 rue d'Ulm, F 75230 Paris cedex 5, France}
\email{imbert@ceremade.dauphine.fr}

\title[Lipschitz Regularity for Mixed PIDES]{Lipschitz Regularity of Solutions for Mixed Integro-Differential Equations}

\begin{document}
\graphicspath{./}

\begin{abstract} 
  We establish \emph{new H\"older and Lipschitz} estimates for
  viscosity solutions of a large class of elliptic and parabolic
  nonlinear integro-differential equations, by the classical
  Ishii-Lions's method. We thus extend the H\"older regularity results
  recently obtained by Barles, Chasseigne and Imbert (2011).  In
  addition, we deal with a new class of nonlocal equations that we
  term \emph{mixed integro-differential equations}. These equations
  are particularly interesting, as they are degenerate both in the
  local and nonlocal term, but their overall behavior is driven by the
  local-nonlocal interaction, e.g. the fractional diffusion may give
  the ellipticity in one direction and the classical diffusion in the
  complementary one.
\end{abstract}

\maketitle

\bigskip

\bigskip

\section{Introduction}
\bigskip

Recently regularity results for integro-differential equations have
been investigated by many authors: we provide below some references
but the list is by no means complete. In particular, H\"older
estimates for viscosity solutions of a large class of elliptic and
parabolic nonlinear integro-differential equations are obtained in
\cite{BCI:11:HdrNL}, by the classical Ishii-Lions's method.

The aim of this article is twofold: on one hand, we extend these
results to provide Lipschitz estimates in a similar framework and, on
the other hand, we deal with a new class of nonlocal equations that we
call \emph{mixed integro-differential equations} for which we also
give complementary H\"older estimates. The simplest example of such
mixed integro-differential equations is given by
\begin{equation}\label{simpl-mide} 
-\Delta_{x_1}u + (-\Delta_{x_2})^{\beta/2}u= f(x_1,x_2)
\end{equation}
where $x_1\in \R^{d_1}$, $x_2 \in \R^{d_2}$, and $
(-\Delta_{x_2})^{\beta/2} u$ denotes the fractional Laplacian with
respect to the $x_2$-variables
\begin{equation*}
(-\Delta_{x_2})^{\beta/2}u= -\int_{\R^{d_2}}\big(u(x_1,x_2+z_2)-u(x_1,x_2)-
D_{x_2}u(x_1,x_2)\cdot z_2 1_{B^{d_2}}(z_2)\big)\frac{dz_2}{|z_2|^{d_2+\beta}}
\end{equation*}
where $B^{d_2}$ is the unit ball in $\R^{d_2}$. In this case local
diffusions occur only in the $x_1$-directions and fractional
diffusions in the $x_2$-directions.

To be more specific about our approach, we first recall that Ishii and
Lions introduced in \cite{IL:90:ViscPDE} a simple method to prove
$C^{0,\alpha}$ ($0<\alpha \leq 1$) regularity of viscosity solutions
of fully nonlinear, possibly degenerate, elliptic partial differential
equations, which has the double advantage of providing \emph{explicit
  $C^{0,\alpha}$ estimates} combined with a \emph{light localization
  procedure}.

This simple method, closely related to classical viscosity solutions
theory, was recently explored by the first, second and fourth authors in
\cite{BCI:11:HdrNL} for \emph{second order, fully nonlinear elliptic
  partial integro-differential equations}, dealing with a large class
of integro-differential operators, whose singular measures depend on
$x$.  They prove that the solution is $\alpha$-H\"older
continuous for any $\alpha<\min(\beta, 1)$, where $\beta$
characterizes the singularity of the measure associated with the
integral operator. However, in the case $\beta\geq 1$ the respective
ad-literam estimates do not yield Lipschitz regularity.

In order to treat a large class of nonlinear equations, the authors of
\cite{BCI:11:HdrNL} assume the nonlinearity satisfies a suitable
ellipticity growth assumption. Roughly speaking, this assumption gives
a suitable meaning to a generalized ellipticity of the equation in the
sense that at each point of the domain, the ellipticity comes either
from the second order term (the equation is strictly elliptic in the
classical fully nonlinear sense), or from the nonlocal term (the
equation is strictly elliptic in a nonlocal nonlinear sense).

In a recent study of the strong maximum principle for
integro-differential equation \cite{Ciomaga:SMaxP}, the third author
introduced another type of \emph{mixed ellipticity}: at each point,
the nonlinearity may be \emph{degenerate in the second-order term, and
  in the nonlocal term}, but \emph{ the combination of the local and
  the nonlocal diffusions renders the nonlinearity uniformly
  elliptic}. Equation~(\ref{simpl-mide}) is the typical example of
such \emph{mixed integro-differential equations} since the diffusion
term gives the ellipticity in certain directions, whereas it is given
by the nonlocal term in the complementary directions. For this type of
nondegenerate equations, the assumptions in \cite{BCI:11:HdrNL} are
not satisfied.

\subsection{Main results}

Using Ishii-Lions's viscosity method, we give both \emph{H\"older and
  Lipschitz regularity results} of viscosity solutions for a
\emph{general class of mixed elliptic integro-differential equations}
of the type
\begin{eqnarray}\label{stPIDEs}
\nonumber
F_0(u(x),Du, D^2u, \mathcal I[x,u]) 
&+ & F_1(x_1,D_{x_1}u, D_{x_1x_1}^2u, \mathcal I_{x_1}[x,u])\\
& + &  F_2(x_2,D_{x_2}u, D_{x_2x_2}^2u, \mathcal I_{x_2}[x,u]) = f(x)
\end{eqnarray}
as well as evolution equations
\begin{eqnarray}\label{evPIDEs}\nonumber
u_t+ F_0(u(x),Du, D^2u, \mathcal I[x,u]) 
&+ & F_1(x_1,D_{x_1}u, D_{x_1x_1}^2u, \mathcal I_{x_1}[x,u])\\
& + &  F_2(x_2,D_{x_2}u, D_{x_2x_2}^2u, \mathcal I_{x_2}[x,u]) = f(x).
\end{eqnarray}
A point in $x\in\R^d$ is written as $x= (x_1,x_2)\in \R^{d_1}\times \R^{d_2}$, with $d=d_1+d_2$. 
The symbols $u_t$, $Du$, $D^2u$ stand for the derivative with respect to time, 
respectively the gradient and the Hessian matrix with respect to $x$. 
Subsequently, we write the gradient on components as $Du=(D_{x_1}u,D_ {x_2}u)$ 
and the Hessian matrix $D^2u\in\mathbb{S}^d$ (with $\mathbb{S}^d$ the set of 
real symmetric $d\times d$ matrices) as a block matrix of the form
\begin{equation*}
D^2u = 
\begin{bmatrix} 
D_{x_1x_1}^2u &  D_{x_1x_2}^2u \\
D_{x_2x_1}^2u & D_{x_2x_2}^2u\\
\end{bmatrix}. 
\end{equation*} 
$\mathcal{I}[x,u]$ is an integro-differential operator, taken on the whole space $\R^d$,
associated to L\'evy processes
\begin{equation*}
 \mathcal{I}[x,u]=\int_{\R^d} (u(x+z)-u(x)-Du(x)\cdot z 1_B(z))\mu_x(dz)
\end{equation*}
where $1_B(z)$ denotes the indicator function of the unit ball $B$ and $\big(\mu_x\big)_{x\in\R^d}$ 
is a family of L\'evy measures, i.e. nonnegative, possibly singular, Borel measures on $\R^d$ such that
$$\sup_{x\in\R^d}\int_{\R^d}\min(|z|^2,1)\mu_x(dz)<\infty.$$
Accordingly, one has the directional integro-differential operators
\begin{eqnarray*}
\mathcal{I}_{x_1}[x,u]=\int_{\R^{d_1}} \left(u(x_1+z,x_2)-u(x_1,x_2)-D_{x_1}u(x)\cdot z 1_{B^{d_1}}(z)\right)\mu^1_{x_1}(dz) \\
\mathcal{I}_{x_2}[x,u]=\int_{\R^{d_2}} \left(u(x_1,x_2+z)-u(x_1,x_2)-D_{x_2}u(x)\cdot z 1_{B^{d_2}}(z)\right)\mu^2_{x_2}(dz). 
\end{eqnarray*}
where $\big(\mu^i_{x_i}\big)_{x_i\in\R^{d_i}}$, $i=1,2$ are L\'evy measures and $1_{B^{d_i}}$ 
is the indicator function of the unit ball $B^{d_i}$ in $\R^{d_i}$. 
We consider as well the special class of L\'evy-It\^o operators, defined as follows
\begin{equation*}
 \mathcal{J}[x,u]=\int_{\R^d}\left(u(x+j(x,z))-u(x)-Du(x)\cdot j(x,z) 1_B(z)\right)\mu(dz)
\end{equation*}
where $\mu$ is a L\'evy measure and  $j(x,z)$ is the size of the jumps at $x$ satisfying
$$\sup_{x\in\R^d}\int_{\R^d}\min(|j(x,z)|^2,1)\mu(dz)<\infty.$$
Similarly, we deal with directional L\'evy-It\^o integro-differential operators
\begin{eqnarray*}
&&\mathcal{J}_{x_1}[x,u]=\int_{\R^{d_1}} (u(x_1+j(x_1,z),x_2)-u(x_1,x_2)-D_{x_1}u(x)\cdot j(x_1,z) 1_{B^{d_1}}(z))\mu^1(dz) \\
&&\mathcal{J}_{x_2}[x,u]=\int_{\R^{d_2}} (u(x_1,x_2+j(x_2,z))-u(x_1,x_2)-D_{x_2}u(x)\cdot j(x_2,z) 1_{B^{d_2}}(z))\mu^2(dz). 
\end{eqnarray*}
We assume the nonlinearities are continuous and \emph{degenerate elliptic}, i.e.
\begin{equation*}
F_i(...,X,l)\leq  F_i(...,Y,l') \hbox{ if } X\geq Y, \ l\geq l',
\end{equation*}
for all $X,Y\in\mathbb{S}^{d_i}$ and $l,l'\in\R$, $i=0,1,2$.

In addition, we suppose that the three nonlinearities satisfy suitable {strict ellipticity and growth conditions}, 
that we omit here for the sake of simplicity, but will be made precise in the following section.
 These structural growth conditions can be illustrated on the following example:
\begin{eqnarray*}
-a_1(x_1)\Delta_{x_1}u   -a_2(x_2)\mathcal I_{x_2}[x,u] -  \mathcal I[x,u] 
+b_1(x_1)|D_{x_1}u_1|^{k_1} +b_2(x_2)|D_{x_2}u|^{k_2} + |Du|^n + c u= f(x)
\end{eqnarray*}
where the nonlocal term $\mathcal I_{x_2}[x,u]$ has fractional exponent $\beta\in(0,2)$ and $a_i(x_i)>0$, for $i=1,2$. Thus
\begin{eqnarray*}
F_0(u(x),Du, D^2u, \mathcal I[x,u]) & = &  -  \mathcal I[x,u] + |Du|^n + c u \\
F_1(x_1,D_{x_1}u, D_{x_1x_1}^2u, \mathcal J_{x_1}[x,u])  & = &  -a_1(x_1)\Delta_{x_1}u + b_1(x_1)|D_{x_1}u_1|^{k_1} \\
F_2(x_2,D_{x_2}u, D_{x_2x_2}^2u, \mathcal J_{x_2}[x,u])  & = & -a_2(x_2)\mathcal I_{x_2}[x,u]+b_2(x_2)|D_{x_2}u|^{k_2}.
\end{eqnarray*}
When $\beta>1$, we show that the solution is Lipschitz continuous for
mixed equations with gradient terms $b_i(x_i)|D_{x_i}u|^{k_i}$ having
a natural growth $k_i\leq\beta$ if $b_i$ bounded.  If in addition
$b_i$ are $\tau$-H\"older continuous, then the solution remains
Lipschitz for gradient terms with natural growth $k_i\leq\tau+\beta$.
When $\beta\leq1$, the solution is $\alpha$-H\"older continuous for
any $\alpha<\beta$.  The critical case $\beta = 1$ is left open.

\subsection{Known results}

The classical theory for second order, uniformly elliptic
integro - differential equations includes a priori estimates, weak and
strong maximum principles, etc. In particular, existence and
uniqueness results have been extended from elliptic partial
differential equations to elliptic integro-differ\-ential
equations. For results in the framework of Green functions and
classical solutions we send the reader to the up-to-date book of
Garroni and Menaldi \cite{GarMen:02:eID} and the references therein.

More recently there have been many papers dealing with $C^{0,\alpha}$
estimates and regularity of solutions (not necessarily in the
viscosity setting) for fully nonlinear integro-differential equations
and the literature has been considerably enriched. It is not possible
to give an exhaustive list of references but we next try to give the
flavour of the known results.  

In the framework of potential theory (hence linear equations), Bass
and Levin first establish Harnack inequalities \cite{BL:02:Hnk}. Then
Kassmann \cite{kass-cras,kass-cvpde} adapted the de Giorgi theory to
non-local operators. In the same spirit, Silvestre gave in
\cite{Si:06:RegNL} an analytical proof of H\"older continuity for
harmonic functions with respect to the integral operator.

In the setting of \emph{viscosity solutions}, there are essentially
two approaches for proving H\"older or Lipschitz regularity: either by
the Ishii-Lions's method or by ABP estimates and Krylov - Safonov and
Harnack type inequalities. These methods do not cover the same class
of equations, they have different aims and each of them has its own
advantages.  

The powerful Harnack approach was first introduced by Krylov and
Safonov \cite{ks,ks2} for linear equations under non-divergence form
and then adapted to fully non-linear elliptic equations by Trudinger
\cite{tru} and Caffarelli \cite{caff91}. This theory applies to
\emph{uniformly elliptic, fully nonlinear equations}, with \emph{rough
  coefficients}. The existing theory for second order elliptic
equations has been extended to integro-differential equations by
Caffarelli and Silvestre in \cite{CS:09:REgNL}.  Both for local and
non-local equations, this theory leads to further regularity such as
$C^{1,\alpha}$. But as far as nonlocal equations are concerned, it
requires in particular some integrability condition of the measure at
infinity.  

On the contrary, direct viscosity methods apply under weaker
ellipticity assumptions but require H\"older continuous coefficients
and do not seem to yield further regularity.  Finally these methods
allow measures which are only bounded at infinity.  

Very recently, Cardaliaguet and Rainer showed H\"older regularity of
viscosity solutions for nonlocal Hamilton Jacobi equations with
superquadratic gradient growth \cite{CR:HdrHJ2}, using probablistic
representation formulas.

We would like to conclude this introduction by mentioning that this
work was motivated by the study of long time behaviour of periodic
viscosity solutions for integro-differential equations, that we are
considering in a companion paper. We point out that long time behaviour
comes to the resolution of the stationary ergodic problem, which is
basically the cell problem in homogenization. The periodic
homogenization for nonlinear integro-differential equations has been
adressed by Schwab in \cite{S:10:PerHgzNL}. However, it is restricted
to a certain family of equations, due to a lack of fine ABP
estimate. Recently, Schwab and Guillen provided \cite{GS:11:ABP} and
ABP estimate that would help solve the homogenization for a wider
class of nonlinearities.

The paper is organized as follows. In Section \S
\ref{sec_Lip_assumptions} we give the appropriate definition of
viscosity solution, make precise the ellipticity growth conditions to
be satisfied by the nonlinearities and list the assumptions on the
nonlocal terms.  Section \S \ref{sec_Lip_regularity} is devoted to
the main results, which for the sake of clarity are given in the
periodic setting.  We state partial regularity results, provide the
complete proof, and then present the global regularity result.  In the
next Section \S \ref{sec_Lip_examples} we consider several
significant examples and discuss the main assumptions required by the
regularity results and their implications.  Extensions to the
nonperiodic setting, parabolic versions of the equations,
Bellman-Isaacs equations and multiple nonlinearities are recounted in
Section \S \ref{sec_Lip_extensions}.  At last we detail in Section
\S \ref{sec_Lip_estimates} the technical Lipschitz and H\"older
estimates for the general nonlocal operators and L\'evy-It\^o
operators, which are essentially the backbone of the main results.

\section{Notations and Assumptions}\label{sec_Lip_assumptions}

\subsection{Viscosity Solutions for Integro-Differential Equations}
To overcome the difficulties imposed by behavior at infinity of the measures $(\mu_x)_x$,
as well as the singularity at the origin, we often need to split the nonlocal terms into 
\begin{eqnarray}
\nonumber
 \mathcal{I}^1_\delta [x,u] & = &\int_{|z|\leq \delta}\big(u(x+z)-u(x)-Du(x)\cdot z 1_B(z)\big)\mu_x(dz) \\
\nonumber
 \mathcal{I}^2_\delta [x,p,u]& = &\int_{|z|>\delta}\big(u(x+z)-u(x)-p\cdot z 1_B(z)\big)\mu_x(dz)
\end{eqnarray}
respectively, in the case of L\'evy-It\^o operators,
\begin{eqnarray}
\nonumber
 \mathcal{J}^1_\delta [x,u] & = &\int_{|z|\leq \delta}\big(u(x+j(x,z))-u(x)-Du(x)\cdot j(x,z) 1_B(z)\big)\mu(dz) \\
\nonumber
 \mathcal{J}^2_\delta [x,p,u]& = &\int_{|z|>\delta}\big(u(x+j(x,z))-u(x)-p\cdot j(x,z) 1_B(z)\big)\mu(dz)
\end{eqnarray}
with $0<\delta<1$ and $p\in\R^{ d}$.  \smallskip

One of the very first definitions of viscosity solutions for integro-differential equations was introduced by Sayah in \cite{Sayah:91:HJID2}. In particular, for mixed integro-differential equations, the definition can be stated as follows.

\begin{definition}\label{def_viscsolNL}[Viscosity solutions]
An upper semi-continuous ( in short usc) function $u:\R^d\rightarrow\R$ is a \emph{subsolution} of (\ref{stPIDEs}) 
if for any $\phi\in {C^2}(\R^d)$ such that $u-\phi$ attains a global maximum at $x\in\R^d$  
\begin{eqnarray*}
&&F_0(u(x),D\phi(x), D^2\phi(x), \mathcal{I}^1_\delta[x,t,\phi]+\mathcal{I}^2_\delta[x,t,D\phi(x,t),u]) + \\
&&F_1(x_1,D_{x_1}\phi(x),D_{x_1x_1}^2\phi(x), \mathcal{I}^1_{x_1,\delta}[x,t,\phi]+\mathcal{I}^2_{x_1,\delta}[x,t,D\phi(x,t),u]) + \\
&&F_2(x_2,D_{x_2}\phi(x),D_{x_2x_2}^2\phi(x), \mathcal{I}^1_{x_2,\delta}[x,t,\phi]+\mathcal{I}^2_{x_1,\delta}[x,t,D\phi(x,t),u]) \leq f(x).
\end{eqnarray*}
A lower semi-continuous (in short lsc) function $u:\R^d\rightarrow\R$ is a \emph{subsolution} of (\ref{stPIDEs}) 
if for any $\phi\in {C^2}(\R^d)$ such that $u-\phi$ attains a global minimum at $x\in\R^d$  
\begin{eqnarray*}
&&F_0(u(x),D\phi(x), D^2\phi(x), \mathcal{I}^1_\delta[x,t,\phi]+\mathcal{I}^2_\delta[x,t,D\phi(x,t),u]) + \\
&&F_1(x_1,D_{x_1}\phi(x),D_{x_1x_1}^2\phi(x), \mathcal{I}^1_{x_1,\delta}[x,t,\phi]+\mathcal{I}^2_{x_1,\delta}[x,t,D\phi(x,t),u]) + \\
&&F_2(x_2,D_{x_2}\phi(x),D_{x_2x_2}^2\phi(x), \mathcal{I}^1_{x_2,\delta}[x,t,\phi]+\mathcal{I}^2_{x_1,\delta}[x,t,D\phi(x,t),u]) \geq f(x).
\end{eqnarray*}
\end{definition} 

However, there are several equivalent definitions of viscosity solutions. Thoughout this paper, we use the definition involving sub and super-jets, which was shown in \cite{BI:08:ViscNL} to be equivalent with Definition \ref{def_viscsolNL}. One just has to replace in the viscosity inequalities  the derivatives of the test function $(D\phi, D^2\phi)$ with semi-jets $(p,X)$. To avoid technical details due to partial derivatives with respect to $x_1$ and $x_2$ we omit it here, and just recall the notions of semi-jets.

If $u:\R^d\rightarrow\R$ and $v:\R^d\rightarrow\R$ are respectively a lsc and an usc function, we denote by $\mathcal D^{2,-}u(x)$ the subjet of $u$ at $x\in \R^d$ and by $\mathcal D^{2,+}v(x)$ the superjet of $v$ at $x\in \R^d$. We recall that they are given by
$$
\mathcal D^{2,-}u(x) = \left\{ (p,X)\in\R^d\times\SS^d; \; u(x+z) \geq u(x) + p\cdot z + \frac12 Xz\cdot z + o(|z|^2)\right\}
$$
$$
\mathcal D^{2,+}v(x) = \left\{ (p,X)\in\R^d\times\SS^d; \; u(x+z) \leq u(x) + p\cdot z + \frac12 Xz\cdot z + o(|z|^2)\right\}.
$$

\subsection{Ellipticity Growth Conditions}

We assume that the nonlinearities $F_i$, with $i=0,1,2$, satisfy (one or more of) the next assumptions. In the sequel of this subsection, the notation $F$ stands for any of the nonlinearties $F_i$.
The precise selection for each of the nonlinearities shall be given later on, when the regularity result is stated. 
Further examples and  comments upon the restrictions of these nonlinearities are provided in Section \S \ref{sec_Lip_examples}.
In the sequel of this subsection, the notation $F$ stands for any of the nonlinearties $F_i$.

\begin{itemize}
 \item [$(H0)$] There exists $\tilde \gamma\in\R$ such that for any $u,v\in\R$, $p\in\R^{\tilde d}$, $X\in \SS^{\tilde d}$ and $l\in\R$
 \begin{equation*}
 F(u,p,X,l) - F(v,p,X,l) \geq \tilde \gamma(u-v) \hbox{ when } u\geq v.
 \end{equation*}
 \item [$(H1)$]
There exist two functions ${\Lambda_1},{\Lambda_2}:\R^{\tilde d}\rightarrow [0,\infty)$ such that ${\Lambda_1}(x)+{\Lambda_1}(x)\geq{\Lambda_0}>0$
and some constants $k\geq0$, $\tau\in(0,1]$
$\theta, \tilde\theta \in(0,1]$
such that for any $x,y\in\R^{\tilde d}$, $p\in\R^{\tilde d}$, $l\leq l'$  and any $\varepsilon >0$
\begin{eqnarray*}&& 
F(y,p,Y,l')-F(x,p,X,l) \leq \\ && \hspace{2cm} 
{\Lambda_1}(x)\left((l-l') + \frac{|x-y|^{2\theta}}{\varepsilon}+ |x-y|^\tau|p|^{k+\tau} + {{C_1}|p|^{k}}\right)+ \\ && \hspace{2cm}
{\Lambda_2}(x)\left( \hbox{tr}(X-Y) + \frac{|x-y|^{2\tilde\theta}}{\varepsilon} + |x-y|^\tau|p|^{2+\tau}+{{C_2}|p|^2}\right)
\end{eqnarray*}
if $X,Y\in\SS^{\tilde d}$ satisfy the inequality
\begin{equation}\label{eq_matrix_ineq}
- \frac{1}{\varepsilon}
\begin{bmatrix} 
I &  0 \\
0 & I 
\end{bmatrix} 
\leq
\begin{bmatrix} 
X &  0 \\
0 & -Y 
\end{bmatrix} 
\leq \frac{1}{\varepsilon}
\begin{bmatrix} 
 Z & -Z \\
-Z &  Z 
\end{bmatrix},
\end{equation}
with $ Z = I - \omega\hat a\otimes \hat a$, for some unit vector $\hat a\in\R^{\tilde d}$, and {$\omega\in (1,2)$}.
\item [$(H2)$] 
$F(\cdot, l)$ is Lipschitz continuous, uniformly with respect to all the other variables.
\item [$(H3)$] 
There exists a modulus of continuity $\omega_F$ such that for any $\varepsilon >0$ 
\begin{eqnarray*}
 F(y,\frac{x-y}{\varepsilon},Y,l) - F(x,\frac{x-y}{\varepsilon},X,l) \leq \omega_{F}\left(\frac{|x-y|^2}{\varepsilon} + |x-y|\right)
\end{eqnarray*}
for all $x,y\in \R^{\tilde d}$, $X, Y \in \SS^{\tilde d}$ 
satisfying  the matrix inequality (\ref{eq_matrix_ineq}) with $Z= I$ and $l\in\R$.

\end{itemize}
\subsection{L\'evy Measures for General Nonlocal Operators}

We recall that in this case, the nonlocal term $\mathcal I[x,u] $ is an integro differential operator defined by
\begin{equation}\label{eq_NLterm}
\mathcal I[x,u] = \int_{\R^{\tilde d}} \big(u(x + z) - u(x) - Du(x)\cdot z  1_B(z)\big)\mu_{x}(dz)
\end{equation}
where $1_B$ denotes the indicator function of the unit ball and $\big(\mu_{x}\big)_x$ is a family of L\'evy measures.
We need to make a series of assumptions for the family of L\'evy measures that we make precise now.
\begin{itemize}
\item [($M1$)] 
	There exists a constant $\tilde C_\mu>0$ such that
	$$ 
	\sup_{x\in\R^{\tilde d}}\big(\int_B|z|^2\mu_{x}(dz)+\int_{\R^{\tilde d}\setminus B}\mu_{x}(dz)\big)\leq \tilde C_\mu. 
	$$
\item [($M2$)] 
	There exists $\beta\in(0,2)$ such that for every $a\in\R^{\tilde d}$ there exist $0<\eta<1$ 
and a constant $C_\mu>0$ such that the following holds for any $x\in\R^{\tilde d}$
	$$ 
	 \forall \delta> 0 \;\; \int_{\mathcal C_{\eta, \delta}(a)}|z|^2 \mu_{x}(dz)\geq C_\mu\ \eta^{\frac{\tilde d-1}{2}}\ \delta^{2-\beta}
	$$
	with $\mathcal C_{\eta,\delta}(a):= \{z\in B_\delta; (1-\eta)|z||a|\leq|a\cdot z|\}$.
\item [($M3$)] 
	{There exist $\beta\in(0,2)$, $\gamma\in(0,1)$ and a constant $C_\mu>0$ such that for any $x,y\in\R^{\tilde d}$ and all $\delta>0$
	$$
	\int_{B_\delta}|z|^2|\mu_x-\mu_y|(dz)\leq C_{\mu} |x-y|^\gamma\ \delta^{2-\beta}
	$$
	and
	$$
	\int_{B\setminus B_\delta}|z||\mu_x-\mu_y|(dz)\leq \left\{
	\begin{array}{ll}
	 C_{\mu} |x-y|^\gamma\ \delta^{1-\beta} & \hbox{ if } \beta\neq1 \\
	 C_{\mu} |x-y|^\gamma\ |\ln\delta|    & \hbox{ if } \beta = 1.
	\end{array}\right.
	$$}
\end{itemize}
At the same time, we assume that the directional L\'evy measures satisfy similar assumptions.

\begin{example}{\rm
To make precise the form of $(M2)$ we consider the fractional Laplacian with exponent $\beta$ and compute in $\R^2$
\begin{eqnarray*}
\int_{\mathcal C_{\eta,\delta}(a)} |z|^2\frac{dz}{|z|^{2+\beta}}  & = & 
\frac{\hbox{vol}(\mathcal C_{\eta,\delta}(a))}{\hbox{vol}(B_\delta)} \int_{B_\delta} |z|^2\frac{dz}{|z|^{2+\beta}} = 
\frac{\hbox{vol}(\mathcal C_{\eta,1}(a))}{\hbox{vol}(B_1)} \int_{B_\delta} |z|^2\frac{dz}{|z|^{2+\beta}}
\\  &=& 
\delta^{2-\beta}\frac{\hbox{vol}(\mathcal C_{\eta,1}(a))}{\hbox{vol}(B_1)}  \int_{B_1} |z|^2\frac{dz}{|z|^{2+\beta}} =
\delta^{2-\beta}\frac{\theta}{\pi}  \int_{B_1} |z|^2\frac{dz}{|z|^{2+\beta}},
\end{eqnarray*}
where $\theta$ denotes the angle measuring the aperture of the cone. 
Taking into account the definition of $\mathcal C_{\eta,1}(a)$ we have for small angles $\theta$
$$\eta = 1 - \cos(\theta) =  \frac{\theta^2}{2} + o(\theta^2)$$
and hence $\theta \simeq \sqrt{\eta}$, from where we deduce $(M2)$. 

In higher dimension $d\geq 3$, the volume of the cone is given in spherical coordinates, 
with normal direction $a = (0, 0, ... , 1)$, polar angle $\phi_1 \in[0,\pi]$, and angular coordinates $\phi_2,...,\phi_{d-2}\in[0,\pi]$, $\phi_{d-1}\in[0,2\pi]$,  by the formula 
\begin{eqnarray*}
\hbox{vol}(\mathcal C_{\eta,1}(a))=
\int_0^\theta \sin^{d-2}(\phi_1) d\phi_{1} 
... 
\int_0^\pi    \sin(\phi_{d-2})   d\phi_{d-2}  
\int_0^{2\pi} 			 d\phi_{d-1}
\int_0^1 r^{d-1} dr.
\end{eqnarray*}
For small angles $\theta$ the volume can be approximated by
\begin{eqnarray*}
\hbox{vol}(\mathcal C_{\eta,1}(a))\approx 
\frac{\theta^{d-1}}{d-1}
\int_0^\pi    \sin^{d-3}(\phi_{2})   d\phi_{2}
... 
\int_0^\pi    \sin(\phi_{d-2})   d\phi_{d-2}  
\int_0^{2\pi} 			 d\phi_{d-1}
\int_0^1 r^{d-1} dr. 
\end{eqnarray*}
Therefore there exists a positive constant  $C>0$ such that
$$
\frac{\hbox{vol}(\mathcal C_{\eta,1}(a))}{\hbox{vol}(B_1)} \geq C \theta^{d-1}= C\eta^{\frac{d-1}{2}}
$$
and hence, denoting by $C_\mu = C \int_{B_1} |z|^2\frac{dz}{|z|^{2+\beta}}$,  $(M_2)$ is satisfied
$$
\int_{\mathcal C_{\eta,\delta}(a)} |z|^2 \frac{dz}{|z|^{2+\beta}} \geq C\eta^{\frac{d-1}{2}} \delta^{2-\beta}  \int_{B_1} |z|^2\frac{dz}{|z|^{2+\beta}} 
= C_\mu\eta^{\frac{d-1}{2}} \delta^{2-\beta}.
$$
}\end{example}

\subsection{L\'evy Measures for L\'evy-It\^o Operators}
L\'evy-It\^o operators are defined by
\begin{equation}\label{eq_LIterm}
\mathcal J[x,u] = \int_{\R^{\tilde d}} \big( u(x + j(x,z)) - u(x) - Du(x)\cdot j(x,z)1_B(z)\big)\mu(dz).
\end{equation}
In the sequel, we assume that the jump function(s) satisfies the following conditions.
\begin{itemize}
\item [($J1$)] {There exists a constant $\tilde C_\mu>0$ such that for all $x\in\R^{\tilde d}$
	$$
	\int_B|j(x,z)|^2\mu(dz)+\int_{\R^{\tilde d}\setminus B}\mu(dz)\leq \tilde C_\mu.
	$$}
\item [($J2$)] There exists $\beta\in(0,2)$ such that {for every $a\in\R^{\tilde d}$ there exist $0<\eta<1$ 
and a constant $C_\mu>0$ }such that the following holds for any $x\in\R^{\tilde d}$
	$$ 
	 \forall \delta>0 \; \; \int_{\mathcal C_{\eta, \delta}(a)}|j(x,z)|^2 \mu(dz)\geq C_\mu\ \eta^{\frac{d-1}{2}}\ \delta^{2-\beta}
	$$
	with $\mathcal C_{\eta,\delta}(a):= \{z; |j (x,z)|\leq\delta, \space\ (1-\eta)|j(x,z)||a|\leq|a\cdot j(x,z)|\}$.
\item [($J3$)] There exists $\beta\in(0,2)$ such that for $\delta>0$ small enough
	$$
	\int_{B\setminus B_\delta} |z| \mu(dz) \leq 
	\left\{ 
	\begin{array}{ll}
	 \tilde C_\mu \delta^{1-\beta}, & \hbox{ if } \beta \neq 1\\
	 \tilde C_\mu |\ln\delta| & \hbox{ if } \beta = 1.
	 \end{array}
	\right.
	$$
\item [($J4$)] {There exist $ \gamma\in(0,1]$ 
and two constants $c_0, C_0>0$} such that for any $x\in\R^{\tilde d}$ and $z\in \R^{\tilde d}$
	$$
	c_0 |z|\leq |j(x,z)|\leq C_0|z|
	$$
	and for all $z\in B$ and  $x,y\in\R^{\tilde d}$
	$$
 	|j(x,z)-j(y,z)|\leq C_0|z||x-y|^\gamma.
 	$$
\item	[($J5$)] {There exist $ \gamma\in(0,1]$ and a constant $\tilde C_0>0$ such that for all $z\in \R^{\tilde d}\setminus B$ and  $x,y\in\R^{\tilde d}$
	$$
 	|j(x,z)-j(y,z)|\leq \tilde C_0|x-y|^\gamma.
 	$$}
\end{itemize}
When several assumptions hold simultaneously, the constants denoted similarly are considered to be the same 
(e.g. $\beta$, $C_\mu$, $\tilde C_\mu$).

\section{Lipschitz Continuity of Viscosity Solutions}\label{sec_Lip_regularity}

In this section we present the main regularity results for mixed
integro-differential equations.  We deal with \emph{general
  nonlinearities} derived from the toy model, namely
Equation~(\ref{simpl-mide}), where the fractional diffusion gives the
ellipticity in certain directions and the classical diffusion in the
complementary ones.  We first establish partial regularity results,
namely H\"older and Lipschitz regularity of the solution with respect
to the $x_1$-variables.  This is because of the lack of complete local
or nonlocal diffusion. We then derive the global regularity of the
solution.  \smallskip

For the sake of simplicity, we give the statements and proofs in the
periodic setting. This yields $C^{0,\alpha}$ regularity instead of
local regularity.  At the same time it allows us to avoid the
localization terms, meant to overcome the behavior at infinity of the
solutions, which is related to the integrability of the singular
measure away from the origin.

\subsection{Partial Regularity Results}

We first give partial regularity estimates, in which
  case we use classical regularity arguments in one set of variables,
  and uniqueness type arguments in the other variables.  Regularity
  arguments apply for both general nonlocal operators and L\'evy-It\^o
  operators. However, uniqueness applies only for the latter.
  Consequently, we state two results: one for equations that mix
  general nonlocal operators with L\'evy-It\^o ones, and another one
  for equations dealing only with L\'evy-It\^o operators.  

\begin{theorem}[\textbf{Partial regularity for periodic, mixed PIDEs - general nonlocal operators}]
\label{thm_part_Lip_NL}
Let $f$ be a continuous, periodic function.
Assume the nonlinearities $F_i$, $i=0,1,2$ are degenerate elliptic and that they satisfy the following:
\begin{itemize*}
 \item [-]$F_0$ is $\Z^d$-periodic and satisfies assumptions $(H0)$, $(H2)$ with $\tilde d = d$ and some constant $\tilde \gamma$;
 \item [-]$F_1$ is $\Z^{d_1}$-periodic and satisfies $(H1)$ with $\tilde d = d_1$,
		for some functions ${\Lambda_1}$ , ${\Lambda_2}$ and some parameters ${\Lambda_0}$, $k\geq 0,$
		$\tau,\theta,\tilde \theta\in(0,1]$;
 \item [-]$F_2$ is $\Z^{d_2}$-periodic and satisfies $(H2)$, $(H3)$ with $\tilde d = d_2$.
\end{itemize*}
Let $\mu^0$, $\big(\mu^1_{x_1}\big)_{x_1}$ and $\mu^2$ be L\'evy measures on $\R^d$, $\R^{d_1}$, $\R^{d_2}$ 
respectively associated to the integro - differential operators 
$ \mathcal I[x,u]$, $\mathcal I_{x_1}[x,u] $ and $\mathcal J_{x_2}[x,u]$. 
Suppose
\begin{itemize*}
 \item [-]$\big(\mu^1_{x_1}\big)_{x_1}$  satisfies $(M1)-(M3)$ for some
 	      $C_{\mu^1}$, $\tilde C_{\mu^1}$, $\beta$ and $\gamma$, 
	      with $\left\{ \begin{array}{ll}
                          k\leq \beta,& \beta> 1\\
                          k  <  \beta, & \beta \leq  1;
                          \end{array}\right.$
 \item [-]the jump function $j(x_2,z)$ satisfies $(J1)$,{$(J4)$} and $(J5)$ 
	      for some  $C_{\mu^2}$, $\tilde C_{\mu^2}$, and $\gamma = 1$. 
\end{itemize*}
Then any periodic continuous viscosity solution $u$ of 
\begin{eqnarray}
F_0(u(x),Du, D^2u, \mathcal I[x,u]) 
&+& F_1(x_1,D_{x_1}u, D_{x_1x_1}^2u, \mathcal I_{x_1}[x,u]) \\ \nonumber 
&+&  
F_2(x_2,D_{x_2}u, D_{x_2x_2}^2u, \mathcal J_{x_2}[x,u]) = {f(x)}
\end{eqnarray}
\begin{itemize}
 \item[(a)] is Lipschitz continuous in the $x_1$ variable if $\beta>1$;
 \item[(b)]  is $C^{0,\alpha}$ continuous in the $x_1$ variable with $\alpha<\frac{\beta-k}{1-k}$, if $\beta\leq1$.
\end{itemize}
The Lipschitz / H\"older constant $L$ depends on $||u||_\infty$, the dimension of the space $d$, the constants associated to the L\'evy measures as well as the constants required by the growth condition $(H1)$.
\end{theorem}

\begin{remark}
 In particular, when $d_1 = d$ and $F_0\equiv 0,F_2\equiv 0$ 
we extend to Lipschitz the H\"older regularity result, recently obtained by Barles, Chasseigne and Imbert in \cite{BCI:11:HdrNL}.
\end{remark}

\begin{remark}
When $k=\beta=1$, the solution is $\alpha$-H\"older continuous, with $\alpha$ small enough. 
Unfortunately in this case we cannot characterize the  H\"older exponent $\alpha$.
\end{remark}

\begin{remark}
When $\beta<1$, if ${C_1} = 0$ in $(H1)$ and $\beta(k+\tau)>k$, 
then the solution is exactly $C^{0,\beta}$.
\end{remark}

Since the concave estimates for L\'evy-It\^o operators are of the same order as those for general nonlocal operators, 
similar regularity results hold. Namely, we have the following.

\begin{theorem}[\textbf{Partial regularity for periodic, mixed PIDEs - L\'evy-It\^o operators}]
{Let $f$ and $F_i$, $i=0,1,2$ satisfy the same assumptions as in Theorem \ref{thm_part_Lip_NL}.}
Let $\mu^0$, $\mu^1$ and $\mu^2$ be L\'evy measures on $\R^d$, $\R^{d_1}$ and $\R^{d_2}$, 
respectively associated to the integro-differential operators 
$ \mathcal I[x,u]$, $\mathcal J_{x_1}[x,u] $ and $\mathcal J_{x_2}[x,u]$. 
Suppose
\begin{itemize*}
 \item [-]the jump function $j^1(x_1,z)$ satisfies assumptions $(J1)$ - $(J4)$, for some parameters 
$\beta$, $C_{\mu^1}$, $\tilde C_{\mu^1}$, and $\gamma \in (1-\beta/2,1]$, and in addition $\left\{ \begin{array}{ll}
                          k\leq \beta,& \beta> 1\\
                          k  <  \beta, & \beta \leq  1;
                          \end{array}\right.$
 \item [-]the jump function $j^2(x_2,z)$ satisfies 
 $(J1)$,{$(J4)$} and $(J5)$ for some  $C_{\mu^2}$, $\tilde C_{\mu^2}$, and $\gamma = 1$. 
\end{itemize*}
Then any {periodic continuous} viscosity solution $u$ of 
\begin{eqnarray}
F_0(u(x),Du, D^2u, \mathcal I[x,u]) 
& + &  F_1(x_1,D_{x_1}u, D_{x_1x_1}^2u, \mathcal J_{x_1}[x,u]) \\ \nonumber
& + & F_2(x_2,D_{x_2}u, D_{x_2x_2}^2u, \mathcal J_{x_2}[x,u]) = f(x)
\end{eqnarray}
\begin{itemize}
 \item [(a)] is Lipschitz continuous in the $x_1$ variable, if $\beta>1$;
 \item [(b)] is $C^{0,\alpha}$ continuous in the $x_1$ variable with $\alpha<\frac{\beta-k}{1-k}$, if $\beta\leq1$.
\end{itemize}
The Lipschitz / H\"older constant $L$ depends on $||u||_\infty$, the dimension $d$ of the space , 
the constants associated to the L\'evy measures as well as the constants required by the growth condition $(H1)$.
\end{theorem}

\begin{remark}
In order to establish Lipschitz or H\"older regularity results for the solution $u$, 
we shift the function and show that the corresponding difference can be uniformly controlled by 
$$\phi(t) = L t^\alpha, \hbox{ for all } \alpha\in(0,1].$$
\begin{figure}[h!]
\centering
\includegraphics[width=0.65\linewidth]{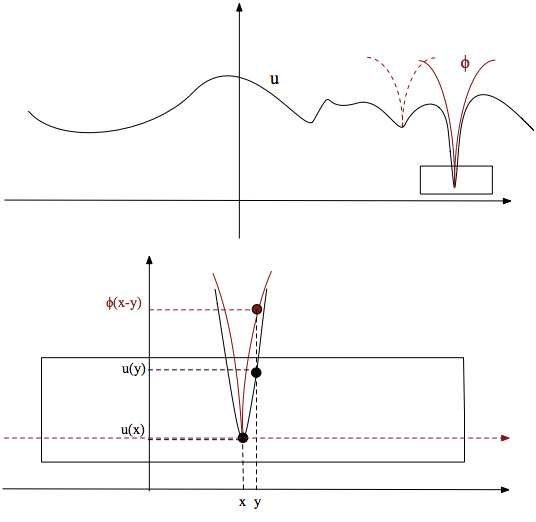}
\caption{\small Uniformly controlling the shift of $u$ by $\phi(|x-y|) = L |x-y|^\alpha$, for all $\alpha\in(0,1]$ . }
\end{figure}
Roughly speaking, one has to look at the maximum of the function
$$
(x,y) \mapsto u(x) - u(y) - \phi(|x-y|)
$$
and, \emph{in the case of elliptic PDEs,} follow the uniqueness proof 
with a careful analysis of the matrix inequality given by Jensen-Ishii's lemma.  
Precise computations show that we just need ellipticity of the equation in the gradient direction.
In the case of nonlocal diffusions, one has to translate in a proper way the ellipticity in the gradient direction.
This is reflected in the nondegeneracy conditions $(M2)$ (respectively $(J2)$) required by the family of L\'evy measures.
\end{remark}

\begin{proof}[Proof of Theorem \ref{thm_part_Lip_NL}]
The proof of the regularity of $u$ consists of two steps: 
we first show that the solution $u$ is $C^{0,\alpha}$ continuous for all $\alpha\in(0,1)$, 
then we check that in the subcritical case $\beta>1$ this implies the Lipschitz continuity. 
We use the viscosity method introduced by Ishii and Lions in \cite{IL:90:ViscPDE}. \medskip

STEP 1. We introduce the auxiliary function 
$$\psi(x_1,y_1,x_2) = u(x_1,x_2) - u(y_1,x_2) - \phi(x_1-y_1)$$
where $\phi$ is a radial function of the form
$$\phi(z) =\varphi(|z|)$$ 
with a suitable choice of a smooth increasing concave function $\varphi:\R_+\rightarrow\R_+$ satisfying
$\varphi(0)=0 $ and 
$\varphi(t_0)\geq 2||u||_\infty \hbox{ for some }t_0>0.$
{Our aim is to show that for all $x_2 \in\R^{d_2}$
\begin{equation}\label{eq_part_reg}
\psi(x_1,y_1,x_2)\leq 0 \; \hbox{ if } |x_1 - y_1|<t_0.
\end{equation}}
This yields the desired regularity result, for a proper choice of $\varphi$. 
Namely, $\varphi =L t^\alpha$ will give the partial H\"older regularity of the solution
$$
|u(x_1,x_2) - u(y_1,x_2)| \leq L|x_1 - y_1|^\alpha, \hbox{ if } |x_1 - y_1|<t_0
$$
and $\varphi =L( t -\rho t^{1+\alpha})$ the partial Lipschitz regularity 
$$
|u(x_1,x_2) - u(y_1,x_2)| \leq L|x_1 - y_1|, \hbox{ if } |x_1 - y_1|<t_0.
$$

STEP 2. To this end, we argue by contradiction and assume that $\psi(x_1,y_1,x_2)$ has a positive strict maximum 
at some point $(\bar x_1,\bar y_1,\bar x_2)$ with $|\bar x_1 - \bar y_1|<t_0$:
$$M = \psi(\bar x_1, \bar y_1,\bar x_2) = \max_{\substack{x_1,y_2\in\R^{d_1},x_2\in\R^{d_2} \\ |x_1-y_1|<t_0}}\psi(x_1,y_1,x_2) > 0.$$
Denote by $\bar x = (\bar x_1,\bar x_2)$ and by $\bar y = (\bar y_1, \bar x_2)$. Then
\begin{eqnarray}
\label{eq_mod_cont} 
&&\varphi(|\bar x - \bar y|) \leq u(\bar x) - u(\bar y) \leq \omega_u(|\bar x - \bar y|) \\
\label{eq_mod_inf}  
&&\varphi(|\bar x - \bar y|) \leq u(\bar x) - u(\bar y) \leq 2||u||_\infty.
\end{eqnarray}

To be able to extract some valuable information hereafter, we need to construct test functions defined on the whole space $\R^d$. 
For this reason, we penalize $\psi$ around the maximum by doubling the variables, staying at the same time  as close as possible to the maximum point. Therefore, we consider the auxiliary function
$$\psi_\varepsilon(x,y) = u(x_1,x_2) - u(y_1,y_2) - \phi(x_1-y_1)-\frac{|x_2-y_2|^2}{\varepsilon^2}$$
whose maximum is attained, say at $( x^\varepsilon, y^\varepsilon)$. Denote its maximum value by
$$M^\varepsilon = \psi_\varepsilon ( x^\varepsilon, y^\varepsilon) = \max_{x,y\in\R^d} \psi_\varepsilon(x,y).$$
Then the following holds.

\begin{lemma}\label{lemma_max_points}
There exists $(\bar x,\bar y)$ such that $ M = \psi(\bar x_1, \bar y_1,\bar x_2)$ and up to a subsequence, the sequences of maximum points $\big(( x^\varepsilon, y^\varepsilon)\big)_\varepsilon$ and of maximum values $(M^\varepsilon)_\varepsilon$ satisfy$\hbox{ as }\varepsilon\rightarrow 0 $
$$
 M^\varepsilon \rightarrow M,  \hspace{2cm}
\frac{| x_2^\varepsilon - y_2^\varepsilon|^2}{\varepsilon^2}  \rightarrow0, \hspace{2cm}
( x^\varepsilon, y^\varepsilon) \rightarrow(\bar x,\bar y).
$$
\end{lemma}

The proof of this lemma is classical and therefore omitted in this paper.\medskip

STEP 3. Let \ $ \ \bar a = (\bar a_1, \bar a_2) = \bar x - \bar y\ $, $ \ p = (p_1, p_2) = (D\phi(\bar a_1),0)\ $ and denote by
$$
a^\varepsilon= (a_1^\varepsilon,a_2^\varepsilon)=  x^\varepsilon - y^\varepsilon, \hspace{0.5cm}
\hat a^\varepsilon = \frac{a^\varepsilon}{|a^\varepsilon|},\hspace{0.5cm}
p^\varepsilon = (p_1^\varepsilon, p_2^\varepsilon) = 
(D\phi(a_1^\varepsilon), 2\frac{ x_2^\varepsilon -y_2^\varepsilon}{\varepsilon^2}).
$$
{Since $x_1^\varepsilon\neq y_1^\varepsilon$, for $\varepsilon$ small enough} 
the function $\phi$ is smooth and we can apply the Jensen-Ishii's lemma 
for integro-differential equations \cite{BI:08:ViscNL}. This yields the existence, for each $\varepsilon>0$, of two sequences of matrices 
$(X^{\varepsilon,\zeta})_\zeta, (Y^{\varepsilon,\zeta})_\zeta \subset \SS^d$ of the form
\begin{equation}\label{eq_block_matrices_eps}
X^{\varepsilon,\zeta}
= 
\begin{bmatrix} 
X_1^{\varepsilon,\zeta} &  0 \\
0 & X_2^{\varepsilon,\zeta}
\end{bmatrix} 
\hbox{ and }
Y^{\varepsilon,\zeta}
= 
\begin{bmatrix} 
Y_1^{\varepsilon,\zeta} &  0 \\
0 & Y_2^{\varepsilon,\zeta}
\end{bmatrix}, 
\end{equation}
which correspond to the subjets and superjets of $u$ at the points $ x^\varepsilon$ and $y^\varepsilon$.
In  addition the block diagonal matrix satisfies
\begin{equation}\label{matrix_ineq}
- \frac{1}{\zeta}
\begin{bmatrix} 
I_d &  0 \\
0 & I_d
\end{bmatrix} 
\leq 
\begin{bmatrix} 
X^{\varepsilon,\zeta} &  0 \\
0 & - Y^{\varepsilon,\zeta}\\
\end{bmatrix} 
\leq 
\begin{bmatrix} 
 Z & -Z \\
 -Z & Z
\end{bmatrix} + o_\zeta(1),
\end{equation} 
with $Z$ a block matrix of the form 
\begin{equation}
 \begin{bmatrix} 
Z_1 &  0 \\
0 & Z_2
\end{bmatrix}
\end{equation}
with blocks
\begin{eqnarray*}
Z_1 & = & D^2\phi( a_1^\varepsilon) = \frac{\varphi'(| a_1^\varepsilon|)}{| a_1^\varepsilon|} I_{d_1}+
\left(\varphi''(| a_1^\varepsilon|) - \frac{\varphi'(| a_1^\varepsilon|)}{| a_1^\varepsilon|} \right)
\hat{a}_1^\varepsilon \otimes \hat{a}_1^\varepsilon\\
Z_2 & = & \frac{2}{\varepsilon^2} I_{d_2}.
\end{eqnarray*}
By Lemma \ref{lemma_block_ineq} the triple of block matrices
$(X_i^{\varepsilon,\zeta},Y_i^{\varepsilon,\zeta},Z_i)$ for $i=1,2$ satisfy  (\ref{matrix_ineq}).
Then, by sup and inf matrix convolution (see Lemmas \ref{lemma_bloc_ineq_conv} and 
\ref{lemma_matrix_conv} in Appendix) we build matrices, that we still denote by
$X^{\varepsilon,\zeta}$ and  $Y^{\varepsilon,\zeta}$, for which the corresponding blocks 
$X_i^{\varepsilon,\zeta}$ and  $Y_i^{\varepsilon,\zeta}$ for $i=1,2$ satisfy 
{uniform} bounds 
\begin{eqnarray}
\label{eq_matrix_convex_est}
-\frac{2}{\bar \varepsilon} 
\begin{bmatrix} 
I_{d_1} &  0 \\
0 & I_{d_1}
\end{bmatrix} 
&\leq 
\begin{bmatrix} 
X_1^{\varepsilon,\zeta} &  0 \\
0 & - Y_1^{\varepsilon,\zeta}\\
\end{bmatrix} 
&\leq 
\begin{bmatrix} 
 \tilde Z_1 & - \tilde Z_1\\
-\tilde Z_1 &   \tilde Z_1
\end{bmatrix} + o_\zeta(1)\\
\label{eq_matrix_qdr_est}
- \frac{4}{\varepsilon^2}
\begin{bmatrix} 
I_{d_2} &  0 \\
0 & I_{d_2}
\end{bmatrix} 
&\leq 
\begin{bmatrix} 
X_2^{\varepsilon,\zeta} &  0 \\
0 & - Y_2^{\varepsilon,\zeta}\\
\end{bmatrix} 
&\leq \frac{4}{\varepsilon^2}
\begin{bmatrix} 
I_{d_2} &  0 \\
0 & I_{d_2}
\end{bmatrix}  + o_\zeta(1)
\end{eqnarray} 
with $\tilde Z_1 = Z_1^{\frac{\bar \varepsilon}{2}}$, where
$$
\bar \varepsilon = \frac{| a_1^\varepsilon|}{\varphi'(| a_1^\varepsilon|)}.
$$
In addition, from the monotonicity of the sup and inf convolution (\ref{eq_mon_conv}) the new block matrices $X^{\varepsilon,\zeta}$ and $Y^{\varepsilon,\zeta}$ are still sub and superjets of $u$ at $ x^\varepsilon$, respectively $y^\varepsilon$
\begin{eqnarray*}
(p^\varepsilon,X^{\varepsilon,\zeta})\in \mathcal D^{2,+}(u( x^\varepsilon))\\
(p^\varepsilon,Y^{\varepsilon,\zeta})\in \mathcal D^{2,-}(u(y^\varepsilon)).
\end{eqnarray*}
Since the bounds in (\ref{eq_matrix_convex_est}) and (\ref{eq_matrix_qdr_est}) are uniform with respect to $\zeta$, we can let $\zeta\rightarrow 0$ and obtain two matrices $X^\varepsilon$ and  $Y^\varepsilon$ satisfying the double inequality required by the ellipticity growth condition (H1), which are still sub and superjets of $u$ at $ x^\varepsilon$ and $y^\varepsilon$ respectively.
Hence, they satisfy the viscosity inequalities
\begin{eqnarray*}
 F_0(u( x^\varepsilon), p^\varepsilon, X^\varepsilon, \mathcal I[ x^\varepsilon, p^\varepsilon, u]) + 
\sum_{i=1,2} F_i(\bar x_i^\varepsilon, p_i^\varepsilon, X_i^\varepsilon, \mathcal I_{x_i}[ x^\varepsilon, p_i^\varepsilon, u]) 
\leq {f(x^\varepsilon)} \\
F_0(u(y^\varepsilon), p^\varepsilon, Y^\varepsilon, \mathcal I[y^\varepsilon, p^\varepsilon, u]) + 
\sum_{i=1,2} F_i(\bar y_i^\varepsilon, p_i^\varepsilon, Y_i^\varepsilon, \mathcal I_{y_i}[y^\varepsilon, p_i^\varepsilon, u]) 
\geq {f(y^\varepsilon)}. 
\end{eqnarray*}
Subtracting the above inequalities and denoting
\begin{eqnarray*}
E_0( x^\varepsilon,y^\varepsilon,u) & = &  F_0\left(u(y^\varepsilon), p^\varepsilon, Y^\varepsilon, \mathcal I[y^\varepsilon, p^\varepsilon, u]\right) -  
		  F_0\left(u( x^\varepsilon), p^\varepsilon, X^\varepsilon, \mathcal I[ x^\varepsilon, p^\varepsilon, u]\right)
	{+ f(x^\varepsilon) - f(y^\varepsilon)}	  
		  \\
E_i(\bar x_i^\varepsilon,\bar y_i^\varepsilon,u) & = &  F_i\left(\bar y_i^\varepsilon, p_i^\varepsilon, Y_i^\varepsilon, \mathcal I_{y_i}[y^\varepsilon, p_i^\varepsilon, u]\right) 
	        -F_i\left(\bar x_i^\varepsilon, p_i^\varepsilon, X_i^\varepsilon, \mathcal I_{x_i}[ x^\varepsilon, p_i^\varepsilon, u]\right), \; i=1,2,
\end{eqnarray*}
we get that 
\begin{equation}\label{eq_visc_ineq_estimate}
0\leq E_0( x^\varepsilon,y^\varepsilon,u)  + E_1(x_1^\varepsilon,y_1^\varepsilon,u) + E_2( x_2^\varepsilon,y_2^\varepsilon,u).
\end{equation}

STEP 4. In the following we estimate each of these terms as $\varepsilon \rightarrow 0$, bringing into play the ellipticity growth assumptions satisfied by each nonlinearity.

Since $u(y^\varepsilon) \leq u( x^\varepsilon)$, $X^\varepsilon \leq Y^\varepsilon$, the monotonicity assumption $(H_0)$, the ellipticity $(E)$ with respect to the second order term {and the nonlocal term}  and the Lipschitz continuity $(H2)$  of $F_0$ with respect to the nonlocal term yield
\begin{equation*}
E_0( x^\varepsilon,y^\varepsilon,u) \leq \tilde\gamma \big(u(y^\varepsilon) - u(x^\varepsilon)\big) + 
L_{F_0} \big(\mathcal I[ x^\varepsilon, p^\varepsilon, u] - \mathcal I[y^\varepsilon, p^\varepsilon, u] \big)_+ 
+ f(x^\varepsilon) - f(y^\varepsilon).
\end{equation*}
As the L\'evy measures corresponding to the nonlinearity $F_0$ do not depend on $x$, we immediately  deduce from the maximum condition that 
$$
u( x^\varepsilon + z) - v(y^\varepsilon + z ) \leq u( x^\varepsilon) - v(y^\varepsilon)
$$
renders nonpositive the difference of the nonlocal terms
$$
\mathcal I[ x^\varepsilon, p^\varepsilon, u] - \mathcal I[y^\varepsilon, p^\varepsilon, u] \leq 0.
$$
Therefore, passing to the limits as $\varepsilon\rightarrow 0$ and employing Lemma \ref{lemma_max_points} we have  
\begin{equation}\label{eq_estE0}
\limsup_{\varepsilon\rightarrow 0} E_0 ( x^\varepsilon,y^\varepsilon,u) \leq -\tilde\gamma M.
\end{equation}

The estimate of $E_2$ does not depend on the choice of $\varphi$ and is given by the growth condition $(H3)$ and the Lipschitz continuity $(H2)$ of $F_2(\cdot,l)$, uniformly with respect to all the other variables
\begin{eqnarray*}
E_2 ( x_2^\varepsilon, y_2^\varepsilon, u) &\leq & 
\omega_{F_2}\left(\frac{|a_2^\varepsilon|^2}{\varepsilon^2} + |a_2^\varepsilon|\right) + 
L_{F_2}\left(\mathcal I_{x_2}[ x^\varepsilon, p_2^\varepsilon, u]-\mathcal I_{y_2}[y^\varepsilon, p_2^\varepsilon, u]\right)_+	
\end{eqnarray*}
where $L_{F_2}$ is the Lipschitz constant of  $F_2(\cdot, l)$.
From Proposition \ref{prop_QdrEstLI} in Section \ref{sec_Lip_estimates} the quadratic estimates for L\'evy-It\^o operators hold
{
\begin{eqnarray*}
\mathcal I_{x_2}[ x^\varepsilon, p_2^\varepsilon, u]-\mathcal I_{y_2}[y^\varepsilon, p_2^\varepsilon, u] 
&\leq & C\frac{1}{\varepsilon^2} \int_{B_\delta} |z_2|^2 \mu^2(dz_2) + CC_{\mu^2}\frac{|a_2^\varepsilon|^{2}}{\varepsilon^2}.  
\end{eqnarray*}
}
for some positive constant $C$. As $\delta\rightarrow 0$, the estimate gives
\begin{eqnarray*}
&& \mathcal I_{x_2}[ x^\varepsilon, p_2^\varepsilon, u]-\mathcal I_{y_2}[y^\varepsilon, p_2^\varepsilon, u] \leq 
C \tilde C_{\mu^2}\frac{|a_2^\varepsilon|^{2}}{\varepsilon^2}.
\end{eqnarray*}
Letting now $\varepsilon \rightarrow 0$ and using Lemma \ref{lemma_max_points} which ensures that $\frac{|a_2^\varepsilon|^{2}}{\varepsilon^2}\rightarrow 0$ we are finally lead to
\begin{equation}\label{eq_estE2}
\limsup_{\varepsilon\rightarrow 0} E_2 ( x_2^\varepsilon, y_2^\varepsilon, u)\leq 0.
\end{equation}

For the estimate of $E_1$, we use the ellipticity growth condition $(H1)$ 
\begin{eqnarray}\nonumber
E_1 (x_1^\varepsilon, y_1^\varepsilon, u) &\leq &
{\Lambda_1}(x_1^\varepsilon)
	      \Big( \big( \mathcal I_{x_1}[ x^\varepsilon, p_1^\varepsilon, u]-\mathcal I_{y_1}[y^\varepsilon, p_1^\varepsilon, u]\big)+ \frac{| a_1^\varepsilon|^{2\theta}}{\bar \varepsilon} +
	      | a_1^\varepsilon|^\tau|p_1^\varepsilon|^{k+\tau} + 
	      {{C_1} |p_1^\varepsilon|^k}\Big)
\\ \label{eq_estE1}&&
+ {\Lambda_2}(x_1^\varepsilon)
	      \Big(\hbox{tr}(X_1^\varepsilon-Y_1^\varepsilon) + 
	      \frac{| a_1^\varepsilon|^{2\tilde\theta}}{\bar \varepsilon} + 
	      | a_1^\varepsilon|^\tau|p^\varepsilon_1|^{2+\tau}+ 
	      {{C_2} |p_1^\varepsilon|^2}\Big)
\end{eqnarray}
where we recall that $p_1^\varepsilon = D\phi(a_1^\varepsilon) =L\varphi'(|a_1^\varepsilon|) \hat a_1^\varepsilon$. 
The goal is to show that, for each choice of $\varphi$ (measuring either the H\"older or the Lipschitz continuity), the right hand side quantity is negative, arriving thus to a contradiction by combining (\ref{eq_visc_ineq_estimate}), (\ref{eq_estE0}),  (\ref{eq_estE2})  and  (\ref{eq_estE1}).

\medskip
STEP 5.1. \textbf{H\"older continuity.} In order to establish the H\"older regularity of solutions, we consider the auxiliary function 
$$
\varphi = L t^\alpha, \hbox{ with } \alpha < \min(1, \beta).
$$
In this case, we apply Corollary \ref{cor_HoeEst} from Section \ref{sec_Lip_estimates}, to the functions  $ u(\cdot,x_2)$ and  $u(\cdot,y_2)$, which yields the following H\"older estimate for the difference of the nonlocal terms
\begin{eqnarray*}&&
  \mathcal I_{x_1}[ x^\varepsilon, p_1^\varepsilon, u]-\mathcal I_{y_1}[y^\varepsilon, p_1^\varepsilon, u] \leq 
  - L |a_1^\varepsilon|^{\alpha -\beta}\left\{ {\alpha C({\mu^1})} - o_{|a_1^\varepsilon|}(1)\right\}+O(1).
\end{eqnarray*}
Lemma \ref{lemma_TrEst} from Appendix applies with $\tilde Z_1 = Z_1^{\frac{\bar \epsilon}{2}}$, 
$\bar\varepsilon =\big(L\alpha | a_1^\varepsilon|^{\alpha - 2}\big)^{-1}$, $\omega = 2-\alpha$ and hence the trace is bounded by
\begin{equation}
\hbox{trace}(X_1^\varepsilon - Y_1^\varepsilon) \leq - 8 \bar \omega \big(L\alpha |a_1^\varepsilon|^{\alpha - 2}\big)
\end{equation}
where $\bar \omega = \frac{\omega-1}{\omega+1}$ is a constant in $(0,\frac13)$.
We plug these estimates into the inequality for  $E_1$. Letting $\varepsilon$ go to zero and 
employing the penalization Lemma \ref{lemma_max_points} and  $(H4)$ we obtain the following bound
\begin{eqnarray*}
\limsup_{\varepsilon\rightarrow 0} E_1 (x_1^\varepsilon, y_1^\varepsilon, u) \leq 
{\Lambda_0} \ \mathcal E^1(|\bar a|) + {\Lambda_0}\ \mathcal  E^2(|\bar a|) +{O(1)} 
\end{eqnarray*}
where for $2\theta + \beta > 2$ 
\begin{eqnarray*}
\mathcal E^1(|\bar a|) & = & - L |\bar a|^{\alpha -\beta}\left( {\alpha C(\mu^1)}- o_{|\bar a|}(1)\right)
	     + |\bar a|^{2\theta}\big(L\alpha |\bar a|^{\alpha - 2}\big)
	     + |\bar a|^\tau \left(L\alpha |\bar a|^{\alpha-1}\right)^{k+\tau} +  
	      {{C_1}  \left(L\alpha |\bar a|^{\alpha-1}\right)^k }
\\	
& = &  
- L |\bar a|^{\alpha -\beta}
\left\{ 
{\alpha C(\mu^1)}- o_{|\bar a|}(1) - 
\alpha^{k+\tau} |\bar a|^{\beta-k} \left( L |\bar a|^\alpha \right)^{k+\tau - 1} - 
{{C_1} \alpha^{k} |\bar a|^{\beta-k} \left( L |\bar a|^\alpha \right)^{k-1}} 
\right\}   
\end{eqnarray*}
and 
\begin{eqnarray*}
\mathcal E^2(|\bar a|) & = & - 8 \bar \omega \big(L\alpha |\bar a|^{\alpha - 2}\big) + 
	     {|\bar a|^{2\tilde\theta} }\big(L\alpha |\bar a|^{\alpha - 2}\big) + 
	      |\bar a|^\tau\left(L\alpha |\bar a|^{\alpha-1}\right)^{2+\tau}+	  
	      {{C_2}  \left(L\alpha |\bar a|^{\alpha-1}\right)^2} 
\\
& = &  
- L |\bar a|^{\alpha -2} 
\left\{ 
\alpha \big( 8\bar \omega -{|\bar a|^{2\tilde\theta} } \big) -
\alpha^{2+\tau} \left( L|\bar a|^\alpha \right)^{1+\tau} -
{{C_2} \alpha^2 L|\bar a|^\alpha}
\right\}.  
\end{eqnarray*}
Using the fact that $L |\bar a|^{\alpha} \leq 2 ||u||_\infty$ we have 
 \begin{eqnarray*}
\mathcal E^2(|\bar a|) & \leq &  
- L |\bar a|^{\alpha -2} 
\left\{ 
\alpha \big( 8\bar \omega -{|\bar a|^{2\tilde\theta} } \big) -
\alpha^{2+\tau} \left( 2||u||_\infty \right)^{1+\tau} -
{{C_2}  \alpha^2\left( 2||u||_\infty\right)}
\right\}.  
\end{eqnarray*}
As far as  $\mathcal E^1$ is concerned, we further argue differently for the subcritical and supercritical case, with respect to the L\'evy exponent $\beta$, and accordingly with respect to $k$ and $\tau$. Namely
\begin{itemize}
 \item [(a)] if {$1< k \leq \beta$, in which case $k+\tau -1 >0$, $k-1> 0$, we have}
 \begin{eqnarray*}
\mathcal E^1(|\bar a|) & \leq &  
- L |\bar a|^{\alpha -\beta}
\Big\{ 
{\alpha C(\mu^1)}- o_{|\bar a|}(1)
- \alpha^{k+\tau} {|\bar a|^{\beta-k}}\left( 2||u||_\infty\right)^{k+\tau - 1} 
\\ && \hspace{3.75cm}{-{C_1}  \alpha^{k} {|\bar a|^{\beta-k}}\left( 2||u||_\infty\right)^{k-1}} 
\Big\}.  
\end{eqnarray*}
 \item [(b)] if {$k<\min(1,\beta)$}, then 
 \begin{itemize}
 \item [(b.1)] for $0< k {\leq} 1 - \tau$ and  $\beta-k + \alpha(k+\tau-1)>0$ 
 \begin{eqnarray*}
\mathcal E^1(|\bar a|) & \leq &  - L |\bar a|^{\alpha -\beta}
\Big\{ 
{\alpha C(\mu^1)}- o_{|\bar a|}(1) - 
\alpha^{k+\tau} |\bar a|^{\beta-k + \alpha(k+\tau-1)}  L ^{k+\tau - 1}  
\\ && \hspace{3.75cm}
{-{C_1}  \alpha^{k} |\bar a|^{\beta-k + \alpha(k-1)} L ^{k-1}} 
\Big\}   \\
& = &  - L |\bar a|^{\alpha -\beta} \Big( {\alpha C(\mu^1)}- o_{|\bar a|}(1) \Big).
\end{eqnarray*}
\item [(b.2)] {for $1-\tau < k \leq 1 $} and  $\beta-k + \alpha(k+\tau-1)>0$ 
 \begin{eqnarray*}
\mathcal E^1(|\bar a|) & \leq &  
- L |\bar a|^{\alpha -\beta}
\Big\{ 
{\alpha C(\mu^1)}- o_{|\bar a|}(1) - 
\alpha^{k+\tau} \left( 2||u||_\infty\right)^{k+\tau - 1} \\
&& \hspace{3.75cm}
{ -{C_1}  \alpha^{k} |\bar a|^{\beta-k + \alpha(k-1)} L ^{k-1}} 
\Big\}   \\
& = & - L |\bar a|^{\alpha -\beta}\Big\{ {\alpha C(\mu^1)}- o_{|\bar a|}(1) - \alpha^{k+\tau} \left( 2||u||_\infty\right)^{k+\tau - 1}\Big\}.
\end{eqnarray*}
\end{itemize}
\end{itemize}
This implies that for $\alpha$ small enough the two terms become (large) negative 
$$
\lim_{L\rightarrow\infty}\mathcal E^1(|\bar a|) = - \infty \hbox{ and } \lim_{L\rightarrow\infty}\mathcal E^2(|\bar a|) =- \infty.
$$ 
Hence
\begin{equation}\label{eq_estE1_Hdr}
\lim_{L\rightarrow\infty}\limsup_{\varepsilon\rightarrow 0} E_1 (x_1^\varepsilon, y_1^\varepsilon, u) =- \infty.
\end{equation}
{We now turn back to inequality (\ref{eq_visc_ineq_estimate}), let first $\varepsilon\rightarrow 0$ and then $L\rightarrow\infty$. Plugging in the estimates (\ref{eq_estE0}) - (\ref{eq_estE1_Hdr})  we arrive to a contradiction.} 
Therefore, we have proved up to this point the  $C^{0,\alpha}$ regularity of the solution, for $\alpha$ small enough. 
{Note that the exponent $\alpha$ only depends on $||u||_\infty$, $k$ and $\tau$.} 

We further use this first step to provide the $C^{0,\alpha}$ regularity for all $\alpha\in(0,1)$. To this end, we estimate 
$L|\bar a|^{\alpha}$ with the modulus of continuity of $u$ and get 
\begin{eqnarray*}
\mathcal E^2(|\bar a|) & \leq &  
- L |\bar a|^{\alpha -2} 
\left\{ 
\alpha \big( 8\bar \omega -{|\bar a|^{2\tilde\theta} } \big) -
\alpha^{2+\tau} \left( \omega_u(|\bar a|) \right)^{1+\tau} -
{{C_2} \alpha^2 \omega_u(|\bar a|)}
\right\}.
\end{eqnarray*}
Taking into account that {$\omega_u(|\bar a|)\leq \bar L |\bar a|^{\bar \alpha}$} for some $\bar \alpha$ small, we come back to the original estimates in case {$k>1$} and to the estimates given in $(b.1)$ when $k\in(0,1-\tau)$, respectively $(b.2)$ when $k\in(1-\tau,1)$, where $\alpha$ is everywhere replaced with $\bar \alpha$.
By similar arguments we obtain
\begin{eqnarray*}
\mathcal E^1(|\bar a|) &\leq &- L |\bar a|^{\alpha -\beta} \Big({\alpha C(\mu^1)}- o_{|\bar a|}(1)\Big)\\
\mathcal E^2(|\bar a|) &\leq& - L |\bar a|^{\alpha -2} \Big({\alpha C(\mu^1)}- o_{|\bar a|}(1)\Big).
\end{eqnarray*}
This yields (\ref{eq_estE1_Hdr}) for $L$ sufficiently large, and therefore completes the $C^{0,\alpha}$ regularity result.
\bigskip

STEP 5.2. \textbf{Lipschitz continuity.}
In the case $\beta>1$, we establish the Lipschitz regularity of solutions. Therefore, we consider the auxiliary function
\begin{equation*}
 \varphi(t) = \left\{ 
 \begin{array}{ll}
  L\left ( t - \varrho t^{1+\alpha}\right), & t\in [0,t_0]\\
  \varphi(t_0), &t>t_0
 \end{array}\right.
\end{equation*}
where $\alpha\in(0,1)$ {will be chosen} small enough, $\rho$ and $t_0$ as in Corollary \ref{cor_LipEst} in Section \S \ref{sec_Lip_estimates}.
We remind that $\alpha$ is related  to the aperture of the cone corresponding to $\eta\sim |\bar a|^{2\alpha}$. 
In order to estimate the difference of the nonlocal terms, we apply Corollary \ref{cor_LipEst}, 
to the same choice of functions  $ u(\cdot,x_2)$ and  $u(\cdot,y_2)$:
\begin{eqnarray*}&&
  \mathcal I_{x_1}[ x^\varepsilon, p_1^\varepsilon, u]-\mathcal I_{y_1}[y^\varepsilon, p_1^\varepsilon, u] \leq 
  - L | a_1^\varepsilon|^{(1-\beta)+\alpha(d_1+2-\beta)}\left\{ \varTheta(\varrho,\alpha,\mu^1) - o_{| a_1^\varepsilon|}(1)\right\} +O(1).
\end{eqnarray*}
At this point, we fix $\rho$ such that the constant $ \varTheta(\varrho,\alpha,\mu^1)$ is positive.
We then apply Lemma \ref{lemma_TrEst} in Appendix with $\tilde Z_1 = Z_1^{\frac{\bar \epsilon}{2}}$, where this time
$$
\bar\varepsilon =\frac{| a_1^\varepsilon|}{\varphi'(|a_1^\varepsilon|) }= 
\left( L| a_1^\varepsilon|^{-1} - L \rho (1+\alpha) |a_1^\varepsilon|^{\alpha-1}\right)^{-1}.
$$
{Indeed $\omega =  1 - \varphi''(| a_1^\varepsilon|) \bar \varepsilon  \in (1,2)$ for $\varepsilon$ sufficiently small.} 
Hence
\begin{eqnarray*}
\hbox{trace}(X_1^\varepsilon - Y_1^\varepsilon) &\leq &
 - \frac{8}{\bar \varepsilon}\frac{\omega - 1}{\omega + 1}  =  \frac{8\varphi''(| a_1^\varepsilon|)}{2 -\varphi''(| a_1^\varepsilon|)\bar\varepsilon}. 
\end{eqnarray*} 
Note that  in this case $\frac{\omega-1}{\omega+1}$ depends on $|a_1^\varepsilon|$. 
However there exists a positive constant $\bar \omega$ such that for $\varepsilon$ sufficiently small
\begin{eqnarray*}
\frac{8\varphi''(| a_1^\varepsilon|)}{2 -\varphi''(| a_1^\varepsilon|)\bar\varepsilon} 
& \leq & 8\bar \omega  \varphi''(| a_1^\varepsilon|).
\end{eqnarray*}
Hence, denoting by ${c}= \rho (1+\alpha)$, second order terms are bounded by
\begin{eqnarray*}
\hbox{trace}(X_1^\varepsilon - Y_1^\varepsilon) &\leq &- 8{c}\bar \omega \left( L\alpha | a_1^\varepsilon|^{\alpha-1}\right).
\end{eqnarray*}
We plug these estimates into the inequality for  $E_1$. Letting $\varepsilon$ go to zero and employing Lemma \ref{lemma_max_points} we arrive as before to
\begin{eqnarray*}
\limsup_{\varepsilon\rightarrow 0} E_1 (x_1^\varepsilon, y_1^\varepsilon, u) \leq 
{\Lambda_0} \ \mathcal E^1(|\bar a|) + {\Lambda_0}\ \mathcal  E^2(|\bar a|) +{O(1)} ,
\end{eqnarray*}
where denoting by $ {C(\mu^1)}= \varTheta(\varrho,\alpha,\mu^1)$ the terms  $\mathcal E^1$, $\mathcal E^2$ are given by 
\begin{eqnarray*}
\mathcal E^1(|\bar a|) & = & - L |\bar a|^{(1-\beta)+\alpha(d_1+2-\beta)}\left( {C(\mu^1)} - o_{|\bar a|}(1)\right)+
	      |\bar a|^{2\theta}\Big(L|\bar a|^{-1}\big(1 - {c}|\bar a|^{\alpha}\big) \Big)   \\
&&            + \ |\bar a|^\tau \Big(L\big( 1 - {c}|\bar a|^\alpha \big)\Big)^{\beta + \tau}\ \
	      {+{C_1}  \Big(L\big( 1 - {c}|\bar a|^\alpha \big)\Big)^\beta }   \\ \medskip
\mathcal E^2(|\bar a|) & = & - 8\ {c}\ \bar \omega \Big( L\alpha |\bar a|^{\alpha-1}\Big)	+  
	     {|\bar a|^{2\tilde\theta} }\Big(L|\bar a|^{-1}\big(1 - {c}|\bar a|^{\alpha}\big) \Big) \\
&&            + \ |\bar a|^\tau \Big(L\big( 1 - {c}|\bar a|^\alpha \big)\Big)^{2+\tau}\ +\
	      { {C_2}  \Big(L\big( 1 - {c}|\bar a|^\alpha \big)\Big)^2}. 
\end{eqnarray*}
Whenever $\alpha(d_1 + 3 -\beta) < 2\theta - 2 - \beta$ the second term in $\mathcal E^1$ behaves like $o\left(|\bar a|^{(1-\beta)+\alpha(d_1+2-\beta)}\right)$.
Taking $L |\bar a|^{(1-\beta)+\alpha(d_1+2-\beta)}$ as a common multiplier and using that $1 - {c} |\bar a|^\alpha\leq 1$  we have
\begin{eqnarray*}
\mathcal E^1(|\bar a|) 
& \leq &  
- L |\bar a|^{(1-\beta)+\alpha(d_1+2-\beta)}
\Big\{ 
 {C(\mu^1)} - o_{|\bar a|}(1)
\\ && \hspace{3.8cm} 
-|\bar a|^{- \alpha(d_1+2-\beta)} \Big(L|\bar a| - {c}L|\bar a|^{\alpha+1} \Big)^{\beta + \tau-1}   
\\ && \hspace{3.8cm} 
{-{C_1}  |\bar a|^{ - \alpha(d_1+2-\beta)} \Big(L|\bar a| - {c}L|\bar a|^{\alpha+1} \Big)^{\beta -1 }} 
\Big\}\\
& \leq &  
- L |\bar a|^{(1-\beta)+\alpha(d_1+2-\beta)}
\Big\{ 
 {C(\mu^1)} - o_{|\bar a|}(1)
\\ && \hspace{3.8cm} 
-2|\bar a|^{- \alpha(d_1+2-\beta)} \Big(\varphi(|\bar a|) \Big)^{\beta + \tau-1}   
\\ && \hspace{3.8cm} 
{-2{C_1}|\bar a|^{ - \alpha(d_1+2-\beta)} \Big(\varphi(|\bar a|) \Big)^{\beta -1 } }
\Big\}.  
\end{eqnarray*}
On the other hand, similar techniques give us an estimate for $\mathcal E^2$ :
\begin{eqnarray*}
\mathcal E^2(|\bar a|) 
& \leq &  
- L |\bar a|^{\alpha-1} 
\Big\{ 
 8 {c} \alpha \bar \omega - |\bar a|^{2\tilde\theta}|\bar a|^{-\alpha}
\\ && \hspace{2.7cm} 
- |\bar a|^{-\alpha} \Big(L|\bar a| - {c}L|\bar a|^{\alpha+1} \Big)^{1+\tau}  
\\ && \hspace{2.7cm} 
{ -{C_2}  |\bar a|^{-\alpha} \Big(L|\bar a| - {c}L|\bar a|^{\alpha+1} \Big) }
\Big\}\\
& \leq &  
- L |\bar a|^{\alpha-1} 
\Big\{ 
 8 {c} \alpha \bar \omega - |\bar a|^{2\tilde\theta}|\bar a|^{-\alpha}
\\ && \hspace{2.7cm} 
- 2|\bar a|^{-\alpha} \Big(\varphi(|\bar a|) \Big)^{1+\tau}   
\\ && \hspace{2.7cm} 
{ -2{C_2} |\bar a|^{-\alpha} \Big(\varphi(|\bar a|) \Big) }
\Big\}.
\end{eqnarray*}
When $\alpha$ is small enough we have  ${|\bar a|^{2\tilde\theta} }|\bar a|^{-\alpha} = o_{|\bar a|}(1).$
Then
\begin{eqnarray*}
\mathcal E^2(|\bar a|) 
& \leq &  
- L |\bar a|^{\alpha-1} 
\Big\{ 
 {C}  - o_{|\bar a|} (1)-
2|\bar a|^{-\alpha} \Big(\varphi(|\bar a|) \Big)^{1+\tau} 
{ -2{C_2} |\bar a|^{-\alpha} \Big(\varphi(|\bar a|) \Big) }
\Big\}.
\end{eqnarray*}
Since we have just seen that $u$ is H\"older continuous for any $\tilde\alpha \in(0,1)$, we have
$$
\varphi(|\bar a|) |\bar a|^{-\tilde\alpha} \rightarrow 0, \hbox{ as } L\rightarrow \infty. 
$$
Using this relation in the previous inequalities estimating $\mathcal E^1$ and $\mathcal E^2$ we get that, for $L$ large enough
\begin{eqnarray*}
\mathcal E^1(|\bar a|) &\leq &- L |\bar a|^{(1-\beta)+\alpha(d_1+2-\beta)} \Big({C(\mu^1)} - o_{|\bar a|}(1)\Big)\\
\mathcal E^2(|\bar a|) &\leq& - L |\bar a|^{\alpha -1} \Big({C} - o_{|\bar a|}(1)\Big).
\end{eqnarray*}
Hence (\ref{eq_estE1_Hdr}) holds and this further yields the desired contradiction.
\end{proof}

\subsection{Global Regularity}

It follows immediately from the previous results that as long as both nonlinearities $F_1$ and $F_2$ 
satisfy assumptions $(H1) - (H3)$, the solution is global Lipschitz or H\"older continuous. \smallskip

\begin{corollary}[\textbf{Global regularity for periodic, mixed PIDEs}] 
Let the nonlinearities $F_i$, $i=0,1,2$ be degenerate elliptic, continuous and periodic, $f$ continuous and periodic.
Assume the following:
\begin{itemize*}
 \item $F_0$  satisfies assumptions $(H0)$, $(H2)$ with $\tilde d = d$ and some constant \textcolor{black}{$\tilde \gamma>0$};
 \item $F_i$  with $i=1,2$ satisfy assumptions {$(H1) - (H3)$} 
	      with $\tilde d = d_i$, for some functions $\Lambda_i^1$, $\Lambda_i^2$ and 
	      some constants $ k_i\geq 0, \; \tau_i\in[0,1], \theta_i, \tilde\theta_i \in(0,1].$
\end{itemize*}
Let $\mu^0$, $\mu^i$, with $i=1,2$ be L\'evy measures on  $\R^d$, $\R^{d_i}$ 
respectively associated to the integro-differential operators 
$\mathcal I[x,u]$, $\mathcal{J}_{x_i}[x,u]$ 
and suppose the corresponding jump functions $j^i(x_i,z_i)$  satisfy assumptions $(J1)-(J5)$ 
for some constants $\beta_i$, $C_{\mu^i}$, $\tilde C_{\mu^i}$, with $\gamma=1$.
Then any {periodic continuous} viscosity solution $u$ of 
  \begin{eqnarray}\label{eq_stPIDEs}
  &&
  F_0(u(x),Du, D^2u, \mathcal I[x,u]) + \\ \nonumber && \hspace{1cm} 
  F_1(x_1,D_{x_1}u, D_{x_1x_1}^2u, \mathcal J_{x_1}[x,u]) +  
  F_2(x_2,D_{x_2}u, D_{x_2x_2}^2u, \mathcal J_{x_2}[x,u]) = f(x)
  \end{eqnarray}
\begin{itemize}
 \item[(a)] is Lipschitz continuous,  if $\beta_i>1$ and $k_i\leq\beta_i$ for $i=1,2$; 
 \item[(b)] is $C^{0,\alpha}$ continuous with $\alpha<\min(\frac{\beta_1 - k_1}{1-k_1},\frac{\beta_2-k_2}{1-k_2})$, if $\beta\leq1$ and $k_i<\beta_i$ for $i=1,2$.
\end{itemize}
The Lipschitz / H\"older constant depends on $||u||_\infty$, on the dimension  $d$ of the space and  on the constants associated to the L\'evy measures and on the constants required by the growth condition $(H1)$. 
\label{cor_Lip_globalReg}
\end{corollary}

At first glance, the fact that $(H1)$  and $(H3)$ must hold simultaneously seems to exclude a large class of nonlinear equations dealing with directional gradient or drift terms  such as  $|D_{x_i}u|^r$ or $|b(x_i)|D_{x_i}u|^{k+\tau}$, $r,k> 0$. Indeed, taking in the ellipticity growth condition $(H1)$ $l=l'$,  $p = \frac{x-y}{\varepsilon}$ and $\tilde \theta = \theta$ we get 
$$
F(y,\frac{x-y}{\varepsilon},Y,l) - F(x,\frac{x-y}{\varepsilon},X,l) \leq
\Lambda(x) \Big(\hbox{tr}(X-Y) +
\frac{|x-y|^{2\theta}}{\varepsilon} + \frac{|x-y|^{k+2\tau}}{\varepsilon^{k+\tau}} +\frac{|x-y|^{r}}{\varepsilon^{r}} \Big).
$$
Hence $(H3)$ would hold whenever $k=r=0$, $\theta = 1$.
In this case $(H1)$ and $(H3)$ could be joined together in assumption
\begin{itemize}
  \item [$(H)$] {
There exist two functions ${\Lambda_1},{\Lambda_2}:\R^{\tilde d}\rightarrow [0,\infty)$ such that ${\Lambda_1}(x)+{\Lambda_1}(x)\geq\Lambda_0>0$
and a modulus of continuity $\omega_{F}(r)\rightarrow 0 $, as $r\rightarrow 0$ 
such that  for any $x,y\in\R^{\tilde d}$, $p\in\R^{\tilde d}$, $l\leq l'$  and any $\varepsilon >0$
\begin{eqnarray*}&& 
F(y,p,Y,l')-F(x,p,X,l) \leq \\ && \hspace{2cm} 
{\Lambda_1}(x)(l-l') + 
{\Lambda_2}(x)\hbox{tr}(X-Y) +
\omega_{F}\left(|x-y|(1+|p|) + \frac{|x-y|^{2}}{\varepsilon}\right)
\end{eqnarray*}
if $X,Y\in\SS^{\tilde d}$ satisfy inequality
(\ref{eq_matrix_ineq})
with $ Z = I - \bar\omega\hat z\otimes \hat z$, for $z\in\R^{\tilde d}$ and $\bar \omega\geq 1$.}
\end{itemize}
{Nevertheless, one can argue under weaker growth assumptions, by a \emph{cut-off gradients} argument for equations of the type (\ref{eq_stPIDEs})

where $F_i$, for $i=1,2$ satisfy assumptions $(H1) - (H2)$ and $F_0$ satisfies $(H2)$ and $(H0)$ {with $\tilde \gamma >0$}.}

Roughly speaking, one should look at the approximated equation with  $|Du|$ replaced by $|Du| \wedge R$, for $R>0$ and remark that its solutions are Lipschitz continuous, with the Lipschitz norm independent of $R$, thus the solution of the original problem is also Lipschitz continuous. This is made precise by defining, for each $i=0,1,2$ the following functions
$$
F_i^R(\cdot,p,X,l) = \left\{ 
\begin{array}{ll}
 F_i (\cdot,p,X,l), & \hbox{ if } |p|\leq R \\
 F_i (\cdot,R \frac{p}{|p|},X,l), & \hbox{ if } |p|\geq R. 
\end{array}
\right.
$$
Consider then the approximated problem
\begin{eqnarray}\label{eq_stPIDEs_approx}
&&
F^R_0({u^R(x)},Du^R, D^2u^R, \mathcal I[x,u^R]) + \\ \nonumber && \hspace{1cm} 
F^R_1(x_1,D_{x_1}u^R, D_{x_1x_1}^2u^R, \mathcal J_{x_1}[x,u^R]) +  
F^R_2(x_2,D_{x_2}u^R, D_{x_2x_2}^2u^R, \mathcal J_{x_2}[x,u^R]) = f(x)
\end{eqnarray}
and remark that $(H3)$ holds. Thus the approximated problem (\ref{eq_stPIDEs_approx})
has a Lipschitz/H\"older viscosity solution, whose continuity constant  
depends on $||u^R||_\infty$ the constants required by the L\'evy measures 
and those appearing in the ellipticity growth assumption $(H1)$. 

Let  
$$
M:=|F_1(0,0,0,0)|+ ||F_1(x_1,0,0,0)||_\infty + ||F_2(x_2,0,0,0)||_\infty +\textcolor{black}{||f||_\infty}.
$$
Since $M (\tilde \gamma)^{-1}$ and $ -M (\tilde \gamma)^{-1}$ are respectively a
supersolution and a subsolution of the approximated problem (\ref{eq_stPIDEs_approx}), 
by a comparison result between sub and super-solutions we have due to $(H0)$
$$
||u^R||_\infty \leq \frac{M}{\tilde\gamma} .
$$
Therefore, the Lipschitz constant of $u^R$ is independent of $R$. 
Observing that for $R$ large enough the solution $u^R$ of the approximated problem is as well a solution of the original, we conclude.

\section{Examples and Discussion on Assumptions}\label{sec_Lip_examples}

In this section, we illustrate the partial and global regularity
results on several examples.  We start with two examples of classical
nonlinearities for which we deal with global regularity: a model
equation as in \cite{BCI:11:HdrNL} and the advection fractional
diffusion.  Then we present the partial and global regularity results
for pure mixed equations: first on the toy model and then on a general
nonlinearity dealing with mixed gradient terms.

\subsection{Classical Nonlinearities}

As already presented in the introduction, the Lipschitz regularity
result applies for equations that are strictly elliptic in a
generalized sense: at each point, the nonlinearity is \emph{either non
  degenerate in the second-order term,} or is \emph{nondegenerate in
  the nonlocal term}. More precisely, by Theorem
\ref{thm_part_Lip_NL} we extend the H\"older regularity result in
\cite{BCI:11:HdrNL} to Lipschitz regularity when the nonlocal exponent
$\beta> 1$.

\subsubsection{Model Equation}

A model equation for such nondegenerate equations is
\begin{equation}\label{eqn:example}
- \mathrm{tr} \, (A(x)D^2 u) - c(x)\mathcal{I} [x,u] + b(x)|Du|^{k} + |Du|^r = 0\quad\hbox{in  }\R^d\; ,
\end{equation}
where $A$ and $c$ are continuous functions, $b\in C^{0,\tau}(\R^d)$, with $0\leq\tau \leq 1$, $k,r\in(0,2+\tau)$.
$\mathcal{I} [x,u]$ is a non-local term of type \eqref{eq_NLterm} or \eqref{eq_LIterm} of exponent $\beta \in (0,2)$. 
In the following, we discuss the ellipticity growth assumption $(H1)$ and make precise the role of each term.

\begin{itemize}
 \item One has to assume that equation \eqref{eqn:example}  is strictly elliptic in the sense that	
\begin{equation}\label{eq_}
A(x) \geq \Lambda_1 (x) I \quad\hbox{and}\quad c(x) \geq \Lambda_2 (x) \quad\hbox{in  }\R^d\; 
\end{equation}
with $$\Lambda_1 (x) + \Lambda_2 (x) \ge \Lambda_0 >0.$$ 
Thus the equation may be degenerate in the local or the nonlocal term 
as for all $x\in \R^d$, $A(x)\geq 0$ and $c(x)\geq 0$. 
However, at each point either $A(x)$ is a positive definite matrix and the equation is strictly elliptic in the classical sense, 
or $c(x)>0$ and $\mathcal{I} [x,u]$ satisfies suitable nondegeneracy assumptions (that we discuss below)
and the equation is strictly elliptic with respect to the integro-differential term.

\item  \emph{$A=\sigma^T \sigma$ with $\sigma$ a bounded, uniformly continuous function}
 which maps $\R^d$ into the space of $N \times p$-matrices for some $p \le N$. It can be checked that
$$
-\left( \hbox{tr}(A(x)X) - \hbox{tr}(A(y)Y) \right) \leq d \frac{\omega_\sigma^2(|x-y|)}{\varepsilon}  
$$
for any $X,Y\in\SS^{d}$ satisfying inequality (\ref{eq_matrix_ineq}).

\item The nonlocal term can be writen as a general nonlocal operator
\begin{eqnarray*}
c(x)\mathcal{I} [x,u]
&  =  & c(x) \int_{\R^d} \big( u(x+z) - u(x) - Du(x)\cdot z1_B(z) \big) \mu_x(dz) \\
&  =  & \int_{\R^d} \big( u(x+z) - u(x) - Du(x)\cdot z1_B(z) \big) c(x)\mu_x(dz)
\end{eqnarray*}
where $\big(\mu_x\big)_{x}$ is a family of L\'evy measures, satisfying
assumptions $(M1) - (M3)$.  When \emph{$c:\R^d\to \R$ is
  $\gamma$-H\"older continuous} the results for general nonlocal
operators literally apply for the new family of operators associated
to the L\'evy measures $\tilde \mu_x = c(x)\mu_x$.\smallskip

For a L\'evy-It\^o type operator, the nonlocal term can be writen as 
\begin{eqnarray*}
c(x)\mathcal{I} [x,u]
& = &  c(x)\int_{\R^d} \big( u(x+j(x,z)) - u(x) - Du(x)\cdot j(x,z) 1_B(z)\big)\mu(dz)\\
& = & \int_{\R^d} \big( u(x+j(x,z)) - u(x) - Du(x)\cdot j(x,z) 1_B(z)\big) c(x)\mu(dz)
\end{eqnarray*}
where the jump function $j(x,z)$ satisfies assumptions $(J1) - (J5)$.
In this case, the results for general nonlocal operators do not apply
ad-literram!  Otherwise we could have considered L\'evy-It\^o
operators as a particular case of general integro-differential
operators.  However, when \emph{ $c$ is $\gamma$-H\"older continuous},
combining estimates arguments (see Section \S
\ref{sec_Lip_estimates}) used for L\'evy-It\^o operators with those
for general nonlocal operators, we arrive to the same conclusion.

\item \emph{$b:\R^d\to \R$ is a $\tau$-H\"older continuous function,
    or just a bounded continuous function.}  The growth conditions
  $k,r$ on the gradient are related to the regularity of coefficients
  of $b$.

When $\beta>1$, the solution is Lipschitz continuous for gradient
terms $b(x)|Du|^{k}$ with natural growth $k\leq\beta$ and $b$ bounded.
If in addition $b$ is $\tau$-H\"older continuous, then the solution
remains Lipschitz for gradient terms with growth
$k\leq\tau+\beta$. Similarly, the solution is Lipschitz for any term
gradient term $|Du|^r$ with $r\leq\beta$.
\end{itemize} 

\subsubsection{Advection Fractional Diffusion Equation}

Several recent papers deal with the regularity of solutions for the
advection fractional diffusion equation
$$
u_t + (-\Delta_x)^{\beta/2} u + b(x)\cdot Du = f.
$$
One distinguishes three cases, according to the order of fractional
diffusion.  The case $\beta<1$ is known as the supercritical case,
since the fractional diffusion is of lower order than the advection;
conversely, $\beta>1$ is the subcritical case. In between we have the
critical value $\beta=1$, when the drift and the diffusion are of the
same order. \smallskip

In the critical case, it was shown by Caffarelli and Vasseur
\cite{CV:11:RegDrifFr} by using De Giorgi's approach that the solution
is smooth for $L^2$ initial data, $f \equiv 0$, and divergence free
vector fields $b$ belonging to the BMO class. The key step is to prove
first that it is H\"older continuous.  Their motivation comes from the
quasi-geostrophic model in fluid mechanics. We mention that for smooth
periodic initial data, Kiselev, Nazarov and Volberg \cite{KNV:07:GWP}
proved that the solution of the quasi-geostrophic equation remains
smooth. \smallskip

Recently, Silvestre \cite{S:11:HdrAdvFrD} proved H\"older estimates
for solutions of this equation (and nonlinear versions of it) by
Harnack techniques. He also showed \cite{S:11:DiffAdvFrD} that when
$\beta\geq 1$ and the vector field $b$ is $C^{1-\beta+\tau}$, the
solution becomes $C^{1,\tau}$.  \smallskip

As we shall see in the following Section~\S\ref{sec_Lip_extensions},
our regularity results apply as well in the parabolic and/or
non-periodic setting. Hence for such an equation (and nonlinear
versions of it), we obtain that the solution is Lipschitz continuous in
the subcritical case $\beta>1$ with $b$ bounded; hence the fractional
diffusion is stronger than the advection and prescribes the regularity
of the solution.  In the supercritical case $\beta\leq 1$, the
solution is $\beta$ H\"older continuous whenever $b$ is
$C^{1-\beta+\tau}$, where $\tau>0$. \smallskip

\subsection{Mixed nonlinearities}

As discussed before, there is another interesting type of \emph{mixed ellipticity}: at each point, the nonlinearity is 
\emph{degenerate both in the second-order term, and in the nonlocal term}, 
but \emph{ the combination of the local and the nonlocal diffusions renders the nonlinearity uniformly elliptic}.
For this type of equations, partial regularity results apply first and then they are used to derive the global regularity.

\subsubsection[A Toy-Model for the mixed case]{A Toy-Model for the Mixed Case}

The simplest example of pure mixed equations is given by
\begin{equation*} 
-\Delta_{x_1}u + (-\Delta_{x_2})^{\beta/2} u = f(x_1,x_2) 
\end{equation*}
where $ (-\Delta_{x_2})^{\beta/2} u$ denotes the fractional Laplacian with respect to the $x_2$-variable
\begin{equation*}
(-\Delta_{x_2})^{\beta/2} u= 
-\int_{\R^{d_2}}\big(u(x_1,x_2+z_2)-u(x_1,x_2)-D_{x_2}u(x_1,x_2)\cdot z_2 1_B(z_2)\big)\frac{dz_2}{|z_2|^{d_2+\beta}}.
\end{equation*}
\begin{figure}[ht]
\centering
\includegraphics[width = 0.5\linewidth]{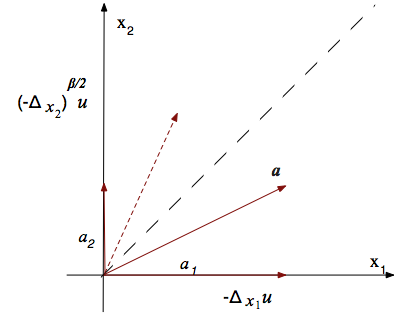}
\caption{\small Local diffusions occur only in $x_1$-directions and fractional diffusions in $x_2$-directions.}
\end{figure}\smallskip

It is clear that the equation is degenerate both with respect to the local and the nonlocal term, 
as both the Laplacian and the fractional Laplacian are incomplete. {Indeed, 
the directional classical Laplacian has all of the eigenvalues corresponding to the $x_2$ variable equal to zero, and 
therefore the nonlinearity $F$ is degenerate with respect to the second order term $D^2u$.} On the other hand, the degeneracy 
with respect to the nonlocal term comes from the fact that $$\mu(dz_2) = \frac{dz_2}{|z_2|^{d_2+\beta}}$$ could be viewed as the restriction 
of the fractional Laplacian to the subspace $\{z_1 = 0\}$
$$
\nu(dz) =  1_{\{z_1 = 0\}}(dz_1) \mu(dz_2).
$$
{Therefore, for a cone whose direction $a$ is orthogonal to the $x_2$-direction,} we have  
$$ 
\int_{\mathcal C^d_{\eta, \delta}}|z|^2 \nu(dz) = 
\int_{\mathcal C^{d_2}_{\eta, \delta}}|z_2|^2 \mu(dz_2)  = 0 
$$
where $ \mathcal C^{d_2}_{\eta, \delta} =  \{z_2\in B^{d_2}_\delta; (1-\eta)|z_2||a|\leq|a_2\cdot z_2|\}$. 
Thus, $(M2)$ and $(J2)$ fail and the H\"older regularity results of \cite{BCI:11:HdrNL} do not apply.

{Instead, the partial regularity results of Theorem \ref{thm_part_Lip_NL} hold: the solution is Lipschitz continuous with respect to the $x_2$ variable when $\beta\geq1$ and H\"older continuous when $\beta<1$, and Lipschitz continuous with respect to the $x_1$ variable.}

\begin{remark}
If we try to argue directly in $\R^d$ and apply the regularity result as if 
we had only one nonlinearity defined on the whole space, then the best result
 we can get is  H\"older regularity of the solution, except for the diagonal direction, i.e. 
for all $\varepsilon \in(0,1]$ the following holds for all $\alpha \in(0,\varepsilon)$ 
$$
u(x) - u(y) \leq C |x-y|^\alpha,  \forall x,y,\in\R^d \hbox{ s.t. } \max_{i=1,2}\frac{|x_i - y_i|}{|x-y|}\geq\sqrt{\frac{1}{2-\varepsilon}}.
$$
In addition, the further we go from the  diagonal, the better the regularity of the solution is. 
\end{remark}

\noindent Let us check that when the gradient direction is the diagonal between  $x_1$ and
$x_2$ it is not possible to retrieve H\"older continuity directly. For this purpose, consider two matrices  $X,Y$ 
satisfying inequality (\ref{eq_matrix_ineq}), with $Z = D\phi(a)$, where $\phi(z) = L|z|^\alpha$. 
Let $a = (a_1,a_2) = \bar x - \bar y$ be the gradient direction.
The matrix inequality can be rewritten as follows
\begin{equation}\label{eq:matrix_ineq_vect}
Xz\cdot z - Yz'\cdot z' \leq D^2\phi(a) (z-z')\cdot (z-z'). 
\end{equation}
\textit{Estimate of the diffusion terms.}
Applying (\ref{eq:matrix_ineq_vect}) to $z = - z' = e_1 = \frac{1}{|a_1|}(a_1,0)$ 
and to $z = z' = (e,0)$ for any unit vector $e$ orthogonal to $e_1$ we obtain
\begin{equation*}
\hbox{tr} \big(X_1 - Y_1\big )  \leq 4 D^2\phi(a)e_1\cdot e_1.
\end{equation*}
Therefore taking into account the expression for 
 $D^2\phi(a) = \varphi'(|a|)\frac{1}{|a|}(I - \hat a\otimes \hat a) + \varphi''(|a|) \hat a\otimes \hat a,$ we get
$$
\hbox{tr} \big(X_1 - Y_1\big ) \leq 4 \frac{\varphi'(|a|)}{|a|} (1-\frac{ |a_1|^2}{|a|^2}) + 4 \varphi''(|a|) \frac{|a_1|^2}{|a|^2}.
$$
Using that  $\phi (z) = L |z|^\alpha$ with $\alpha\in(0,\varepsilon)$  and $L>0$ the previous inequality reads
\begin{equation}\label{eq:estimate_trace_diffusion}
\hbox{tr} \big(X_1 - Y_1\big ) \leq 4L\alpha |a|^{\alpha-2}(1 + (\alpha-2)\frac{|a_1|^2}{|a|^2}).
\end{equation}
This expression is negative only if $$\frac{|a_1|^2}{|a|^2}>\frac{1}{2-\varepsilon}.$$ 
Hence, when the gradient direction is "closer`` to the $x_1$-axis, 
the classical diffusion gains and the regularity is driven by the classical Laplacian. \smallskip

\noindent
\textit{Estimate of the nonlocal terms.} 
As already made precise, the ellipticity of the equation comes in this case from the nondegeneracy 
assumption $(M2)$ with respect to the L\'evy measures. Accordingly, the estimate that renders the nonlocal difference negative
comes from the evaluation on the cone in the gradient direction. In view of $(M2)$ we have by rough approximations 
(see Proposition \ref{prop_concave_est_NL} and its Corollaries) that
for $e_2 = \frac{1}{|a_2|}(0,a_2)$
\begin{eqnarray*}
\mathcal I_{x_2} [\bar x,u] - \mathcal I_{x_2}[\bar y,u]& \leq& 
\int_{\mathcal C_{\eta,\delta}} \sup_{|s|<1} \big(D_{a_2a_2}^2\phi(a+s(0,z_2)) z_2 \cdot z_2 \big)\mu(dz_2) + c L\alpha |a|^{\alpha-2}\\
& =  &  \int_{\mathcal C_{\eta,\delta}} \sup_{|s|<1}\big((1-\tilde\eta^2)\frac{\varphi'(|a+s(0,z_2)|)}{|a+s(0,z_2)|} 
+ \tilde\eta^2 \varphi''(|a+s(0,z_2)|)\big)|z_2|^2 \mu(dz)\\&&+ c L\alpha |a|^{\alpha-2}\\
& \le & C \Big((1-\tilde\eta^2)\frac{\varphi'(|a|)}{|a|} (1-\frac{ |a_2|^2}{|a|^2}) 
+ \tilde\eta^2\varphi''(|a|) \frac{|a_2|^2}{|a|^2} \Big) + c L\alpha |a|^{\alpha-2}\\
& =  & C L\alpha |a|^{\alpha-2}\left(1 + \tilde\eta^2(\alpha-2)\frac{|a_2|^2}{|a|^2}\right)+ c L\alpha |a|^{\alpha-2}.
\end{eqnarray*}
This expression is negative only if $$\frac{|a_1|^2}{|a|^2}>\frac{1}{\tilde\eta^2(2-\varepsilon)}.$$ 
Similarly, when the gradient direction is "closer`` to the $x_2$-axis, 
the fractional diffusion gains and the regularity is driven by the (directional) fractional Laplacian.

\subsubsection{Mixed Integro-Differential Equations with First-Order Terms}

Partial and global, H\"older and Lipschitz regularity results apply for a general class of mixed integro-differential equations.
As pointed out in the previous theorems, the three nonlinearities must satisfy suitable strict ellipticity and growth conditions. 
The typical examples one can solve under those assumptions can be summed up by the following equation
\begin{eqnarray*}
{-a_1(x_1)\Delta_{x_1}u}     {-a_2(x_2)\mathcal I_{x_2}[x,u]} -  \mathcal I[x,u] 
{+b_1(x_1)|D_{x_1}u_1|^{k_1}} {+b_2(x_2)|D_{x_2}u|^{k_2}} + |Du|^n + c u= f(x)
\end{eqnarray*} 
where for $i=1,2$ $a_i(x_i)\geq0$ and $a_i\in C^{0,\gamma}(\R^{d_i})$, $b_i\in C^{0,\tau}(\R^{d_i})$ 
with $0\leq\tau \leq 1$, $k_i\in(0,2+\tau)$, $n\geq 0$ and $c>0$.
We have thus considered
\begin{eqnarray*}
F_0(u(x),Du, D^2u, \mathcal I[x,u]) & = &  -  \mathcal I[x,u] + |Du|^n + c u \\
F_1(x_1,D_{x_1}u, D_{x_1x_1}^2u, \mathcal J_{x_1}[x,u])  & = &  -a_1(x_1)\Delta_{x_1}u + b_1(x_1)|D_{x_1}u_1|^{k_1} \\
F_2(x_2,D_{x_2}u, D_{x_2x_2}^2u, \mathcal J_{x_2}[x,u])  & = & -a_2(x_2)\mathcal I_{x_2}[x,u]+b_2(x_2)|D_{x_2}u|^{k_2}.
\end{eqnarray*}
Let us have a look at each of these terms and see the assumptions they have to satisfy, in order
to ensure partial or global regularity of solutions. To fix ideas, suppose the nonlocal term $\mathcal I_{x_2}[x,u]$ 
is an integro-differential operator of fractional exponent $\beta\in(0,2)$.\smallskip

In both situations, the nonlocal term $\mathcal I[x,u]$
can either be a general nonlocal operator associated to some L\'evy measures $\mu^0$ 
or a L\'evy-It\^o operator. We emphasize the fact that the associated L\'evy measure has no $x$-dependency. This explains
as well the lack of any coefficient $a_0(x)$ in front of the nonlocal term $\mathcal I[x,u]$. 
The gradient term $|Du|^n$ is allowed to have any possible growth $n\geq 0$. \smallskip

As far as we are interested in partial regularity results, the constant $c$ may be any real number, since we just need $cu$ to be bounded.
Yet, when combining the partial regularity results to obtain global regularity, 
$F_1$ and $F_2$ are submitted to rather restrictive assumptions, due to the uniqueness requirements. 
Thus, when $b_1$ and $b_2$ depend explicitly on $x_1$, respectively $x_2$ the 
corresponding gradient terms are restrained to sublinear growth.
To turn around this difficulty and obtain regularity of solutions in superlinear cases,
one can argue by approximation, truncating the gradient terms and using Corollary \ref{cor_Lip_globalReg} 
for obtaining uniform gradient bounds. To perform this program, $c$ must be positive: $c>0$.\smallskip

We first discuss the \emph{partial regularity} of the solution with respect to each of its variables. 
To this end, we need classical regularity assumptions in one set of variables, 
and uniqueness type  assumptions in the other variables. \smallskip

\emph {Partial regularity in $x_2$-variable} requires ellipticity of the equation in $x_2$ direction: 
$$
 \forall x_1\in\R^{d_1}, x_2\in\R^{d_2} \;\;\; a_1(x_1) \geq 0 \hbox{ and } a_2(x_2)> 0.
$$
To ensure the uniqueness argument in $x_1$-variable, we must take $a_1 (x) =\sigma_1(x)^2$ with $\sigma_1$ a Lipschitz continuous function.
The nonlocal term $\mathcal I_{x_2}[x,u]$ is either a general integro-differential operator  or a L\'evy-It\^o operator.

When $\beta>1$, the solution is Lipschitz continuous in the $x_2$ variable
for directional gradient terms $b_2(x_2)|D_{x_2}u|^{k_2}$ having a natural growth $k_2\leq\beta$ if $b_2$ is bounded \emph{and} 
directional gradient terms $b_1(x_1)|D_{x_1}u|^{k_1}$ with linear growth $k_1 = 1$ if $b_1$ is Lipschitz (or sublinear growth $k_1<1$
if $b_1\in C^{0,k_1}$. 
If in addition $b_2$ is $\tau$-H\"older continuous, 
then the solution remains Lipschitz for  gradient terms up to growth $k_2\leq\tau+\beta$. 
When $\beta\leq1$, the solution is $\alpha$-H\"older continuous for any $\alpha<\frac{\beta-k_2}{1-k_2}$.
\smallskip

\emph{Partial regularity in $x_1$-variable} requires nondegeneracy of the equation in $x_1$ direction 
$$
a_1(x_1) > 0,\; \forall x_1\in\R^{d_1}.
$$
In this case, in the $x_2$ variable, we can only deal  with nonlocal operators of L\'evy-It\^o type
$
\mathcal I_{x_2}[x,u] = \mathcal J_{x_2}[x,u],
$ for which the jump function is Lipschitz continuous and satisfies the structural conditions $(J1)$, $(J4)$ and $(J5)$.
The uniqueness constraint with respect to $x_2$ does not allow any $x_2$-dependence of the 
L\'evy-measure associated to the nonlocal term, and hence $a_2(x_2)$ should be a constant function.

Then the solution is Lipschitz  in the $x_1$ variable, for directional gradient terms $b_1(x_1)|D_{x_1}u|^{k_1}$ 
having a natural growth $k_1\leq 2+\tau$ with $b_1\in C^{0,\tau}(\R^{d_1})$, $0\leq\tau\leq 1$.
Once again, the uniqueness hypothesis forces directional gradient terms $b_2(x_2)|D_{x_2}u|^{k_2}$ 
to have growth  $k_2 = 1$ and $b_2$ is Lipschitz continuous.  
\smallskip

\emph{Global regularity} holds under slightly weaker assumptions than the partial regularity. 
It follows by interchanging the roles of $x_1$ and $x_2$. 
Accordingly, the equation must be strongly elliptic both in the local and nonlocal term
$$
a_1(x_1) > 0 \; \hbox{ and }\; a_2(x_2) > 0 \; \forall x_1\in\R^{d_1}, x_2\in\R^{d_2}.
$$
The nonlocal term $\mathcal I_{x_2}[x,u]$ is necessarily a L\'evy-It\^o operator, satisfying the nondegeneracy assumption $(J2)$, as well as 
the rest of structural conditions $(J1) - (J5)$. In addition
$$a_1(x_1) = \sigma_1(x_1)^2>0$$ with $\sigma_1$ Lipschitz continuous and $a_2(x)\equiv a_2 > 0$ constant function.

Joining the partial Lipschitz regularity results, we get Lipschitz continuity of the solution
whenever $b_1$ and  $b_2$ are Lipschitz continuous 
for linear, directional gradient terms  $b_1(x_1)|D_{x_1}u|$ and  $b_2(x_2)|D_{x_2}u|$. 
The linear growth is constraint by the uniqueness argument.

However, looking at the approximated equations with  $|Du|$ replaced by $|Du| \wedge R$, 
for $R>0$ and noting that the solutions are Lipschitz continuous, with the Lipschitz norm independent of $R$ when $c>0$, 
we obtain Lipschitz continuous viscosity solutions for general equations, dealing with 
gradient terms of growth $k_1\leq 2, k_2\leq\tau+\beta$, when $b_2\in C^{0,\tau}(\R^{d_2})$.  
Similarly, we get $\alpha$-H\"older continuous solutions, for any $\alpha<\frac{\beta-k_2}{1-k_2}\leq 1$.

\section{Extensions}\label{sec_Lip_extensions}

\subsection{Non-periodic Setting}

\begin{theorem}
Let $f$ be continuous, the nonlinearities $F_i$, $i=0,1,2$ be degenerate elliptic, continuous,
such that $F_0$ satisfies $(H0)$ {with $\tilde \gamma >0$}
and $(H2)$,  and that both $F_i$, for $i=1,2$ satisfy assumptions $(H2)$ and $(H1')$, 
with $\tilde d = d_i$, for some functions $\Lambda_i^1$, $\Lambda_i^2$ and some constants 
$ k_i\geq 0, \; \tau_i, \theta_i, \tilde\theta_i \in(0,1]$, where
\begin{itemize}
\item [$(H1')$]
There exist two functions $\Lambda^1,\Lambda^2:\R^{\tilde d}\rightarrow [0,\infty)$ such that $\Lambda^1(x)+\Lambda^1(x)\geq\Lambda^0>0$
and for each $0<R<\infty$ there exist some constants $k\geq0$, $\tau,\theta, \tilde\theta \in(0,1]$
such that for any $x,y\in\R^{\tilde d}$, $p,q\in\R^{\tilde d}$, \textcolor{black}{$|q|<R$}, $l\leq l'$  and any $\varepsilon >0$
\begin{eqnarray*}
F(y,p,Y,l') & - &F(x,p{+q},X,l) \\ 
& \leq & \Lambda_1(x)\left((l-l') + \frac{|x-y|^{2\theta}}{\varepsilon}+ |x-y|^\tau|p|^{k+\tau} + C^1|p|^{k} \right)\\ 
&  +   & \Lambda_2(x)\left( \hbox{tr}(X-Y) + \frac{|x-y|^{2\tilde\theta}}{\varepsilon} + |x-y|^\tau|p|^{2+\tau}+C^2|p|^2\right) 
\textcolor{black}{+O(K,R)}
\end{eqnarray*}
if $X,Y\in\SS^{\tilde d}$ satisfy, inequality 
\begin{equation*}
- \frac{1}{\varepsilon}
\begin{bmatrix} 
I &  0 \\
0 & I 
\end{bmatrix} 
\leq
\begin{bmatrix} 
X &  0 \\
0 & -Y 
\end{bmatrix} 
\leq \frac{1}{\varepsilon}
\begin{bmatrix} 
 Z & -Z \\
-Z &  Z 
\end{bmatrix} 
\textcolor{black}{
+ K\begin{bmatrix} 
 I &  0 \\
-0 &  0 
\end{bmatrix}},
\end{equation*}
for some $Z = I - \omega \hat a \otimes \hat a$, with $\hat a\in\R^d$ a unit vector, and $\omega\in(1,2)$.
\end{itemize}
Let $\mu^0$, $\mu^i$, with $i=1,2$ and $j^i(x_i,z_i)$  satisfy assumptions $(J1)-(J5)$ for some constants $\beta_i$, $C_{\mu^i}$, $\tilde C_{\mu^i}$, with $\gamma=1$ in $(J3)$.
Then any \textcolor{black}{bounded continuous} viscosity solution $u$ of  (\ref{eq_stPIDEs}) is 
\begin{itemize}
 \item [(a)]\textcolor{black}{locally} Lipschitz continuous,  if $\beta_i>1$ and $k_i\leq\beta_i$ for $i=1,2$, and
 \item [(b)]\textcolor{black}{locally} $C^{0,\alpha}$ continuous with $\alpha<\min(\frac{\beta_1 - k_1}{1-k_1},\frac{\beta_2-k_2}{1-k_2})$, if $\beta\leq1$ and $k_i<\beta_i$ for $i=1,2$.
\end{itemize}
The Lipschitz/H\"older constant depends on $||u||_\infty$, on the dimension $d$ of the space and on the constants associated to the L\'evy measures and on the constants required by the growth condition $(H1)$. 
\end{theorem}

\begin{proof}[Sketch of the proof]
The fact that the solution is not periodic anymore, requires a localization term when measuring the shift of the solution. 
Thus, in order to prove the local continuity of the solution, either if it refers to H\"older or Lipschitz, we need to show that for each $x^0$ in the domain, there exists a constant $K$, depending on $x^0$, such that for a proper choice of $\alpha$ (both in the H\"older in the Lipschitz case) there exists a constant $L$, depending on $x^0$, large enough such that the auxiliary function	
$$ 
\psi(x_1,y_1,x_2) = u(x_1,x_2) - u(y_1,x_2) - L\varphi(|x_1-y_1|) - \frac K2|(x_1,x_2)-(x^0_1,x^0_2)|^2 
$$ 
attains a nonpositive maximum.
The proof is technically  the same, except that here there will be an additional contribution in the estimate of the nonlocal terms,  
coming from the localization term. The point is to show that this contribution is of order $O(K)$.

\end{proof}

\subsection{Parabolic Integro-Differential Equations}

The techniques previously developed apply literally to parabolic integro-differential equations.
\begin{corollary}
Let $f$, the nonlinearities $F_i$ and the jump functions  $j^i(x_i,z_i)$  satisfy the assumptions of Corollary \ref{cor_Lip_globalReg}.
If, for some $T>0$, $u:[0,T)\times \R^d\rightarrow\R$ is a $x-periodic$, continuous viscosity solution of 
\begin{eqnarray}
&&
\textcolor{black}{u_t +} F_0(u(x),Du, D^2u, \mathcal I[x,u]) + F_1(x_1,D_{x_1}u, D_{x_1x_1}^2u, \mathcal I_{x_1}[x,u]) +\\ \nonumber && \hspace{1cm} 
F_2(x_2,D_{x_2}u, D_{x_2x_2}^2u, \mathcal I_{x_2}[x,u]) = f(x) \quad\hbox{in  }(0,T)\times \R^d
\end{eqnarray}
\begin{itemize}
 \item[(a)] If $\beta_i>1$, $k_i\leq\beta_i$ for $i=1,2$ and if $u_0\in Lip(\R^d)$, then $u$ is Lipschitz continuous with respect to $x$ on $[0,T]$.
 \item[(b)] If $\beta\leq1$, $k_i<\beta_i$ for $i=1,2$ and if $u_0\in C^{0,\alpha}(\R^d)$, then $u$ is $C^{0,\alpha}$ with respect to $x$ on $[0,T]$, with $\alpha<\min(\frac{\beta_1 - k_1}{1-k_1},\frac{\beta_2-k_2}{1-k_2})$, .
\end{itemize}
The Lipschitz / H\"older constant depends on $||u||_\infty$, on the dimension $d$ of the space  
and on the constants associated to the L\'evy measures and on the constants required by the growth condition $(H1)$. 
\end{corollary}

\begin{proof}[Sketch of proof]
 
The key difference with the previous proof consists in considering the space-time auxiliary function
$$
\psi(t,x_1,y_1,x_2) = u(t,x_1,x_2) - u(t,y_1,x_2) - \phi(x_1-y_1)
$$
and show that 
$
\max_{t,x_1,x_2,y_2}\psi(t,x_1,y_1,x_2) <0.
$
By small space-time perturbations
$$
\psi_{\varepsilon,\varsigma}(x,y,s,t) = u(t,x_1,x_2) - u(s,y_1,y_2) - \phi(x_1-y_1)-\frac{|x_2-y_2|^2}{\varepsilon^2} - \frac{(t-s)^2}{\varsigma^2},
$$
this leads to considering in the nonlocal Jensen-Ishii's lemma the parabolic sub and superjets
\begin{eqnarray*}
(r^{\varepsilon,\varsigma}, p^{\varepsilon,\varsigma},X^{\varepsilon,\varsigma})\in \mathcal D_p^{2,+}(u(x^{\varepsilon,\varsigma}))\\
(r^{\varepsilon,\varsigma}, p^{\varepsilon,\varsigma},Y^{\varepsilon,\varsigma})\in \mathcal D_p^{2,-}(u(y^{\varepsilon,\varsigma}))
\end{eqnarray*}
with $r^{\varepsilon,\varsigma} = 2\frac{t-s}{\varsigma^2}$. Writing down the viscosity inequalities, note that the $r^{\varepsilon,\varsigma}$ is the common term corresponding to the first order time-derivative, and hence it vanishes by subtraction. Therefore, when passing to the limits in inequality (\ref{eq_visc_ineq_estimate}), we can first let $\varsigma$ go to zero. The rest of the proof is literally the same.
\end{proof}

\subsection{Bellman-Isaacs Equations}

These results can be extended to fully nonlinear equations, 
that arise naturally in stochastic control problems for jump-diffusion processes. 
The following Bellman-Isaacs type equation arises
$$
\sup_{\gamma\in\Gamma} \inf_{\delta\in \Delta}\Big( F^{\gamma,\delta}_0(..., \mathcal J^{\gamma,\delta}[x,u]) + 
F^{\gamma,\delta}_1(...,\mathcal J^{\gamma,\delta}_{x_1}[x,u]) + 
F^{\gamma,\delta}_2(...,\mathcal J^{\gamma,\delta}_{x_2}[x,u]) - f^{\gamma,\delta}(x)\Big) = 0
$$
where $\mathcal J^{\gamma,\delta} [x,u]$ is a family of L\'evy-It\^o operators
associated with a common L\'evy measure $\mu^0$ and  a family of jump functions $j_0^{\gamma,\delta}(x,z)$, respectively
$\mathcal J_{x_i}^{\gamma,\delta} [x,u]$ are  families of L\'evy-It\^o operators
associated with the L\'evy measures $\mu^i$ and  the families of jump functions $j_i^{\gamma,\delta}(x_i,z)$, for $i=1,2$.\smallskip

A typical (and practical) example is
\begin{eqnarray*}
F_0^{\gamma,\delta} &  =  & c u - \frac12 \hbox{tr}(A^{\gamma,\delta}  (x) D^2 u)  - \mathcal J^{\gamma,\delta} [x,u] 
- b^{\gamma,\delta}(x) \cdot D u \\
F_i^{\gamma,\delta} &  =  & \hspace{0.5cm} - \frac12 \hbox{tr}(a^{\gamma,\delta} _i(x_i) D_{x_ix_i}^2 u)  - \mathcal J_{x_i}^{\gamma,\delta} [x,u] 
- b_i^{\gamma,\delta}(x) \cdot D_{x_i} u.
\end{eqnarray*}
Similar techniques to the previous ones yield the H\"older and Lipschitz continuity of solutions of Bellman-Isaacs equations, provided that
the structure condition $(H1)$ is uniformly satisfied by $F^{\gamma,\delta}_i$, for $i=1,2$, as well as the assumptions  $(J1)-(J5)$ by the 
family of jump functions $j_i^{\gamma,\delta}(x_i,z)$. In occurrence, the constants and functions appearing therein must be independent of 
$\gamma$ and  $\delta$. For the above example, it is sufficient that 
$A^{\gamma,\delta} (x), a^{\gamma,\delta} _i(x), b^{\gamma,\delta} _i(x), f^{\gamma,\delta}(x) $ are bounded in $W^{1,\infty}$, uniformly in
$\gamma$ and $\delta$. \smallskip

The proof is based on the classical inequality
{
\begin{eqnarray*}
\sup_{\gamma}\inf_\delta \left(F^{\gamma,\delta}(..., \mathcal J^{\gamma,\delta}[x,u])\right) 
& - & \sup_{\gamma}\inf_\delta \left(F^{\gamma,\delta}(..., \mathcal J^{\gamma,\delta}[y,u])\right)\\
& \leq & \sup_{\gamma,\delta} \left(F^{\gamma,\delta}(..., \mathcal J^{\gamma,\delta}[x,u]) - F^{\gamma,\delta}(..., \mathcal J^{\gamma,\delta}[y,u])\right).
\end{eqnarray*}}

\subsection{Multiple Nonlinearities}

The problem can be easily generalized to multiple nonlinearities
\begin{eqnarray}
&&
F_0(u(x),Du, D^2u, \mathcal I[x,u]) + \sum_{i\in I}F_i(x_i,D_{x_i}u, D_{x_ix_i}^2u, \mathcal J_{x_i}[x,u]) = f(x).
\end{eqnarray}
The proof can be reduced to the previous one, by grouping all the variables for which we employ uniqueness type arguments.

\section{Estimates for Integro - Differential Operators}\label{sec_Lip_estimates}

All these results are based on a series of estimates for the nonlocal terms, that  we make precise in the following. 
They are similar to those in \cite{BCI:11:HdrNL}. As we have seen, the proof of the Lipschitz regularity of solutions uses H\"older continuity of solutions 
for small orders $\alpha\in(0,\frac{1}{d+1})$, where $d$ is the dimension of the space.
For this reason, the estimates below are first given in a general form, such that they can be used for both regularity proofs.
We then state as corollaries their precise form for Lipschitz and H\"older case.

\subsection{General Nonlocal Operators}

We first give some estimates for general nonlocal operators 
$$
\mathcal I[x,u] = \int_{\R^d}\left( u(x+z) - u(x) - Du(x)\cdot z 1_B\right)\mu_x(dz).
$$
We begin with a general result on concave estimates for these integro-differential operators, under quite general assumptions.
We then derive finer estimates in the particular case of Lipschitz and H\"older control functions. 
However, these special forms will hold for family of L\'evy measures $(\mu_x)_x$ which satisfy some additional assumptions.

\begin{prop}[\bf Concave estimates - general nonlocal operators]
Assume condition $(M1)$ holds.
Let $u,v$ be two {bounded} functions and $\varphi:[0,\infty)\rightarrow\R$ 
be a smooth increasing concave function. Define
$$
\psi(x,y) = u(x) - v(y) - \varphi(|x-y|)
$$
and assume the maximum of $\psi$ is {positive} and reached at $(\bar x,\bar y)$, 
with $\bar x \neq \bar y$. Let 
$$a=\bar x-\bar y, \;\;\;  \hat a = a/|a|, \;\; \; p=\varphi'(|a|) \hat a.$$
Then the following holds
\begin{eqnarray*}
\mathcal  I[\bar x,p,u] &-&  \mathcal I[\bar y,p,v] 
 \; \leq \; 4\tilde C_\mu \max(||u||_\infty, ||v||_\infty) 
\\&&
+\frac12 \int_{\mathcal C_{\eta,\delta}(a)}\sup_{|s|\leq1}\left((1-\tilde\eta^2)\frac{\varphi'(|a+sz|)}{|a+sz|} 
+ \tilde\eta^2\varphi''(|a+sz|)\right) |z|^2 \left(\mu_{\bar x}+\mu_{\bar y}\right)(dz)
\\ &&
+2\varphi'(|a|) \int_{B\setminus B_\delta}|z|\left|\mu_{\bar x}-\mu_{\bar y}\right|(dz)  +
\int_{B_\delta\setminus \mathcal C_{\eta,\delta}(a)}\sup_{|s|\leq1}\frac{\varphi'(|a+sz|)}{|a+sz|}|z|^2\left|\mu_{\bar x}-\mu_{\bar y}\right|(dz),
\end{eqnarray*}
where $$\mathcal C_{\eta,\delta} (a) = \left\{z\in B_\delta;(1-\eta)|z||a|\leq|a\cdot z|\right\}$$ and 
$\delta = |a|\delta_0 >0, \ \tilde \eta = \frac{1- \eta-\delta_0}{1+\delta_0} >0$
with $\delta_0\in(0,1)$, $\eta\in(0,1)$ small enough.
\label{prop_concave_est_NL}
\end{prop}

\begin{remark}
The aperture of the cone is given by $\eta$ and changes according to $|a|$. 
In order to ensure Lipschitz continuity of solutions, $\eta$ must be chosen to behave like a power of $|a|$, 
i.e. $\eta\sim |a|^\alpha$, and thus is diminishing as the modulus of the gradient approaches zero: 
$\lim_{|a|\rightarrow 0} \eta(|a|)= 0$. Remark that as $|a|\rightarrow 0$, $\mathcal C_{\eta,\delta}(a)$ 
degenerates to the line whose direction is given by the gradient. 
This will be made precise when proving Corollary \ref{cor_HoeEst} below.
\end{remark}

\begin{corollary}[\bf Lipschitz estimates] 
Let $(M1) - (M3)$ hold, {with $\beta>1$}.
Under the assumptions of Proposition \ref{prop_concave_est_NL}  with 
\begin{equation*}
 \varphi(t) = \left\{ 
 \begin{array}{ll}
  L\left ( t - \varrho t^{1+\alpha}\right), & t\in [0,{t_0}]\\
  \varphi({t_0}), &t>{t_0}
 \end{array}\right.
\end{equation*}
where $\alpha \in\left(0,\min(\frac{\gamma}{d+1},\frac{\beta-1}{d+2-\beta})\right)$, 
$\varrho$ is a constant such that $ \varrho \alpha 2^{\alpha -1} > 1$,
${t_0} = \max_t (t - \varrho t^{1+\alpha})=\sqrt[\alpha]{\frac{1}{\rho(1+\alpha)}}$ 
and  $L > {\frac{(||u||_\infty + ||v||_\infty)(\alpha+1)}{{t_0}\alpha},}$
the following holds:
{
there exists a positive constant $C = C(\mu)$ such that for 
$\Theta(\varrho,\alpha,\mu) = C\left(\rho\alpha 2^{\alpha-1}-1\right)$ 
}
we have
\begin{eqnarray*}&&
 \mathcal  I[\bar x,p,u] -  \mathcal I[\bar y,p,v] \leq 
 - L |a|^{(1-\beta)+\alpha(d+2-\beta)}\left\{ \varTheta(\varrho,\alpha,\mu)
 - o_{|a|}(1)\right\}+O(\tilde C_\mu).
\end{eqnarray*}\label{cor_LipEst}
\end{corollary}

\begin{corollary}[\bf H\"older estimates]
Let $(M1)-(M3)$ hold, {with $\beta\in(0,2)$.}
Under the assumptions of Proposition \ref{prop_concave_est_NL}  with 
\begin{equation*}
 \varphi(t) = \left\{ 
 \begin{array}{ll}
  Lt^\alpha, & t\in [0,{t_0}]\\
  \varphi({t_0}), &t>{t_0}
 \end{array}\right.
\end{equation*}
where $\alpha \in(0,\min(\beta,1))$, ${{t_0}>0}$, and 
$L>{\frac{||u||_\infty + ||v||_\infty}{{t_0}^\alpha},}$ 
the following holds: 
there exists a positive constant {$ {C(\mu)} >0$} such that 
\begin{eqnarray*}&&
 \mathcal  I[\bar x,p,u] -  \mathcal I[\bar y,p,v] \leq 
 - L |a|^{\alpha -\beta}\left\{ {{\alpha C(\mu)}} 
 - o_{|a|}(1)\right\}+O(\tilde C_\mu).
\end{eqnarray*}\label{cor_HoeEst}
\end{corollary}

\begin{proof}[Proof of Proposition \ref{prop_concave_est_NL}]
We split the domain of integration into three pieces 
and take the integrals on each of these domains. Namely we part the ball $B_\delta$ of radius  $\delta$ into the subset 
$
\mathcal C_{\eta,\delta} (a)
$
with  $\eta=\eta(|a|)$ and $\delta=\delta(|a|)$, and its complementary $B_\delta\setminus\mathcal C_{\eta,\delta} (a)$. 
We write the difference of the
nonlocal terms, corresponding to the maximum point $(\bar x, \bar y)$, as the sum 
\begin{eqnarray}\nonumber 
\mathcal I[\bar x,p,u] - \mathcal I[\bar y,p,v] = 
\mathcal T^1(\bar x,\bar y) + \mathcal T^2(\bar x,\bar y) + \mathcal T^3(\bar x,\bar y) 
\end{eqnarray}
where
\begin{eqnarray*}
\mathcal T^1(\bar x,\bar y) & = & 
  \int_{\R^d\setminus B} \left(u(\bar x + z) - u(\bar x)\right)\mu_{\bar x}(dz)\\&&\hspace{2cm}
   -\int_{\R^d\setminus B} \left(v(\bar y + z) - v(\bar y)\right)\mu_{\bar y}(dz)\\
\mathcal T^2(\bar x,\bar y) & = & 
  \int_{\mathcal C_{\eta,\delta}(a)} \left(u(\bar x + z) - u(\bar x) - p\cdot z\right)\mu_{\bar x}(dz) \\&&\hspace{2cm}
  - \int_{\mathcal C_{\eta,\delta}(a)}\left(v(\bar y + z) - v(\bar y) - p\cdot z\right)\mu_{\bar y}(dz)\\
\mathcal T^3(\bar x,\bar y) & = & 
  \int_{B\setminus\mathcal C_{\eta,\delta}(a)} \left(u(\bar x + z) - u(\bar x) - p\cdot z\right)\mu_{\bar x}(dz)\\&&\hspace{2cm}
  - \int_{B\setminus\mathcal C_{\eta,\delta}(a)}\left(v(\bar y + z) - v(\bar y) - p\cdot z\right)\mu_{\bar y}(dz).
\end{eqnarray*}
Let $\phi(z) = \varphi(|z|).$ Then 
$
p= D\phi(a) .
$
Since $(\bar x,\bar y)$ is a maximum point of $\psi(\cdot,\cdot)$, we have that
\begin{eqnarray}\label{ineq_max} \nonumber  
u(\bar x+z) - u(\bar x) -p\cdot z  & \leq &  v(\bar y+z')-v(\bar y) - p\cdot z' \\ &&
 + \phi(a +z-z') - \phi(a) -D\phi(a)\cdot (z-z'). 
\end{eqnarray} 
In the following we give estimates for each of these integral terms, using inequality (\ref{ineq_max}) 
and properties of the L\'evy measures $\left(\mu_x\right)_x$.

\begin{lemma} \label{lemma_est_NL1}
$\mathcal T^1(\bar x,\bar y)$ is uniformly bounded with respect to all parameters. More precisely
\begin{eqnarray}\nonumber
\mathcal T^1(\bar x,\bar y) \leq 
  4\max(||u||_\infty, ||v||_\infty) \sup_{x \in \R^d} \mu_{x}(\R^d\setminus B).
\end{eqnarray}
\end{lemma}


\begin{proof}[Proof of Lemma \ref{lemma_est_NL1}]
Since the functions $u$ and $v$ are bounded, we immediately deduce that 
\begin{eqnarray*}
\mathcal T^1(\bar x,\bar y) 
&\leq & 2||u||_\infty\int_{\R^d\setminus B}\mu_{\bar x}(dz) + 2||v||_\infty\int_{\R^d\setminus B}\mu_{\bar y}(dz).
\end{eqnarray*}
We conclude by recalling that the measures  $\mu_x$ are uniformly bounded away from the origin, by assumption $(M1)$.

\end{proof}


\begin{lemma} \label{lemma_est_NL2}
Let $\delta = |a|\delta_0$ with $\delta_0 \in(0,1)$ small, $\eta$ be small enough such that $1-\eta-\delta_0>0$
and $$\tilde \eta = \frac{1- \eta-\delta_0}{1+\delta_0}.$$ 
Then the nonlocal term $\mathcal T^2$ satisfies
\begin{eqnarray*}
 \mathcal T^2(\bar x,\bar y)& \leq& \frac{1}{2}\int_{\mathcal C_{\eta,\delta}(a)} \sup_{|s|\leq1}
\left((1-\tilde\eta^2)\frac{\varphi'(|a+sz|)}{|a+sz|} + \tilde\eta^2 \varphi''(|a+sz|)\right)|z|^2 
(\mu_{\bar x}+\mu_{\bar y})(dz).
\end{eqnarray*}
\end{lemma}

\begin{remark}
The previous notations have been introduced to simplify the form of the estimates. It is important to note however that 
the coefficients appearing in the convex combination of the derivatives of $\varphi$ 
depend explicitly on $\tilde\eta$ and not on the aperture of the cone, given in terms of $\eta$. 
We eventually set $\eta \sim |a|^{2\alpha}$ and $\delta_0 \sim |a|^{\alpha}$, thus we expect to have $\tilde \eta \simeq 1$. 
Consequently, the second derivative of $\varphi$ would dominate the nonlocal difference
and would render $\mathcal T^2(\bar x,\bar y)$ as negative as needed.
\end{remark}


\begin{proof}[Proof of Lemma \ref{lemma_est_NL2}]
Taking $z'=0$ and $z=0$ in inequality (\ref{ineq_max}) we have
\begin{eqnarray*}
 u(\bar x + z) -u(\bar x) -p\cdot z & \leq &\phi(a+z) - \phi(a) - D\phi(a)\cdot z \\
- \left(v(\bar y+z') -v(\bar y) - p\cdot z' \right)& \leq &\phi(a-z') - \phi(a) + D\phi(a)\cdot z' .
\end{eqnarray*}
Therefore
\begin{eqnarray*}
 \mathcal T^2(\bar x,\bar y) &\leq & 
      \int_{\mathcal C_{\eta,\delta}(a)}\left(\phi(a+z) - \phi(a) - D\phi(a)\cdot z\right) \mu_{\bar x}(dz) 
\\&& \hspace{1cm}
+ \int_{\mathcal C_{\eta,\delta}(a)}\left(\phi(a-z') - \phi(a) + D\phi(a)\cdot z'\right)\mu_{\bar y}(dz').
\end{eqnarray*}
Using Taylor's formula with integral reminder, the right hand side 
can be rewritten as
\begin{eqnarray*}
&&
\frac12\int_{ 0}^1(1-s)ds\int_{\mathcal C_{\eta,\delta}(a)} \left( D^2\phi(a+sz) z \cdot z \right)\mu_{\bar x}(dz) 
\\ && \hspace{2cm}
+\frac12\int_{-1}^0(1+s)ds\int_{\mathcal C_{\eta,\delta}(a)} \left( D^2\phi(a{+}sz) z \cdot z \right)\mu_{\bar y}(dz).
\end{eqnarray*}
Remark that the first and second derivatives of $\phi(z) = \varphi(|z|)$ are given by the formulas
\begin{eqnarray*}
D\phi(z)   & = &\varphi'(|z|) \hat z\\
D^2\phi(z) & = &\frac{\varphi'(|z|)}{|z|}(I - \hat z\otimes \hat z) + \varphi''(|z|) \hat z\otimes \hat z,
\end{eqnarray*}
and in particular
\begin{eqnarray*}
D^2\phi(a+sz) z \cdot z = \frac{\varphi'(|a+sz|)}{|a+sz|}\left(|z|^2 -| \widehat{(a+sz)}\cdot z|^2\right) 
+ \varphi''(|a+sz|)| \widehat{(a+sz)}\cdot z|^2.
\end{eqnarray*}
On the set $\mathcal C_{\eta,\delta}(a)$ we have the following upper and lower bounds 
\begin{eqnarray}\label{eq_cone_bounds}
&& \nonumber |a + sz|\geq |a| -|s| |z| \geq |a|-\delta = |a|(1-\delta_0) \\
&&|a + sz|\leq |a| +|s| |z| \leq |a|+\delta = |a|(1+\delta_0) \\
&&\nonumber |(a + sz)\cdot z| \geq |a\cdot z| - s |z|^2 \geq |a\cdot z| - \delta |z| \geq (1-\eta-\delta_0)|z||a|.
\end{eqnarray}
Hence we deduce that for all $s\in(-1, 1)$
\begin{equation}\label{eq_est_cone_s}
|\widehat{(a + sz)}\cdot z| \geq \tilde \eta |z| \hbox{ with } \tilde \eta = \frac{1-\eta-\delta_0}{1+\delta_0}.  
\end{equation}
Recalling that $\varphi$ is increasing and concave, we get
\begin{eqnarray*}
&& D^2\phi(a+sz) z \cdot z\leq (1-\tilde\eta^2)\frac{\varphi'(|a+sz|)}{|a+sz|}|z|^2 + \tilde\eta^2\varphi''(|a+sz|)|z|^2.
\end{eqnarray*}
This implies that the integral terms corresponding to $\phi$ are bounded by
\begin{eqnarray*}
&&\frac{1}{2} \int_{\mathcal C_{\eta,\delta}(a)} \sup_{|s|\leq1}
\left((1-\tilde\eta^2) \frac{\varphi'(|a+sz|)}{|a+sz|}+ \tilde\eta^2\varphi''(|a+sz|)\right)
|z|^2 (\mu_{\bar x}+\mu_{\bar y})(dz).
\end{eqnarray*}
which concludes the proof of the lemma.
\end{proof}


\begin{lemma}\label{lemma_est_NL3}
The following estimate holds
\begin{eqnarray*}
\mathcal T^3(\bar x,\bar y) \leq  
\int_{B_\delta\setminus \mathcal C_{\eta,\delta}(a)}\sup_{|s|\leq1}\frac{\varphi'(|a+sz|)}{|a+sz|} |z|^2\left|\mu_{\bar x}-\mu_{\bar y}\right|(dz) 
+ 2\varphi'(|a|) \int_{B\setminus B_\delta}|z|\left|\mu_{\bar x}-\mu_{\bar y}\right|(dz) .
\end{eqnarray*}
\end{lemma}


\begin{proof}[Proof of Lemma \ref{lemma_est_NL3}]
When estimating the nonlocal term outside the cone, one has to keep it as small as possible, though positive. 
Therefore we consider, as in \cite{BCI:11:HdrNL} the signed measure $\mu = \mu_{\bar x} - \mu_{\bar y}$. 
Consider its Jordan decomposition $\mu = \mu^+ - \mu^-$ and denote by $|\mu|$ the corresponding total variation measure. 
Then, if  $K$ is the support of the positive variation $\mu^+$, one can define the minimum of the two measures as
$$\mu_* ={\bf 1}_K\mu_{\bar y} +(1-{\bf 1}_K) \mu_{\bar x}.$$
But then, the measures $\mu_{\bar x}$ and $\mu_{\bar y}$  can be rewritten as 
$
\mu_{\bar x} = \mu_* + \mu^+ \hbox{ and } 
\mu_{\bar y} = \mu_* + \mu^-.
$
With these notations in mind, we rewrite the nonlocal term $\mathcal T^3$ as
\begin{eqnarray*}
\mathcal T^3(\bar x,\bar y) &  = & 
\int_{B\setminus\mathcal C_{\eta,\delta}(a)} \left( u(\bar x+z) -u(\bar x) -p\cdot z  - ( v(\bar y +z) -v(\bar y) -p\cdot z )\right)\mu_*(dz)\\ 
 && + \int_{B\setminus\mathcal C_{\eta,\delta}(a)} ( u(\bar x+z) -u(\bar x) -p\cdot z ) \mu^+(dz) \\ 
 && - \int_{B\setminus\mathcal C_{\eta,\delta}(a)} ( v(\bar y+z) -v(\bar y) -p\cdot z ) \mu^-(dz).
\end{eqnarray*}
Choosing successively $z'=z$, $z'=0$ and $z=0$ in (\ref{ineq_max}) and noting that
\begin{eqnarray*} &&
u(\bar x+z) - u(\bar x) -p\cdot z  \leq  v(\bar y+z)-v(\bar y) - p\cdot z
\end{eqnarray*}
we deduce that
\begin{eqnarray*}
 \mathcal T^3(\bar x,\bar y) 
 &\leq& 
 \int_{B\setminus\mathcal C_{\eta,\delta}(a)} \left(\phi(a+z)-\phi(a)-D\phi(a)\cdot z\right)\mu^+(dz) \\ &&  \hspace{2cm}  
 +\int_{B\setminus\mathcal C_{\eta,\delta}(a)} \left(\phi(a-z)-\phi(a)+D\phi(a)\cdot z\right)\mu^-(dz). 
\end{eqnarray*}
For estimating the integral terms corresponding to $\phi$, we split the domain of integration into 
$B\setminus B_\delta$ and $B_\delta\setminus\mathcal C_{\eta,\delta}(a)$. 
On the first set, from the monotonicity and the concavity  of $\varphi$ we have
\begin{eqnarray*}
\phi(a+z)-\phi(a) - D\phi(a)\cdot z & \leq & \varphi(|a|+|z|)-\varphi(|a|) - \varphi'(|a|)\hat a\cdot z \\
&\leq & 2\varphi'(|a|) |z|.
\end{eqnarray*}
On $B_\delta\setminus\mathcal C_{\eta,\delta}(a)$ we use a second order Taylor expansion 
and we take into account that $\varphi$ is smooth, $\varphi' \geq0$ and $\varphi''\leq0$ to obtain the upper bound
\begin{eqnarray*}
\sup_{|s|\leq1}\left(\phi(a+sz)-\phi(a) - D\phi(a)\cdot z\right) 
& \leq & \sup_{|s|\leq1} D^2\phi(a+sz) z \cdot z\\
& \leq & \sup_{|s|\leq1}\frac{\varphi'(|a+sz|)}{|a+sz|}|z|^2.
\end{eqnarray*}
Therefore we get the estimate 
\begin{eqnarray*}
 \mathcal T^3(\bar x,\bar y) \leq  
 2\varphi'(|a|) \int_{B\setminus B_\delta}|z|\left|\mu_{\bar x}-\mu_{\bar y}\right|(dz) 
 +\int_{B_\delta\setminus \mathcal C_{\eta,\delta}(a)}\sup_{|s|\leq1}\frac{\varphi'(|a+sz|)}{|a+sz|}|z|^2\left|\mu_{\bar x}-\mu_{\bar y}\right|(dz).
\end{eqnarray*}
\end{proof}
\noindent From the three above lemmas, we obtain the final estimate for the nonlocal term.
\end{proof} 

\begin{proof}[Proof of Corollary \ref{cor_LipEst}]
Remark that $|a|\leq{t_0}$. Indeed, since the maximum of $\psi$ is positive and in view of the lower bound on $L$, we have
$$
{
\varphi(|a|)< ||u||_\infty + ||v||_\infty \leq L{t_0}\frac{\alpha}{1+\alpha} = \varphi({t_0})
}
$$
which by the strict monotonicity of $\varphi$ implies the desired inequality. 
We first evaluate the estimate that renders the integral difference negative, namely:
\begin{eqnarray*}
&& \sup_{|s|\leq1}\left((1-\tilde\eta^2)\frac{\varphi'(|a+sz|)}{|a+sz|}+ \tilde\eta^2\varphi''(|a+sz|)\right)\\
&& \hspace{2cm}  
=  L \; \sup_{|s|\leq1}\left((1-\tilde\eta^2)\frac{1-\varrho(1+\alpha)|a+sz|^\alpha}{|a+sz|} -\tilde\eta^2 \varrho\alpha(1+\alpha)|a+sz|^{\alpha-1}\right)\\
&&\hspace{2cm}  
\leq L \; \sup_{|s|\leq1} \left(\frac{1-\tilde\eta^2}{|a+sz|} - \varrho(1+\alpha)(1-\tilde \eta^2+\alpha\tilde\eta^2) |a+sz|^{\alpha-1}\right).
\end{eqnarray*}
{
Using the fact that $\tilde \eta^2 \leq 1 \leq \frac{1}{1-\alpha^2}$ we have that $(1+\alpha)(1-\tilde \eta^2+\alpha\tilde\eta^2)\geq \alpha$ which further implies
} 
\begin{eqnarray*}
&& \sup_{|s|\leq1}\left((1-\tilde\eta^2)\frac{\varphi'(|a+sz|)}{|a+sz|}+ \tilde\eta^2\varphi''(|a+sz|)\right)
\leq L \; \sup_{|s|\leq1}\left(\frac{1-\tilde\eta^2}{|a+sz|} - \varrho\alpha |a+sz|^{\alpha-1}\right).
\end{eqnarray*}
But this quantity has to be integrated over the cone $\mathcal C_{\eta,\delta}(a)$, in which case $|a+sz|$
satisfies
$$
|a|(1-\delta_0)\leq|a+sz|\leq |a|(1+\delta_0).
$$
Thus, observing that $1-\tilde \eta^2\leq 2(1-\tilde \eta)$, the previous inequality takes the form
\begin{eqnarray*}&&
\sup_{|s|\leq1}\left((1-\tilde\eta^2)\frac{\varphi'(|a+sz|)}{|a+sz|}+ \tilde\eta^2\varphi''(|a+sz|)\right)
\leq L \left(\frac{2(1-\tilde\eta)}{|a|(1-\delta_0)} - \varrho\alpha(1+\delta_0)^{\alpha-1}|a|^{\alpha-1}\right).
\end{eqnarray*}
Let $\tilde \eta$ be of the form
$$
1- \tilde \eta = |a|^\alpha\tilde\eta_0
$$
with small $\tilde \eta_0 <\frac14$. Choose accordingly $\delta_0$ and $\eta$ of the form 
$$\delta_0 = c_1 |a|^{\alpha_1} \hspace{1cm}
\eta = c_2 |a|^{\alpha_2}.$$ 
Recalling that
$\tilde \eta = \frac{1-\delta_0 -\eta}{1+\delta_0}$ we get that $c_1, c_2, \alpha_1$ and $\alpha_2$ must satisfy
$$
c_2 |a|^{\alpha_2} + 2c_1 |a|^{\alpha_1}   = c_1\tilde\eta_0 |a|^{\alpha+\alpha_1} +\tilde\eta_0|a|^\alpha .$$
Identifying the coefficients we obtain 
$$
 \delta_0 = \frac12 |a|^\alpha\tilde\eta_0 \hbox{  and  } \eta = \frac12 |a|^{2\alpha}\tilde\eta_0^2.
$$
Subsequently, the choice of parameters $\eta, \delta_0$ and $\tilde \eta_0$ gives us 
\begin{eqnarray*}
\sup_{|s|\leq1}\left((1-\tilde\eta^2)\frac{\varphi'(|a+sz|)}{|a+sz|}+ \tilde\eta^2\varphi''(|a+sz|)\right)
\leq -L \left( \varrho \alpha 2^{\alpha -1} -1\right)|a|^{\alpha-1}.
\end{eqnarray*}
This leads to a negative upper bound of the integral term taken over the cone $\mathcal C_{\eta,\delta}(a)$:
\begin{eqnarray*}
\int_{\mathcal C_{\eta,\delta}(a)}\sup_{|s|\leq1}\left((1-\tilde\eta^2)\frac{\varphi'(|a+sz|)}{|a+sz|}+
	\tilde\eta^2 \varphi''(|a+sz|)\right)|z|^2 \mu_{\bar x}(dz)\\
\leq - L \left( \varrho \alpha 2^{\alpha -1}-1\right)|a|^{\alpha-1}\int_{\mathcal C_{\eta,\delta}(a)} |z|^2 \mu_{\bar x}(dz).
\end{eqnarray*}
Let $\varTheta(\varrho,\alpha) = \varrho \alpha 2^{\alpha -1}-1> 0$ and use $(M2)$ and the fact that $\delta = |a|\delta_0$ to finally get
\begin{eqnarray*}
&& 
\int_{\mathcal C_{\eta,\delta}(a)}\sup_{|s|\leq1}\left((1-\tilde\eta^2)\frac{\varphi'(|a+sz|)}{|a+sz|}
+\tilde\eta^2 \varphi''(|a+sz|)\right)|z|^2 \mu_{\bar x}(dz)\\
&& \hspace{3.5cm} 
\leq - L \varTheta(\varrho,\alpha)  |a|^{\alpha-1} C_\mu \eta^{\frac{d-1}{2}}\delta^{2-\beta} 
\\ && \hspace{3.5cm}  = 
- L \varTheta(\varrho,\alpha) C^1_\mu  |a|^{\alpha-1} |a|^{\alpha(d-1)}|a|^{(1+\alpha)(2-\beta)}.
\end{eqnarray*}
Less technical estimates give us similar upper bounds for the other two integrals. 
More precisely, we have in view of assumption $(M3)$
\begin{eqnarray*}
2\varphi'(|a|) \int_{B\setminus B_\delta}|z||\mu_{\bar x}-\mu_{\bar y}|(dz) &\leq  & 
2 L C_\mu |a|^\gamma \delta^{1-\beta}\\
& = & 
L C^2_\mu   |a|^\gamma |a|^{(1+\alpha)(1-\beta)}
\end{eqnarray*}
and 
\begin{eqnarray*}
\int_{B_\delta\setminus \mathcal C_{\eta,\delta}(a)}\sup_{|s|\leq1}\frac{\varphi'(|a+sz|)}{|a+sz|}|z|^2|\mu_{\bar x}-\mu_{\bar y}|(dz)
& \leq & L \frac{C_\mu |a|^\gamma \delta^{2-\beta}}{|a|(1-\delta_0)}\\ 
&\leq  & L C^3_\mu    |a|^{\gamma-1} |a|^{(1+\alpha)(2-\beta)}.
\end{eqnarray*}

\noindent  For $\beta>1$ and $\alpha>0$ such that $\gamma>\alpha(d+1)$
 the difference of the two nonlocal terms becomes negative:
\begin{eqnarray*}
&&
\mathcal  I[\bar x,p,u] -  \mathcal I[\bar y,p,v] 
\\ &&  \hspace{1cm}  
 \leq -L |a|^{1-\beta}\left\{ C^1_\mu \varTheta(\varrho,\alpha,\mu)|a|^{\alpha(d+2-\beta)} 
 -C^2_\mu|a|^{\gamma+\alpha(1-\beta)}-C^3_\mu |a|^{\gamma+\alpha(2-\beta)}\right\} + O(\tilde C_\mu)
\\ &&  \hspace{1cm}    
= -L|a|^{(1-\beta)+\alpha(d+2-\beta)}\left\{ C^1_\mu \varTheta(\varrho,\alpha,\mu) - o_{|a|}(1)\right\}+O(\tilde C_\mu).
\end{eqnarray*}

\end{proof}

\begin{proof}[Proof of Corollary \ref{cor_HoeEst}]
Estimating the integrand of the nonlocal difference  $\mathcal T^2$ we get
\begin{eqnarray*}
&& 
\sup_{|s|\leq1}\left((1-\tilde\eta^2)\frac{\varphi'(|a+sz|)}{|a+sz|}+ \tilde\eta^2\varphi''(|a+sz|)\right)
\\ && \hspace{2cm}  
= L \alpha\left(1-(2-\alpha)\tilde\eta^2\right) \inf_{|s|\leq1}\left(|a+sz|^{\alpha-2}\right)
\\ && \hspace{2cm}  
\leq - L \alpha\left((2-\alpha)\tilde\eta^2-1\right)(1+\delta_0)^{\alpha-2} |a|^{\alpha-2}.
\end{eqnarray*}
Choose $\eta$ and $\delta_0$ sufficiently small such that {$\delta_0<\frac12$}
$$
(2-\alpha)\tilde\eta^2 = (2-\alpha)\left(\frac{1-\eta - \delta_0}{1+\delta_0}\right)^2> {\frac12}.
$$ 
{Remark that, contrary to the Lipschitz case, $\eta$ and $\delta_0$ do not depend on $|a|$.}
We then obtain {due to $(M2)$} a negative bound of the integral term over the cone $\mathcal C_{\eta,\delta}(a)$, for $\delta= |a|\delta_0$:
\begin{eqnarray*}
&& 
\int_{\mathcal C_{\eta,\delta}(a)}\sup_{|s|\leq1}\left((1-\tilde\eta^2)\frac{\varphi'(|a+sz|)}{|a+sz|}+
	\tilde\eta^2 \varphi''(|a+sz|)\right)|z|^2 \mu_{\bar x}(dz)
\\ && \hspace{2.5cm}
\leq - L {\frac{\alpha}{2}}(1+\delta_0)^{\alpha-2}|a|^{\alpha-2}
\int_{\mathcal C_{\eta,\delta}(a)} |z|^2 \mu_{\bar x}(dz) 
\\ && \hspace{2.5cm}
\leq - L {{\alpha C(\mu)}} |a|^{\alpha-\beta}.
\end{eqnarray*}
In addition, {in view of $(M3)$} we have the estimates of the other two integral terms, when $\beta \neq1$
\begin{eqnarray*}
2\varphi'(|a|) \int_{B\setminus B_\delta}|z||\mu_{\bar x}-\mu_{\bar y}|(dz) 
&\leq & 2 L\alpha |a|^{\alpha-1} C_\mu |a|^\gamma \delta^{1-\beta} \\
& = &  L\alpha C^2_\mu   |a|^\gamma |a|^{\alpha-\beta}
\end{eqnarray*}
and for $\beta  = 1$
\begin{eqnarray*}
&&2 \varphi'(|a|) \int_{B\setminus B_\delta}|z||\mu_{\bar x}-\mu_{\bar y}|(dz)\leq 
L\alpha C^2_\mu   |a|^\gamma |\ln(|a|\delta_0)| |a|^{\alpha-\beta}.
\end{eqnarray*}
Similarly
\begin{eqnarray*}
&& \int_{B_\delta\setminus \mathcal C_{\eta,\delta}(a)}
\sup_{|s|\leq1}\frac{\varphi'(|a+sz|)}{|a+sz|}|z|^2|\mu_{\bar x}-\mu_{\bar y}|(dz)\\
&&\hspace{2cm}
\leq L \alpha \left(|a|(1-\delta_0)\right)^{\alpha-2 }
\int_{B_\delta\setminus \mathcal C_{\eta,\delta}(a)}|z|^2|\mu_{\bar x}-\mu_{\bar y}|(dz)\\
&&\hspace{2cm} 
\leq L \alpha C^3_\mu |a|^\gamma |a|^{\alpha-\beta}.
\end{eqnarray*} 
Therefore  the difference of the nonlocal term becomes negative, as bounded from above by
\begin{eqnarray*}
&& \mathcal  I[\bar x,p,u] -  \mathcal I[\bar y,p,v] \leq   - L |a|^{\alpha-\beta} 
\left({\alpha C(\mu)} - o_{|a|}(1) \right)  + {O(\tilde C_\mu)}.
\end{eqnarray*}
\end{proof}

\subsection{L\'evy-It\^o Operators}
{We now establish similar results for L\'evy-It\^o operators
$$
\mathcal J[x,u] = \int_{\R^d}\left(u(x+j(x,z)) - u(x) - Du(x)\cdot j(x,z)1_B(z))\right) \mu(dz).
$$
As before, we give a general result on concave estimates for the difference of two L\'evy-It\^o operators. Then we present the Lipschitz and H\"older estimates as corollaries.
In addition, we provide the quadratic estimates that are used in the uniqueness argument, in the proof of the partial regularity result, Theorem \ref{thm_part_Lip_NL}.}

\begin{prop}[\bf Concave estimates - L\'evy-It\^o operators]
Assume conditions $(J1)$ and {$(J4)$} hold.
Let $u,v$ be two bounded functions,
$\varphi:[0,\infty)\rightarrow\R$ be a smooth increasing concave function and define
$$
\psi(x,y) = u(x) - v(y) - \varphi(|x-y|).
$$
Assume that $\psi$ attains a positive maximum at $(\bar x,\bar y)$, with $\bar x \neq \bar y$.  
Let $a=\bar x-\bar y$,  $\hat a = a/|a|$ and $p=\varphi'(|a|)\hat a$.
Then the following holds

\begin{eqnarray*}
\mathcal J[\bar x,p,u] & - & \mathcal J[\bar y,p,v] \;\;\; \leq  \;\;\;   4\tilde C_\mu \max(||u||_\infty, ||v||_\infty)   
\\  &&
+\frac12\int_{\mathcal C} \sup_{\substack{|s|\leq1 \\ x= {\bar x,\bar y}}}
\left(\left((1-\tilde\eta^2)\frac{ \varphi'(|a+sj(x,z)|)}{|a+sj(x,z)|}+ 
      \tilde\eta^2\varphi''(|a+sj(x,z)|)\right)|j(x,z)|^2\right) \mu(dz) 
\\ &&  
+ 2\varphi'(|a|) \int_{\substack{B\setminus \mathcal C\\  |\Delta(z)|\geq \delta}} |\Delta( z)|\mu(dz)+
\int_{\substack{B\setminus \mathcal C \\ |\Delta(z)|\leq \delta}}
\sup_{|s|\leq1}\frac{\varphi'(|a+s\Delta (z)|)}{|a+s\Delta (z)|} |\Delta (z)|^2\mu(dz)
\end{eqnarray*}
where $ \Delta (z) = j(\bar x,z) - j(\bar y,z)$,
$$
\mathcal C = \left\{ z; \left| j(\frac{\bar x + \bar y}{2},z)\right|\leq \frac{\delta}{2} 
\hspace{0.3cm} \hbox{ and } \hspace{0.3cm}  
\left|j(\frac{\bar x+\bar y}{2},z)\cdot \hat a\right|\geq(1-\frac{\eta}{2})\left|j(\frac{\bar x+\bar y}{2},z)\right|)\right\}
$$ 
$$
{\left(\frac{|a|}{2}\right)^\gamma  \leq \frac{c_0}{C_0}\frac{\eta}{4-\eta},} 
\hspace{1.8cm} \delta = |a|\delta_0>0,
\hspace{1.8cm} \tilde \eta = \frac{1-\eta-\delta_0}{1+\delta_0}>0
$$ 
with $\delta_0 \in(0,1)$ and ${\eta\in(0,1)}$ both sufficiently small.
\label{prop_concave_estimate_LI}
\end{prop}

\begin{corollary}[\bf Lipschitz estimates]
Let {$\beta>1 \geq 2(1-\gamma)$} and assume that conditions $(J1) - (J4)$ hold.
Under the assumptions of Proposition \ref{prop_concave_estimate_LI}  with 
\begin{equation*}
 \varphi(t) = \left\{ 
 \begin{array}{ll}
  L\left ( t - \varrho t^{1+\alpha}\right), & t\in [0,{t_0}]\\
  \varphi({t_0}), &t>{t_0}
 \end{array}\right.
\end{equation*}
where {$\alpha \in\left(0,\min\left(\frac{\gamma\beta}{d+1},\frac{\beta-1}{d+2-\beta}\right)\right)$}, 
$\varrho$ is a constant such that $ \varrho \alpha 2^{\alpha -1} > 1$,
${t_0} = \max_t (t - \varrho t^{1+\alpha})=\sqrt[\alpha]{\frac{1}{\rho(1+\alpha)}}$ 
and  $L > {\frac{(||u||_\infty + ||v||_\infty)(\alpha+1)}{{t_0}\alpha},}$
the following holds:
there exists a positive constant $C= C(\mu)$ such that for 
$\Theta(\varrho,\alpha,\mu) = C\left(\rho\alpha 2^{\alpha-1}-1\right)$
we have
\begin{eqnarray*}&&
  \mathcal  J[\bar x,p,u] -  \mathcal J[\bar y,p,v] \leq - L |a|^{(1-\beta)+\alpha(d+2-\beta)}\left\{ \varTheta(\varrho,\alpha,\mu) - o_{|a|}(1)\right\}+O(\tilde C_\mu).
\end{eqnarray*}
\label{cor_LipEst_LI}
\end{corollary} 

\begin{remark}{
The condition $\beta>2(1-\gamma)$ connects the singularity of the measure with the regularity of the jumps. 
It says that the more singular the measure is, the less regular the jumps can be.}
\end{remark}

\begin{corollary}[\bf H\"older estimates]
{Let $\beta>2(1-\gamma)$} and assume that conditions $(J1) - (J4)$ hold.
Under the assumptions of Proposition \ref{prop_concave_est_NL}  with 
\begin{equation*}
 \varphi(t) = \left\{ 
 \begin{array}{ll}
  Lt^\alpha, & t\in [0,{t_0}]\\
  \varphi({t_0}), &t>{t_0}
 \end{array}\right.
\end{equation*}
where $\alpha \in(0,\min(\beta,1))$, ${{t_0}>0}$, and 
$L>{\frac{||u||_\infty + ||v||_\infty}{{t_0}^\alpha},}$ 
the following holds: there exists a positive constant  
{$C(\mu) > 0$}
such that
\begin{eqnarray*}&&
  \mathcal  J[\bar x,p,u] -  \mathcal J[\bar y,p,v] \leq - L |a|^{\alpha -\beta}\left\{ {\alpha C(\mu)} - o_{|a|}(1)\right\}+O(\tilde C_\mu).
\end{eqnarray*}
\label{cor_HdrEst_LI}
\end{corollary}

\begin{proof}[Proof of Proposition \ref{prop_concave_estimate_LI}]
In this case, the difference of the nonlocal terms reads
\begin{eqnarray*}
\mathcal J [\bar x, p, u] - \mathcal J[\bar y, p, v] & = &
\int_{\R^d} \left( u(\bar x + j(\bar x,z) - u(\bar x) - p\cdot j(\bar x, z) 1_B(z))\right) \mu(dz) \\
&& - \int_{\R^d} \left( v(\bar y + j(\bar y,z) - v(\bar y) - p\cdot j(\bar y, z) 1_B(z))\right) \mu(dz) .
\end{eqnarray*}
Similarly to general nonlocal operators we split the domain of integration into the  cone $\mathcal C$, 
its complementary in the unit ball $B\setminus \mathcal C$ and the region away from the origin $\R^d\setminus B$.
Remark that the cone has the property 
\begin{equation}\label{eq_cones}
\mathcal C:=\mathcal C_{\delta/2,\eta/2}\left(\frac{\bar x + \bar y}{2}\right) \subset 
\mathcal C_{\delta,\eta}(\bar x) \cap \mathcal C_{\delta,\eta}(\bar y).
\end{equation}
\noindent Indeed, for $ |a|$ sufficiently small such that $ \left(\frac{ |a|}{2}\right)^\gamma \leq \frac{c_0}{C_0}$, if $z\in\mathcal{C}$ then
\begin{eqnarray*}
| j(\bar x ,z)| &\leq  & | j(\frac{\bar x + \bar y}{2},z) - j(\bar x,z)| + | j(\frac{\bar x + \bar y}{2},z)| \\
& \leq & C_0 |z| \left(\frac{ |a|}{2}\right)^\gamma + \frac\delta 2 
 \leq \frac{\delta}{2} \frac{C_0}{c_0} \left(\frac{ |a|}{2}\right)^\gamma + \frac\delta 2  \leq \delta
\end{eqnarray*} 
since $c_0|z|\leq |j(\frac{\bar x + \bar y}{2},z)|\leq \frac\delta2$.
At the same time, we use the fact that $\left(\frac{|a|}{2}\right)^\gamma  \leq \frac{c_0}{C_0}\frac{\eta}{4-\eta},$ to get { from $(J4)$}
\begin{eqnarray*}
| j(\bar x ,z)\cdot \hat a|  & \geq & | j(\frac{\bar x + \bar y}{2},z) \cdot \hat a| - | j(\frac{\bar x + \bar y}{2},z) - j(\bar x,z)| \\
&\geq & (1-\frac\eta 2) |j(\frac{\bar x + \bar y}{2},z)| - | j(\frac{\bar x + \bar y}{2},z) - j(\bar x,z)| \\
&\geq & (1-\frac\eta 2) |j(\bar x, z)| -  (2-\frac{\eta}{2})| j(\frac{\bar x + \bar y}{2},z) - j(\bar x,z)|) \\
&\geq & (1-\frac\eta 2) |j(\bar x, z)| - (2-\frac{\eta}{2}) C_0 |z|\left(\frac{|a|}{2}\right)^\gamma \\
&\geq & (1-\frac\eta 2) |j(\bar x, z)| - (2-\frac\eta 2)\frac{C_0}{c_0} |j(\bar x,z)|\left(\frac{|a|}{2}\right)^\gamma \geq (1-\eta)|j(\bar x,z)|.
\end{eqnarray*} 
\begin{figure}
\centering
\includegraphics[width=6cm]{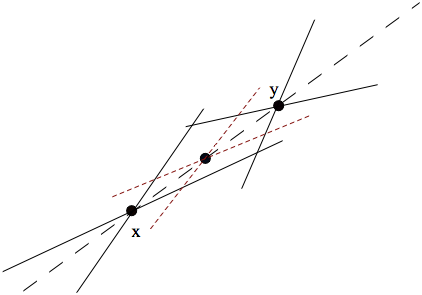}
\caption{The middle cone $\mathcal C_{\delta/2,\eta/2}\left(\frac{\bar x + \bar y}{2}\right) \subset \mathcal C_{\delta,\eta}(\bar x) \cap \mathcal C_{\delta,\eta}(\bar y).$}
\end{figure}

\noindent
Let $\phi(z) = \varphi(|z|)$. Then 
$
p= D\phi(a).
$
Accordingly, we write the previous difference as the sum 
$$
\mathcal J [\bar x, p, u] - \mathcal J[\bar y, p, v] = 
\mathcal T^1( \bar x,\bar y) + \mathcal T^2( \bar x,\bar y) +\mathcal T^3( \bar x,\bar y) ,
$$
where
\begin{eqnarray*}
\mathcal T^1(\bar x,\bar y) & = & \int_{\R^d\setminus B} \left(u(\bar x + j(\bar x,z))- u(\bar x)\right)\mu(dz) \\
&&\hspace{2cm}
-\int_{\R^d\setminus B} \left(v(\bar y +j(\bar y, z)) - v(\bar y)\right)\mu(dz)\\
\mathcal T^2(\bar x,\bar y) & = & \int_{\mathcal C} \left(u(\bar x + j(\bar x, z)) - u(\bar x) - p\cdot j(\bar x,z)\right)\mu(dz)\\
&&\hspace{2cm}
- \int_{\mathcal C}\left(v(\bar y + j(\bar y,z)) - v(\bar y) - p\cdot j(\bar y,z)\right)\mu(dz)\\
\mathcal T^3(\bar x,\bar y) & = & \int_{B\setminus\mathcal C} \left(u(\bar x + j(\bar x,z)) - u(\bar x) - p\cdot j(\bar x,z)\right)\mu(dz)\\
&&\hspace{2cm}
 - \int_{B\setminus\mathcal C}\left(v(\bar y + j(\bar y,z)) - v(\bar y) - p\cdot j(\bar y,z)\right)\mu(dz).
\end{eqnarray*}
As before, we next estimate each of these integral terms.
The first lemma is straightforward.

\begin{lemma} \label{lemma_est_LI1}
$\mathcal T^1(\bar x,\bar y)$ is uniformly bounded with respect to all the parameters, namely
\begin{eqnarray}\nonumber
\mathcal T^1(\bar x,\bar y) \leq 4  \max(||u||_\infty, ||v||_\infty) \sup_{x\in\R^d} \mu_{x}(\R^d\setminus B).
\end{eqnarray}
\end{lemma}


\begin{lemma} \label{lemma_est_LI2}
Let $\delta = |a|\delta_0$ and $\eta\in(0,\frac12)$ such that $1-\eta-\delta_0\geq0$. We have
\begin{eqnarray*}
\mathcal T^2(\bar x,\bar y) \leq 
  \int_{\mathcal C} \sup_{\substack{|s|\leq1,\\ x=\bar x,\bar y}}\left(\left((1-\tilde\eta^2)
  \frac{ \varphi'(|a+sj(x,z)|)}{|a+sj(x,z)|}+ \tilde\eta^2\varphi''(|a+sj(x,z)|)\right)|j(x,z)|^2\right) \mu(dz)
 \end{eqnarray*}
where  $\tilde \eta = (1- \eta-\delta_0)(1+\delta_0)^{-1}$. 
\end{lemma}


\begin{proof}[Proof of Lemma \ref{lemma_est_LI2}]
Writing the maximum inequality at points $\bar x, \bar y$ for the pair
$(z,z') = (j(\bar x,z),0)$ and $(z,z') = (0,j(\bar y,z))$ 
respectively,  we have
 \begin{eqnarray*}
 u(\bar x + j(\bar x, z)) -u(\bar x) -p\cdot j(\bar x,z) &\leq &
 \phi(a+j(\bar x,z)) - \phi(a) - D\phi(a)\cdot j(\bar x,z)\\
 - \left(v(\bar y+j(\bar y,z)) -v(\bar y) - p\cdot j(\bar y,z) \right)&\leq &\phi(a-j(\bar y,z)) - \phi(a) + D\phi(a)\cdot j(\bar y,z).
 \end{eqnarray*}
 Therefore
 \begin{eqnarray*}
  \mathcal T^2(\bar x,\bar y) &\leq & 
 	\int_{\mathcal C}\left(\phi(a+j(\bar x,z)) - \phi(a) - D\phi(a)\cdot j(\bar x,z)\right) \mu(dz) \\
   & & \hspace{2cm}+\int_{\mathcal C}\left(\phi(a-j(\bar y,z)) - \phi(a) + D\phi(a)\cdot j(\bar y,z)\right)\mu(dz).
 \end{eqnarray*}
Taking into account that  the set $\mathcal C$ is included in both $\mathcal C_{\eta,\delta}(\bar x)$ and  $\mathcal C_{\eta,\delta}(\bar y)$ 
(see (\ref{eq_cones})) we have, similarly to (\ref{eq_cone_bounds}) and (\ref{eq_est_cone_s}), the following upper and lower bounds for the jumps 
 \begin{eqnarray*}
 |a|(1-\delta_0)\;\;\; \geq \;\;\; |a + sj(\bar x,z)|&\geq&|a|(1-\delta_0)\\
 |\widehat{(a + sj(\overline x,z))}\cdot z| &\geq &\tilde \eta |j(\bar x,z)|.
 \end{eqnarray*}
We then conclude as we did for general nonlocal operators, within the proof of Lemma \ref{lemma_est_NL2}.
\end{proof}


\begin{lemma}
Denote by $\Delta (z) = j(\bar x,z) - j(\bar y,z)$. Then
\begin{eqnarray*}
\mathcal T^3(\bar x,\bar y) &\leq &
2 \varphi'(|a|) \int_{\left\{z\in B\setminus \mathcal C;\; \ |\Delta(z)|\geq \delta\right\}}|\Delta(z)|\mu(dz) 
\\&& \hspace{1.2cm}
+ \int_{\left\{z\in B\setminus \mathcal C;\; \ |\Delta(z)|\leq \delta\right\}}
\sup_{|s|\leq1}\frac{\varphi'(|a+s\Delta(z)|)}{|a+s\Delta(z)|}|\Delta(z)|^2\mu(dz) .
\end{eqnarray*}
\label{lemma_est_LI3}
\end{lemma}


\begin{proof}[Proof of Lemma \ref{lemma_est_LI3}]
We use again the maximum inequality to obtain the bound
\begin{eqnarray*}
\left(u(\bar x + j(\bar x,z)) - u(\bar x) - p\cdot j(\bar x,z) \right) & -&  
\left(v(\bar y + j(\bar y,z)) - v(\bar y) - p\cdot j(\bar y,z)\right) \\
&\leq & \phi(a+j(\bar x,z)-j(\bar y,z)) - \phi(a) - D\phi(a)\cdot (j(\bar x,z)-j(\bar y,z))
\end{eqnarray*}
which in particular implies
\begin{eqnarray*}
\mathcal T^3(\bar x,\bar y) &\leq & 
\int_{B\setminus\mathcal C}\left(\phi(a+j(\bar x,z)-j(\bar y,z)) - \phi(a) - D\phi(a)\cdot (j(\bar x,z)-j(\bar y,z))\right) \mu(dz).
\end{eqnarray*}
In order to estimate the integral terms corresponding to $\phi$, we split the integral in two parts, as follows
\begin{eqnarray*}
\int_{\left\{z\in B\setminus \mathcal C;\; \ |\Delta(z)|\geq \delta\right\} }\left(\phi(a+\Delta(z)) - \phi(a) - D\phi(a)\cdot \Delta(z)\right) \mu(dz) \\ 
+\int_{\left\{z\in B\setminus \mathcal C;\; \ |\Delta(z)|\leq \delta\right\} }\left(\phi(a+\Delta(z)) - \phi(a) - D\phi(a)\cdot \Delta(z)\right) \mu(dz). 
\end{eqnarray*}
On the first set we use the monotonicity and the concavity  of $\varphi$ to deduce that
\begin{eqnarray*}
\phi(a+\Delta(z))-\phi(a) - D\phi(a)\cdot \Delta(z) \leq 
2\varphi'(|a|) |\Delta(z)|.
\end{eqnarray*}
On $\left\{z\in B\setminus \mathcal C;\; \ |\Delta(z)|\leq \delta\right\}$ we use a second order Taylor expansion and we take into account that $\varphi$ is a smooth increasing function with $\varphi''\leq0$ to obtain the upper bound
\begin{eqnarray*}
 \sup_{|s|\leq1}\left(\phi(a+s\Delta(z))-\phi(a) - D\phi(a)\cdot \Delta(z)\right) &\leq& 
 \frac 12 \sup_{|s|\leq1} D^2\phi(a+s\Delta(z)) \Delta(z) \cdot \Delta(z) \\ &\leq & 
 \frac 12 \sup_{|s|\leq1}\frac{\varphi'(|a+s\Delta(z)|)}{|a+s\Delta(z)|}|\Delta(z)|^2.
\end{eqnarray*}
Therefore we get the desired estimate.
\end{proof}
\noindent The lemmas above yield the global estimate of the difference of the nonlocal terms.
\end{proof}

\begin{proof}[Proof of Corollary \ref{cor_LipEst_LI}]
We first evaluate, as for general nonlocal operators, the expression
\begin{eqnarray*}
\sup_{|s|\leq 1}\left((1-\tilde\eta^2)\frac{\varphi'(|a+sj(x,z)|)}{|a+sj(x,z)|}+ \tilde\eta^2\varphi''(|a+sj(x,z)|)\right)
\\ 
\leq  L\left(\frac{2(1-\tilde\eta)}{|a|(1-\delta_0)} - \varrho\alpha(1+\delta_0)^{\alpha-1}|a|^{\alpha-1}\right).
\end{eqnarray*}
For
$
\tilde \eta = 1-|a|^\alpha\tilde\eta_0
$
with $\tilde \eta_0 <\frac14$, consider the constant $\varTheta(\varrho,\alpha) = \varrho \alpha 2^{\alpha -1}-1>0$. Then, by $(J2)$ we have
\begin{eqnarray*}
&&\int_{\mathcal C}\sup_{|s|\leq1}\left((1-\tilde\eta^2)\frac{\varphi'(|a+sj(\bar x,z)|)}{|a+sj(\bar x, z)|}+
 \tilde\eta^2 \varphi''(|a+sj(\bar x,z)|)\right)|j(\bar x,z)|^2 \mu(dz)
\\ && \hspace{4cm}
\leq - L\varTheta(\varrho,\alpha)|a|^{\alpha-1}\int_{\mathcal C} |j(\bar x, z)|^2 \mu(dz)
\\ && \hspace{4cm}
\leq -L\varTheta(\varrho,\alpha,\mu) |a|^{(1-\beta)+\alpha(d+2-\beta)}.
\end{eqnarray*}
Similarly, taking into account assumptions $(J3)-(J4)$ and that 
$\delta = |a|\delta_0 \sim |a|^{\alpha+1}$ we obtain
\begin{eqnarray*}
\varphi'(|a|) \int_{\left\{z\in B\setminus \mathcal C;\; \ |\Delta(z)|\geq \delta\right\}}|\Delta(z)|\mu(dz)& \leq& 
L C_0|a|^\gamma \int_{\{z\in B\setminus \mathcal C;\; \R^d\setminus B_\delta |a|^{-\gamma}\}}|z|\mu(dz)  \\
&\leq  &L C^2_\mu   |a|^\gamma |a|^{(1+\alpha-\gamma)(1-\beta)}
\end{eqnarray*}
and 
\begin{eqnarray*}
\int_{\left\{z\in B\setminus \mathcal C;\; \ |\Delta(z)|\leq \delta\right\}}\sup_{|s|\leq1}
     \frac{\varphi'(|a+s\Delta(z)|)}{|a+s\Delta(z)|}|\Delta(z)|^2\mu(dz) 
&\leq &\frac{L}{|a|(1-\delta_0)} \int_{\left\{z\in B\setminus \mathcal C;\; \ |\Delta(z)|\leq \delta\right\}}|\Delta(z)|^2\mu(dz)\\
&\leq&
 LC^3_\mu {|a|^{2\gamma-1}}.
 \end{eqnarray*} 
\noindent  Since $\beta>2(1-\gamma)$, $\gamma\beta>\alpha(d+1)$ and  $2\gamma-2+\beta > \alpha(d+2-\beta)$
 the difference of the nonlocal terms is negative, being bounded from above by
\begin{eqnarray*}
&& \mathcal  J[\bar x,p,u] -  \mathcal J[\bar y,p,v] \\
&&  \hspace{1cm}  \leq - L|a|^{1-\beta}\left\{\varTheta(\varrho,\alpha,\mu)|a|^{\alpha(d+2-\beta)} 
 -C_\mu^2|a|^{\gamma+(\alpha-\gamma)(1-\beta)}-C^3_\mu |a|^{2\gamma-2+\beta}\right\}+O(\tilde C_\mu)\\
&&  \hspace{1cm}  = -L|a|^{(1-\beta)+\alpha(d+2-\beta)}\left\{\varTheta(\varrho,\alpha,\mu) - o_{|a|}(1)\right\}+O(\tilde C_\mu).
\end{eqnarray*}
\end{proof}

\begin{proof}[Proof of Corollary \ref{cor_HdrEst_LI}]
Similarly to general nonlocal operators, we use $(J2)$ to get 
\begin{eqnarray*}
&& \int_{\mathcal C}\sup_{|s|\leq1}\left((1-\tilde\eta^2)\frac{\varphi'(|a+sj(x,z)|)}{|a+sj(x,z)|}+
 \tilde\eta^2\varphi''(|a+sj(x,z)|)\right)|j(\bar x, z)|^2 \mu(dz)\\
&& \hspace{5cm}\leq -L\alpha(1-\alpha)2^{\alpha-3}|a|^{\alpha-2}\int_{\mathcal C } |z|^2 \mu(dz) \\
&& \hspace{5cm}\leq - L{\alpha C(\mu)} |a|^{\alpha-\beta}.
\end{eqnarray*}
In addition, {from $(J3)-(J4)$} we have the estimates
\begin{eqnarray*}
\varphi'(|a|) \int_{\left\{z\in B\setminus \mathcal C;\; \ |\Delta(z)|\geq \delta\right\}}|\Delta(z)|\mu(dz)&\leq&
L\alpha |a|^{\alpha-1}C_0|a|^\gamma \int_{B\setminus \mathcal C;\ \R^d\setminus B_\delta |a|^{-\gamma}}|z|\mu(dz)\\
&\leq& L \alpha C^2_\mu  |a|^{\alpha-\beta +\gamma\beta}
\end{eqnarray*}
{if  $\beta\neq 1$, respectively
\begin{eqnarray*}
&&\varphi'(|a|) \int_{\left\{z\in B\setminus \mathcal C;\; \ |\Delta(z)|\geq \delta\right\}}|\Delta(z)|\mu(dz)
\leq L \alpha C^2_\mu  |a|^{\alpha-\beta} |a|^\gamma\ln(|a|\delta_0)
\end{eqnarray*}
for $\beta = 1$.} Finally, {using again $(J3)-(J4)$ we get} 
\begin{eqnarray*}
&&\int_{\left\{z\in B\setminus \mathcal C;\; \ |\Delta(z)|\leq \delta\right\}}\sup_{|s|\leq1}
     \frac{\varphi'(|a+s\Delta(z)|)}{|a+s\Delta(z)|}|\Delta(z)|^2\mu(dz) 
\\ && \hspace{5cm}
\leq L\alpha (|a|(1-\delta_0))^{\alpha-2 }
\int_{\left\{z\in B\setminus \mathcal C;\; \ |\Delta(z)|\leq \delta\right\}}|\Delta(z)|^2\mu(dz) \\
\\ && \hspace{5cm}
\leq L\alpha C^3_\mu |a|^{2\gamma-2+\beta}|a|^{\alpha-\beta}
\end{eqnarray*} 
For $\alpha$ sufficiently small we thus have
\begin{eqnarray*}
&& \mathcal  J[\bar x,p,u] -  \mathcal J[\bar y,p,v] \leq   - L|a|^{\alpha-\beta} 
\left({\alpha C(\mu)} - o_{|a|}(1) \right) + {O(\tilde C_\mu)}.
\end{eqnarray*}
\end{proof}

\begin{prop}[\bf Quadratic estimates - L\'evy-It\^o operators]
Let $(J1)$, {$(J4)$ and $(J5)$} hold.
Let $u,v$ be two bounded functions and 
assume the auxiliary function
$$
\psi_\varepsilon(x,y) = u(x) - v(y) - \frac{|x-y|^2}{\varepsilon^2}
$$
attains a {positive maximum at  $(\bar x,\bar y)$, with $\bar x\neq \bar y$. 
Denote by  $a= \bar x - \bar y$ and by $p = 2\frac{\bar x - \bar y}{\varepsilon^2}.$}
Then the following holds
\begin{eqnarray*}
\mathcal  J[\bar x,p,u]  -  \mathcal J[\bar y,p,u] \leq
2C_0^2\frac{1}{\varepsilon^2} \int_{B_\delta} |z|^2 \mu(dz) +
C_0^2\frac{|a|^{2\gamma}}{\varepsilon^2}\tilde C_\mu +
2C_0\frac{|a|^{\gamma+1}}{\varepsilon^2}\tilde C_\mu.  
\end{eqnarray*} 
\label{prop_QdrEstLI}
\end{prop}

\begin{proof}[Proof of Proposition \ref{prop_QdrEstLI}]
By definition of $(\bar x,\bar y)$, we have 
\begin{equation}
\label{eq_max_qdr}
u(\bar x + j(\bar x, z)) - v(\bar y + j(\bar y, z')) - \frac{|\bar x + j(\bar x, z) - \bar y -j(\bar y, z')|^2}{\varepsilon^2} \leq 
u(\bar x) - v(\bar y) - \frac{|\bar x - \bar y|^2}{\varepsilon^2}.
\end{equation}
We split the difference of the integral terms into
$$
\mathcal J[\bar x, p, u] - \mathcal J[\bar y, p, u] = 
\mathcal T_q^1(\bar x,\bar y) + \mathcal T_q^2(\bar x,\bar y) + \mathcal T_q^3(\bar x,\bar y) 
$$
where this time the integrals are taken over the ball $B_\delta$, the ring $B\setminus B_\delta$ and the exterior of the unit ball $\R^d\setminus B$:
\begin{eqnarray*}
\mathcal T_q^1(\bar x,\bar y)  & = & 
  \int_{B_\delta} \left(u(\bar x + j(\bar x, z)) - u(\bar x) - p\cdot j(\bar x, z)\right)\mu(dz)\\
&&\hspace{1cm}
  - \int_{B_\delta}\left(v(\bar y + j(\bar y, z)) - v(\bar y) - p\cdot j(\bar y, z)\right)\mu(dz)\\
\mathcal T_q^2(\bar x,\bar y)  & = & 
  \int_{B\setminus B_\delta} \left(u(\bar x + j(\bar x, z)) - u(\bar x) - p\cdot j(\bar x, z)\right)\mu(dz)\\
&&\hspace{1cm}
  - \int_{B\setminus B_\delta}\left(v(\bar y + j(\bar y, z)) - v(\bar y) - p\cdot j(\bar y, z)\right)\mu(dz)\\
\mathcal T_q^3(\bar x,\bar y)  & = & 
  \int_{\R^d\setminus B} \left(u(\bar x + j(\bar x, z)) - u(\bar x)\right)\mu(dz)\\
&&\hspace{1cm}
  -\int_{\R^d\setminus B} \left(v(\bar y + j(\bar y, z)) - v(\bar y)\right)\mu(dz).
\end{eqnarray*}


\begin{lemma}\label{lemma_est_Tq1}
The following estimate holds
\begin{eqnarray*}
 \mathcal T_q^1(\bar x,\bar y) & \leq  & 2 C_0^2\frac{1}{\varepsilon^2}\int_{B_\delta} |z|^2 \mu(dz).
\end{eqnarray*}
\end{lemma}

\begin{proof}[Proof of Lemma \ref{lemma_est_Tq1}]
Taking  $z'=0$ and $z=0$ in inequality (\ref{eq_max_qdr}), we have  respectively $j(\bar y, z')=0$, $j(\bar x, z)=0$. Hence, {by direct computations and $(J4)$} we have
\begin{eqnarray*}
u(\bar x + j(\bar x, z)) - u(\bar x) -p \cdot j(\bar x, z) 
&\leq &  \frac{|\bar x + j(\bar x, z) - \bar y|^2}{\varepsilon^2} - \frac{|\bar x - \bar y|^2}{\varepsilon^2} - p\cdot j(\bar x,z) \\
& = &\frac{|j(\bar x, z)|^2}{\varepsilon^2} {\leq C_0^2\frac{|z|^2}{\varepsilon^2}}
\end{eqnarray*}
and
\begin{eqnarray*}
-\left(v(\bar y + j(\bar y, z')) - v(\bar y) - p\cdot j(\bar y, z')\right ) 
&\leq &\frac{|\bar x  - \bar y - j(\bar y, z')|^2}{\varepsilon^2} - \frac{|\bar x - \bar y|^2}{\varepsilon^2} + p\cdot j(\bar y, z') \\
& = &\frac{|j(\bar y, z)|^2}{\varepsilon^2}{\leq C_0^2\frac{|z|^2}{\varepsilon^2}}.
\end{eqnarray*}
Integrating on $B_\delta$ we get the desired estimate.
\end{proof}


{\begin{lemma}\label{lemma_est_Tq2}
The following estimate holds
\begin{eqnarray*} 
 \mathcal T_q^2(\bar x,\bar y) & \leq  & C_0^2 \frac{|a|^{2\gamma}}{\varepsilon^2}\int_{B\setminus B_\delta} |z|^2\mu(dz).
\end{eqnarray*}
\end{lemma}}


\begin{proof}[Proof of Lemma \ref{lemma_est_Tq2}]
Taking $z=z'$ in inequality (\ref{eq_max_qdr}), subtracting the corresponding gradients {and using $(J4)$} we obtain the inequality
\begin{eqnarray*}
\left(u(\bar x + j(\bar x, z)) - u(\bar x) - p\cdot j(\bar x, z) \right) &- &
\left(v(\bar y + j(\bar y, z)) - v(\bar y) - p\cdot j(\bar y, z) \right) \\
&\leq &\frac{|\bar x + j(\bar x, z) - \bar y -j(\bar y, z)|^2}{\varepsilon^2} - 
\frac{|\bar x - \bar y|^2}{\varepsilon^2} - p\cdot ( j(\bar x, z) -  j(\bar y, z))\\
& = &{ \frac{|j(\bar x, z) -  j(\bar y, z)|^2}{\varepsilon^2} \leq C_0^2 \frac{|z|^2|\bar x - \bar y|^{2\gamma}}{\varepsilon^2} }
\end{eqnarray*}
Integrating on the ring $B\setminus B_\delta$, we get the desired estimate.
\end{proof}


{\begin{lemma}\label{lemma_est_Tq3}
The following estimate holds
\begin{eqnarray*}
 \mathcal T_q^3(\bar x,\bar y) & \leq  & 
 C_0^2\frac{|a|^{2\gamma}}{\varepsilon^2}\int_{\R^d\setminus B}\mu(dz)  +
 2C_0\frac{|a|^{\gamma+1}}{\varepsilon^2}\int_{\R^d\setminus B}\mu(dz).  
\end{eqnarray*}
\end{lemma}}

\begin{proof}[Proof of Lemma \ref{lemma_est_Tq3}]
Once again, for $z=z'$ in inequality (\ref{eq_max_qdr}) we obtain the inequality
\begin{eqnarray*}
\left(u(\bar x + j(\bar x, z)) - u(\bar x)\right) &-&
\left(v(\bar y + j(\bar y, z)) - v(\bar y)\right) \\
&\leq & \frac{|\bar x + j(\bar x, z) - \bar y -j(\bar y, z)|^2}{\varepsilon^2} - 
\frac{|\bar x - \bar y|^2}{\varepsilon^2}.
\end{eqnarray*}
Integrating on $\R^d\setminus B$ and computing the right hand side
we get 
\begin{eqnarray*}
 \mathcal T_q^3(\bar x,\bar y) 
 & \leq &
 \int_{\R^d\setminus B}\left(|p| |j(\bar x, z) -  j(\bar y, z)|  +
 \frac{|j(\bar x, z) -  j(\bar y, z)|^2}{\varepsilon^2}\right)\mu(dz).
\end{eqnarray*}
{Taking into account $(J5)$ we get the desired estimate.}
\end{proof}

\noindent {From the three above lemmas and $(J1)$ we conlcude.}

\end{proof}

\section{Appendix}

\begin{lemma}\label{lemma_block_ineq}
Let $X$, $Y$ and $Z$ be block matrices of the form
$$
A = \begin{bmatrix} 
A_1 &  0 \\
0 & A_2\\
\end{bmatrix} 
$$
such that they satisfy the inequality
\begin{equation}\label{eq:matrix_ineq_blocks}
\begin{bmatrix} 
X &  0 \\
0 & -Y \\
\end{bmatrix} \leq  
\begin{bmatrix} 
Z & -Z \\
-Z & Z \\
\end{bmatrix} 
\end{equation}
Then the block matrices $X_i$, $Y_i$ satisfy inequality (\ref{eq:matrix_ineq_blocks}) where $Z$ is replaced with $Z_i$, for $i=1,2$.
\end{lemma}

\begin{proof}
The previous matrix inequality can be rewritten in the form
$$
Xz\cdot z - Yz'\cdot z' \leq Z(z-z')\cdot(z-z').
$$
Due to the form of the block matrices, namely the secondary diagonal null, we can write the inequality on components, for $z=(z_1,z_2)$, $z'=(z_1',z_2')$ 
$$
\sum_{i=1,2} \big(X_iz_i\cdot z_i- Y_iz_i'\cdot z_i'\big) \leq 
\sum_{i=1,2} \big(Z_i(z_i-z_i')\cdot(z_i-z_i')\big).
$$
Thus, taking $z=(z_1,0)$ and $z'=(z_1',0)$, respectively $z=(0,z_2)$ and $z'=(0,z_2')$ we get the corresponding inequality for the block matrices $X_i,Y_i,Z_i$.

\end{proof}
In the next lemma, for a symmetric matrix $A$, $\|A\|$ denotes
$\max_{|\xi|\le 1} |A\xi\cdot \xi|$.
\begin{lemma}\label{lemma_bloc_ineq_conv} 
Let $X$, $Y$ and $Z$ be symmetric matrices satisfying inequality
(\ref{eq:matrix_ineq_blocks}).  Consider the sup-convolution
$X^\varepsilon$ of $X$ and the inf-convolution $Y^\varepsilon$ of $Y$, defined by
$$
X^\varepsilon z \cdot z = \sup_{\xi\in\R^d} \left \{ X\xi\cdot\xi - \frac{|z-\xi|^2}{\varepsilon} \right \}
\hbox{ and }
Y_\varepsilon z \cdot z = \inf_{\xi\in\R^d} \left \{ Y\xi\cdot\xi + \frac{|z-\xi|^2}{\varepsilon} \right \}.
$$
Then there exists $\varepsilon_0= (\max(\|X\|,\|Y\|,2\|Z\|))^{-1}>0$
such that for any $ \varepsilon \in(0,\varepsilon_0)$,
$X^\varepsilon$, $Y_\varepsilon$ and $Z^{2\varepsilon}$ satisfy as
well inequality (\ref{eq:matrix_ineq_blocks}).  In addition we have
\begin{equation}\label{eq_mon_conv}
-\frac{1}{\varepsilon} I, X \leq X^\varepsilon \
\hbox{ and }
Y_\varepsilon\leq Y, \frac{1}{\varepsilon} I.
\end{equation}
\end{lemma}

\begin{proof}
Consider $\eps$ as in the statement of the lemma. Then the
$\eps$-sup-convolutions of the two quadratic forms associated with the
matrix inequality~\eqref{eq:matrix_ineq_blocks} are finite. It must be
checked that it gives the above mentioned inequality. As far as the
left-hand side is concerned, writing matrix inequalities in terms of
quadratic forms, we have for all $\zeta, \alpha \in \R^d$,
$$
\sup_{\xi,\eta}  \left\{ X (\xi-\zeta)\cdot (\xi-\zeta) - Y (\eta-\alpha) \cdot
(\eta- \alpha)- \frac1\eps |\xi|^2 - \frac1\eps |\eta|^2 \right\} =
X^\eps \zeta \cdot \zeta - Y_\eps \alpha \cdot \alpha. 
$$
As far as the right-hand side is concerned, we get
\begin{multline*} \sup_{\xi,\eta} \left\{ Z(\xi-\eta)\cdot (\xi -\eta) - \frac1\eps
|\zeta-\xi|^2 - \frac1\eps |\alpha-\eta|^2 \right\} \\=
\sup_{\tilde\xi} \left\{ Z \tilde\xi\cdot\tilde\xi - \inf_{\tilde\eta}
\left\{ \frac1\eps |\zeta-\tilde\xi-\tilde\eta-\alpha|^2 +
\frac1\eps|\tilde\eta|^2 \right\}\right\} = \sup_{\tilde\xi} \left\{ Z
\tilde\xi\cdot\tilde\xi - \frac1{2\eps} |\zeta-\alpha -\tilde\xi|^2
\right\} \\
= Z^{2\eps} (\zeta-\alpha)\cdot (\zeta-\alpha)
\end{multline*}
where we changed $\xi$ in $\tilde\xi = \xi-\eta$ and $\eta$ in
$\tilde\eta= \eta -\alpha$. The additional matrix inequalities come
directly from the definition of the inf/sup-convolution. The proof of
the lemma is now complete.
\end{proof}

\begin{lemma} \label{lemma_matrix_conv}
Let $Z = \frac1\alpha (I - \omega \hat a \otimes \hat a)$, where $\hat a\in\SS^{d-1}$, $\alpha>0$ and $\omega\geq0$. Then the following holds
\begin{equation}\label{eq_matrix_conv}
Z^{\frac{\alpha}{2}} = \frac{2}{\alpha} \left(I - \frac{2\omega}{1+\omega} \hat a\otimes\hat a \right).
\end{equation}
\end{lemma}

\begin{proof}
By definition
$$
Z^{\frac{\alpha}{2}} z\cdot z= \sup_{\xi} \left\{ Z\xi\cdot \xi - 2\frac{|z-\xi|^2}{\alpha}\right\}
$$
and the supremum is attained at points $\bar\xi$ satisfying
$Z\bar \xi = \frac2\alpha (\bar\xi-z)$, or equivalently
$$
( I - \omega \hat a\otimes \hat a) \bar\xi = 2(\bar\xi - z).
$$
Taking the inner product with $\hat a$ in this identity, we have
$$\bar \xi\cdot\hat a = \frac{2}{1+\omega} z\cdot \hat a.$$
Taking now the inner product with $z$ in the same identity, we have
$$
\bar\xi \cdot z =
2|z|^2 - \omega(z\cdot\hat a) (\bar\xi\cdot\hat a) = 
2|z|^2 - \frac{2\omega}{1+\omega}(z\cdot\hat a)^2. 
$$
Therefore
\begin{eqnarray*}
Z^{\frac{\alpha}{2}} z\cdot z & = &
\frac{2}{\alpha} \left((\bar \xi - z)\cdot \bar\xi -|z-\bar\xi|^2\right) \\
&=&\frac{2}{\alpha} \left((\bar \xi - z)\cdot z\right) =
\frac{2}{\alpha} \left(|z|^2 - \frac{2\omega}{1+\omega}(z\cdot\hat a)^2 \right).
\end{eqnarray*}
\end{proof}

\begin{lemma}\label{lemma_TrEst}
Let $X,Y,Z^{\frac\alpha2}$ satisfy the block inequality  (\ref{eq:matrix_ineq_blocks}), with $Z^{\frac\alpha2}$ given by equation (\ref{eq_matrix_conv}), for some $\omega\geq 1$.
Then the following holds:
$$
\hbox{trace}(X-Y) \leq  - \frac{8(\omega-1)}{\alpha(1+\omega)}.
$$
\end{lemma}

\begin{proof}
Rewrite the matrix inequality in the form
$$
X z\cdot z - Y z'\cdot z' \leq Z^{\frac\alpha2}(z-z')\cdot(z-z').
$$
Taking $z= -z' = \hat a$ we have
$$
X\hat a \cdot \hat a - Y\hat a \cdot \hat a \leq 4 Z^{\frac\alpha2}\hat a \cdot \hat a
$$
whereas for any vector $z$ orthogonal to $\hat a$ 
$$
Xz \cdot z - Yz \cdot z \leq 0.
$$
Therefore
$$
\hbox{trace} (X-Y) \leq \frac8\alpha \left( |\hat a|^2 - \frac{2\omega}{1+\omega} |\hat a|^{2} \right) = - \frac{8(\omega-1)}{\alpha(\omega+1)}.
$$
\end{proof}

\bibliographystyle{siam}

\end{document}